%% file: ms.tex
\documentclass[final]{siamart1116}

\usepackage{amsfonts}
\usepackage{graphicx}
\usepackage{epstopdf}
\ifpdf
  \DeclareGraphicsExtensions{.eps,.pdf,.png,.jpg}
\else
  \DeclareGraphicsExtensions{.eps}
\fi
\usepackage{mathtools}
\usepackage{bm}
\usepackage[caption=false,subrefformat=parens,labelformat=parens]{subfig} 
\usepackage{tikz}
\usepackage{pgfplots}
\usepackage{alphalph}
\usepackage{color}
\usepackage[compact]{titlesec}
\usepackage{multirow}
\usepackage{cite}
\pgfplotsset{compat = 1.3}
\usepackage{enumerate}

\usepackage{algpseudocode} 
\Crefname{ALC@unique}{Line}{Lines} 
\newcounter{algsubstate}

\usetikzlibrary{matrix,positioning}
\usetikzlibrary{fit}
\usetikzlibrary{patterns}
\makeatletter
\newcommand{\specificthanks}[1]{\@fnsymbol{#1}}
\makeatother
\newcommand*\circled[1]{\tikz[baseline=(char.base)]{
            \node[shape=circle,draw,inner sep=0.8pt] (char) {#1};}}
\newcommand*\circledletter[1]{\tikz[baseline=(char.base)]{
        \node[shape=circle,draw,minimum size=2mm, inner sep=0pt] (char)
        {\rule[-1.5pt]{0pt}{\dimexpr2ex+1pt}#1};}}
\newcommand*\blcircled[1]{\tikz[baseline=(char.base)]{
            \node[shape=circle,fill=black,draw,inner sep=0.8pt] (char) {#1};}}
\newcommand*\blcircledletter[1]{\tikz[baseline=(char.base)]{
        \node[shape=circle,fill=black,draw,minimum size=2mm, inner sep=0pt] (char)
        {\rule[-1.5pt]{0pt}{\dimexpr2ex+1pt}#1};}}

\newcommand\mytodo[1]{\textcolor{red}{\textbf{TODO: }#1}\\}
\renewcommand\mytodo[1]{} 

\newcommand{\LinOp}[2]{(#2 A{-}#1 B)}
\newcommand{\LinOpInv}[2]{(#2 A{-}#1 B)^{-1}}

\newcommand{\Rho}{\mathrm{P}}
           
\newcommand{\CC}{\mathbb{C}}

\numberwithin{theorem}{section}
\newsiamremark{remark}{Remark} 
\newsiamremark{example}{Example} 

\newcommand{\TheLongTitle}{A rational QZ method} 
\newcommand{\TheTitle}{RQZ} 
\newcommand{\TheAuthors}{D. Camps, K. Meerbergen, and R. Vandebril}

\headers{\TheTitle}{\TheAuthors}

\title{{\TheLongTitle}\thanks{Submitted to the editors February 12, 2018.
\funding{
The research was partially supported by 
the Research Council KU Leuven, projects 
C14/16/056 (Inverse-free Rational Krylov Methods: Theory and Applications), 
OT/14/074 (Numerical algorithms for large scale matrices with uncertain coefficients)
}}}

\author{
  Daan Camps\thanks{Department of Computer Science, KU Leuven, University of Leuven, 3001 Leuven, Belgium. (\email{daan.camps@kuleuven.be}, \email{karl.meerbergen@kuleuven.be}, \email{raf.vandebril@kuleuven.be})} 
  \and
  Karl Meerbergen\footnotemark[2]
  \and
  Raf Vandebril\footnotemark[2]
}

\usepackage{amsopn}

\ifpdf
\hypersetup{
  pdftitle={\TheTitle},
  pdfauthor={\TheAuthors}
}
\fi

\usepackage{amsmath}
\usepackage{breqn}
\usepackage{bm}

\makeatletter
 \def\@testdef #1#2#3{%
   \def\reserved@a{#3}\expandafter \ifx \csname #1@#2\endcsname
  \reserved@a  \else
 \typeout{^^Jlabel #2 changed:^^J%
 \meaning\reserved@a^^J%
 \expandafter\meaning\csname #1@#2\endcsname^^J}%
 \@tempswatrue \fi}
\makeatother

\begin{document}

\maketitle

\begin{abstract}
\input{sections/00_abstract/abstract.tex}
\end{abstract}



\input{sections/01_introduction/introduction.tex}

\input{sections/02_hessenbergpairs/hessenbergpairs.tex}

\input{sections/03_reduction/reduction.tex}

\input{sections/04_rqz1/rqz1.tex}
\input{sections/07_tightly_packed_shifts/tightly_packed_shifts_short.tex}


\input{sections/05_implicitq/implicitq.tex}


\input{sections/06_subspace_iteration/subspace_iteration.tex}

\input{sections/08_filter_rational_krylov/filter_rational_krylov.tex}

\input{sections/09_conclusion/conclusion.tex}

\bibliographystyle{siamplain}
\bibliography{bibliography_rqz}

\end{document}

%% file: sections/00_abstract/abstract.tex
We propose a \emph{rational QZ} method for the solution of the dense, unsymmetric
generalized eigenvalue problem. This generalization of the classical
QZ method operates implicitly on a Hessenberg, Hessenberg pencil instead of on a Hessenberg,
triangular pencil. 
Whereas the QZ method performs nested subspace iteration driven by a
polynomial, the rational QZ method allows for nested subspace iteration driven by a
rational function, this creates the additional freedom of selecting poles. 
In this article we study Hessenberg, Hessenberg pencils, link them to rational Krylov
subspaces, 
propose a direct reduction method to such a pencil, and 
introduce the implicit rational QZ step.
The link  with rational Krylov subspaces allows us to prove essential
uniqueness (implicit Q theorem) of the rational QZ iterates as well as convergence of the proposed
method. In the proofs, we operate directly on the pencil
 instead of rephrasing it all in terms of a single matrix.
Numerical experiments are included to illustrate competitiveness in terms of speed and
accuracy with the classical approach. Two other types of experiments exemplify 
new possibilities. First we illustrate that good pole selection can be used
to deflate the original problem during the reduction phase, and second we use  
the rational QZ method to
 implicitly filter a rational Krylov subspace in an iterative method.


%% file: sections/01_introduction/introduction.tex
\section{Introduction.}
\label{sec:introduction}

The numerical computation of the eigenvalues of a regular\footnote{Regular means
  that the characteristic polynomial differs from zero.}
matrix pair $A,B \in \CC^{n \times n}$ is the principal problem studied in this paper. The
set of eigenvalues of $(A,B)$ is denoted as $\Lambda$ and defined by
\begin{equation}
\label{def:GEP}
\Lambda = \lbrace 
\lambda=\alpha / \beta \in \bar{\CC}: \text{det}(\beta A - \alpha B)=0 \rbrace,
\end{equation}
with $\bar{\CC}=\CC\cup\lbrace \infty\rbrace$.
If $\beta \neq 0$, the eigenvalue is equal to $\lambda = \alpha / \beta$, while for $\beta
= 0$ the eigenvalue is located at $\infty$. When there are no infinite eigenvalues
$B$ is invertible and the
eigenvalues of the pencil are equal to those of $B^{-1}A$ or $AB^{-1}$ see, e.g., the
monographs \cite{b333,b037}.

The QZ method, originally introduced by Moler \& Stewart \cite{Moler1973}, is presumably
the method of choice for the solution of this problem for general small to medium-size
matrix pairs. The original pencil $(A,B)$ is transformed via unitary equivalences to generalized
Schur form $(S,T)$, where both $S$ and $T$ are upper triangular. The eigenvalues of $(A,B)$ are
readily found as the ratios of the diagonal elements of the pencil $(S,T)$. The method consists conceptually out of $2$ phases,
just as the QR algorithm:
\begin{enumerate}
\item A direct reduction of the matrix pair $(A,B)$ to an equivalent Hessenberg, triangular matrix pair $(H,R)$.
\item An iterative phase during which deflating subspaces of the matrix pair $(H,R)$ are
  determined and the matrix pair is essentially reduced to the triangular, triangular pair
  $(S,T)$.
\end{enumerate}

Various modifications and
additions to the original algorithm have been proposed after its original introduction.
Kaufman \cite{Kaufman1977} added a deflation strategy and Ward \cite{Ward1975} further
refined various aspects of the method. Watkins \& Elsner \cite{Watkins1994} generalized the QZ algorithm to a class of
GZ iterations, and more recently, K\aa gstr\"om \& Kressner
\cite{Kagstrom2007} incorporated an aggressive early deflation strategy into a multishift
QZ iteration. For more information we refer the reader to the monographs of Watkins
\cite{b333} and
Kressner \cite{b339}.

Vandebril \& Watkins \cite{Vandebril2013} proposed a generalization beyond the Hessenberg,
upper triangular pair. Their QZ like method
reduces the matrix pair $(A,B)$ to \emph{condensed} form and iterates
directly on the condensed matrix pair. A matrix pair $(A,B)$ is said to be a condensed pair if both
matrices are Hessenberg matrices and if there is exactly one nonzero element for every subdiagonal
position which can be either at $A$ or $B$.  The classical Hessenberg, triangular format used in the
QZ method is a special instance of a condensed matrix pair which maintains all zero
subdiagonal elements at the side of $B$.

In this paper we propose a further generalization of the QZ method beyond condensed pairs.
We will call this method the \emph{rational QZ} (RQZ) method as it links strongly to
rational Krylov subspaces \cite{Berljafa2015}.
%
As we will demonstrate in detail in the remainder of
the paper, Hessenberg pairs and the associated rational Krylov subspaces are determined by
\emph{poles} that can be exploited to improve the convergence of the method. Both the
original QZ algorithm \cite{Moler1973} and the condensed QZ algorithm \cite{Vandebril2013}
turn out to be special instances of the RQZ method determined by a specific choice of poles.
An implementation of the RQZ method is publicly available on: \url{numa.cs.kuleuven.be/software/rqz}
.

This article is closely related to the article by Berljafa \&
G\"{u}ttel \cite{Berljafa2015}.
Starting from a rational Krylov subspace and the linked Hessenberg
pair, 
their article discusses a way to change the poles by operating solely on the Hessenberg
pair. We will see in this article that their way of introducing poles and
moving them is  related to introducing a shift and chasing it, like in typical QR
algorithms. We will extend these results and formulate
an implicit QZ
algorithm 
that executes nested subspace
iteration driven by a rational function.  Moreover, in the theoretical analysis we 
directly rely on the pair $(A,B)$ instead of rephrasing the relations in terms of a single
matrix $AB^{-1}$ or $B^{-1}A$ as is usually done, assuming thereby nonsingularity of $B$. 

This paper is organized as follows.
The notion of a Hessenberg pair is formally defined in \Cref{sec:Hessenbergpairs}, its
properties are studied subsequently, and two types of operations on Hessenberg
pairs are discussed: the introduction of a new pole and the \emph{swapping} of poles. 
\Cref{sec:dirred} proposes a method to reduce a general matrix pair to a Hessenberg pencil
by means of unitary equivalence transformations. This is the RQZ analogue of the initial
reduction phase in the QZ algorithm.
The generalization of the iterative phase is presented in \Cref{sec:rqz}.
It is illustrated how an RQZ step with a single shift can be performed implicitly and
numerical experiments illustrate the speed and accuracy.
An implicit Q theorem for Hessenberg pairs is stated and used to prove that the RQZ
iteration implicitly performs nested subspace iteration driven by a set of rational
functions in \Cref{sec:implQthm,ssec:subspaceiter}.
In \Cref{sec:RK} we apply the RQZ method to filter a \emph{rational Krylov} subspace in an
iterative method. We conclude in \Cref{sec:conclusion}.

In this article we adopt the following notational conventions.
Scalars $\alpha, \beta, \hdots$ are denoted with Greek letters, 
matrices $A, B, \hdots$ with capital Latin letters.
Vectors $\bm{a}, \bm{b}, \hdots$ are denoted in lowercase boldface Latin letters.
The entry on row $i$ and column $j$ of $A$ is denoted as $a_{ij}$,
and column $i$ of $A$ as $\bm{a}_i$.
\textsc{Matlab}'s colon notation is sometimes used to denote part of a matrix: 
$A_{i:j,:}$ stands for rows $i$ to $j$ of $A$.
$I$ is the identity matrix and $\bm{e}_i$ is its $i$th column. 
$A^*$ is the Hermitian conjugate of $A$, $\mathcal{R}(A)$ is the column space of $A$.
$\mathcal{E}_k = \mathcal{R}(\bm{e}_1, \hdots, \bm{e}_k)$ is the subspace spanned by 
the $k$ first canonical basis vectors.
$\mathcal{K}_{k}(A,\bm{v}) = \mathcal{R}(\bm{v}, A\bm{v}, \hdots, A^{k-1}\bm{v})$ is the
Krylov subspace of order $k$ generated by $A$ from $\bm{v}$.
The complex plane extended with the point at infinity, $\CC \cup \lbrace \infty \rbrace$, is denoted as
$\bar{\CC}$.
For all nonzero scalars $\alpha \in {\CC}$, we define $\alpha /0 = \infty$ and $\alpha / \infty = 0$.

%% file: sections/02_hessenbergpairs/hessenbergpairs.tex
\section{Hessenberg pairs and their poles.}
\label{sec:Hessenbergpairs}
In this section we repeat necessities from the literature and 
introduce some basic concepts linked to Hessenberg pairs. These pairs
 appear naturally in the context of the \emph{rational Krylov} (RK) method
introduced and studied  by Ruhe \cite{Ruhe1984,Ruhe1994a, Ruhe1994b,
  Ruhe1998a}.  We will elaborate on this connection in \Cref{sec:RK}.

\subsection{Proper Hessenberg pairs.}
\label{subsec:RH}
A matrix $H$ is of Hessenberg form if all its elements below the first subdiagonal are
zero.  A \emph{proper} or \emph{irreducible Hessenberg} matrix has all its subdiagonal
elements different from zero. Being proper ensures that there are no
obvious \emph{deflations} allowing us to split the Hessenberg matrix into 
block upper triangular form with
smaller
submatrices. 
 For a pair of Hessenberg matrices there is a subtlety, as there are two less obvious
possibilities for deflation. 
\begin{definition}[Proper Hessenberg pair]
\label{def:ratHess}
A pair of Hessenberg matrices $A,B \in \CC^{n \times n}$ is said to be proper (or
irreducible) if the following two conditions are met:
\begin{enumerate}[I.]
\item There is no $i$ in $1, \ldots, n-1$ so that $a_{i+1,i}$ and $b_{i+1,i}$ are
  simultaneously zero;
\item The first columns of $A$ and $B$ are linearly independent, as are the last rows of $A$ and $B$.
\end{enumerate}
  For a proper Hessenberg pair we define its ordered pole tuple as $\Xi = ( \xi_1, \hdots,
  \xi_{n-1} )$, $\xi_i \in \bar{\CC}$, where $a_{i+1,i}/b_{i+1,i} = \xi_i$ for all $i$ from
  $1$ to $n{-}1$.
\end{definition}



%

The ratios of the subdiagonal elements of $A$ over the subdiagonal elements of $B$ are
thus called the poles of the proper Hessenberg pair. Since we set division by zero equal
to $\infty$ in
$\bar{\CC}$, a pole is located at $\infty$ if the respective subdiagonal element of
$B$ is zero.

The first condition of being proper means that all poles are well-defined over
$\bar{\CC}$, so there is no $0/0$.  Just like in the classical case
 $a_{i+1,i} = b_{i+1,i} = 0$ allows us to deflate the problem into two
independent subproblems.

The second condition is less obvious,
but it is simple to deflate an eigenvalue if it is not met.
Construct a rotation $Q_1$, acting on the first two rows
such that $Q_1^*$ maps the first column of $A$ and $B$ in the direction of $\bm{e}_1$, then the pair
$Q_1^*(A,B)$ allows for a deflation. Similarly we can construct a rotation $Z_{n-1}$ to
transform $(A,B)Z_{n-1}$ to a deflatable format in case the last rows are linearly dependent.
Thus if condition~II does not hold then the pair can be transformed into an equivalent
pair for which condition~I does not hold in the first or last subdiagonal position.

We remark that even if condition~II of the definition of a proper Hessenberg pair were not
met, we still define the first pole $\xi_1$ and last pole $\xi_{n-1}$
as in \Cref{def:ratHess}. Suppose, however, that 
%
there is some scalar $\gamma$ such that $\bm{a}_1 = \gamma \bm{b}_1$, with $\bm{a}_1$ and
$\bm{b}_1$ the first columns of $A$ and $B$ respectively.
This means that $\gamma$ is both the first pole, $\xi_1 = a_{21}/b_{21} = \gamma$, and an
eigenvalue, $A\bm{e}_1 = \gamma B \bm{e}_1$. So condition~II of the definition implies that
the pole $\xi_1$ equals an eigenvalue.
Similarly the last pole $\xi_{n-1}$ is an eigenvalue if the last rows of $A$
and $B$ are linearly dependent.

Properness of the Hessenberg matrix ensures \emph{essential uniqueness} of the QR iterates,
which is crucial in the design of an \emph{implicit} $QR$ algorithm \cite{Fra61, Fra62}
for the standard eigenvalue problem.  
We will prove in \Cref{sec:implQthm} that also proper Hessenberg pairs inherit
a type of essential uniqueness allowing for the design of an implicit method, which is the
implicit Q theorem for Hessenberg pairs.

The other pencils for which QZ algorithms were designed, fit in \Cref{def:ratHess}.
Pairs in Hessenberg, triangular form \cite{Moler1973}
are proper with poles $\Xi = ( \infty,
\infty, \hdots, \infty )$;  a pair of matrices in condensed form
\cite{Vandebril2013} is also a proper Hessenberg pair with poles being either
$0$ or $\infty$.

The properties of proper Hessenberg pairs discussed in the next lemma are frequently
used throughout the paper.
\begin{lemma}
\label{lemma:properHessproperties}
Let $(A,B) \in \mathbb{C}^{n \times n}$ be a proper Hessenberg pair with poles $\Xi =
( \xi_1, \hdots, \xi_{n-1} )$.
Then the following statements hold:
\begin{enumerate}[I.]
\item For  $\mu,\nu\in\CC$, such that $ \mu / \nu \notin \Xi$, we
    have that  $\LinOp{\mu}{\nu}$ is a proper Hessenberg matrix.
\item For  $\mu,\nu\in\CC$, such that $\mu/\nu$ is equal to a certain pole $\xi_k$ ($1\leq k \leq n-1$),
we have that
$N = \LinOp{\mu}{\nu}$, 
is 
block upper triangular,
\begin{align*}
N =
\begin{bmatrix}
N_{11} & N_{12} \\
 & N_{22}
\end{bmatrix},
\end{align*}
where $N_{11}$ and $N_{22}$ are Hessenberg matrices of sizes $k{\times}k$ and ${(n-k)}{\times}{(n-k)}$
respectively. 
\item For $\mu, \nu, \alpha, \beta \in \CC$, such that $\mu \beta \neq \alpha \nu$, we have that,
\begin{equation*}
(M,N) = (\beta A{-}\alpha B, \nu A{-}\mu B),
\end{equation*}
is a proper Hessenberg pair.
\item For $k=1,\hdots, n-1$ we have that $\mathcal{R}(\bm{a}_1, \hdots \bm{a}_k) \neq
  \mathcal{R}(\bm{b}_1, \hdots, \bm{b}_k)$.
\end{enumerate}
\end{lemma}
\begin{proof}
Statements I.\ and II.\ are trivial.
The pencil of statement III.\ satisfies the definition of a proper Hessenberg pair: $M$ and $N$ are clearly upper Hessenberg matrices, their $k$th subdiagonal elements are,
\begin{equation*}
\begin{bmatrix}
m_{k+1,k} \\
n_{k+1,k}
\end{bmatrix}
=
\begin{bmatrix}
\beta & -\alpha \\
\nu & -\mu
\end{bmatrix}
\begin{bmatrix}
a_{k+1,k} \\
b_{k+1,k}
\end{bmatrix}.
\end{equation*}
The vector on the left is different from zero since the matrix is nonsingular and the vector on the right is nonzero.
The first column of $M$ is also linear independent from the first column of $N$ because the same nonsingular matrix is used in the transformation.
The same holds for the last row.
The proof of statement IV.\ is by induction and contradiction.
The case $k=1$ follows from the definition of a proper Hessenberg pair.
Suppose the statement holds up to column $k$.  We assume now, by contradiction, that it breaks down at column
$k+1$, thus $\mathcal{R}(\bm{a}_1,\hdots,\bm{a}_{k+1}) = \mathcal{R}(\bm{b}_1,\hdots,\bm{b}_{k+1})$.
The equality implies the existence of a nonsingular $(k+1){\times}(k+1)$ matrix $C$ such that,
\begin{equation}
\label{eq:colspanpropHess}
\bigl[\bm{a}_1,\hdots,\bm{a}_{k+1} \bigr] = \bigl[\bm{b}_1,\hdots,\bm{b}_{k+1} \bigr] 
\begin{bmatrix}
c_{11} & \hdots & c_{1,k+1} \\
\vdots	& \ddots & \vdots \\
c_{k+1,1} & \hdots & c_{k+1,k+1}
\end{bmatrix}.
\end{equation}
It follows from the induction hypothesis that there is a $j$ with $1 \leq j \leq k$ such that 
$\bm{a}_j \notin \mathcal{R}(\bm{b}_1,\hdots,\bm{b}_k)$. Therefore $c_{k+1,j} \neq 0$.
By the Hessenberg structure,
\begin{equation*}
0 = a_{k+2,j} = \sum_{i=1}^{k+1} b_{k+2,i} \, c_{i,j} = b_{k+2,k+1} c_{k+1,j}.
\end{equation*}
This implies that $b_{k+2,k+1}$ must be zero. 
Equation \eqref{eq:colspanpropHess} consequently implies that also $a_{k+2,k+1} = 0$. 
These two values being simultaneously zero contradicts the properness.
\end{proof}



\subsection{Manipulating the poles of a Hessenberg pair.}
\label{subsec:manipulate}
In this section we will revisit two operations for manipulating the poles of a Hessenberg
pair, namely changing the first or the last pole, and swapping poles (see also Berljafa \& G\"uttel \cite{Berljafa2015}).

\paragraph{Changing poles at the boundaries.}
Let  $A,B \in \CC^{n{\times}n}$ be a proper Hessenberg pair and assume
the first pole $\xi_1$ different from the eigenvalues of $(A,B)$.
The pole $\xi_1$ can be changed to another pole $\hat{\xi}_1 \in
\bar{\CC}$ by multiplying $(A,B)$ from the left with a unitary transformation
$Q_1^*$, where $Q_1^*\bm{x}=\alpha \bm{e}_1$ and,
\begin{eqnarray}\label{eq:intro_single_pole}
\bm{x} & = & \hat{\gamma} \, \LinOp{\hat{\alpha}_1}{\hat{\beta}_1}
\LinOpInv{\alpha_1}{\beta_1} \bm{e}_1 \\
\nonumber & = & \gamma \, \LinOp{\hat{\xi}_1}{}
\LinOpInv{\xi_1}{} \bm{e}_1,
\end{eqnarray}
with $\gamma$ and $\hat{\gamma}$ convenient scaling factors; and $\hat{\alpha}_1,
\hat{\beta}_1,\alpha_1,\beta_1 \in \CC$ are chosen to satisfy the new pole
$\hat{\xi}_1 = \hat{\alpha}_1 / \hat{\beta}_1$ and the old pole $\xi_1 = \alpha_1 / \beta_1$. 
The notation with $\alpha$ and $\beta$ to denote $\LinOp{{\alpha}}{{\beta}}$ is factually
the most correct one. For notational simplicity, however, we will often use the shorthand
notation $ \LinOp{{\xi}}{}$, where $\xi=\alpha/\beta$ instead.
As $\hat{\xi}_1 \neq \xi_1$, otherwise nothing needs to be done,
$\bm{x}$ must be a vector with only the two leading
elements nonzero and thus $Q_1$ is always well defined and can, for example, be chosen as
a rotation matrix.

If $Q_1$ is used to compute $(\hat{A},\hat{B}) = Q_1^* (A,B)$, then
$\hat{\xi}_1$ will become the first pole of
$(\hat{A},\hat{B})$ because the first subdiagonal element of $(\hat{A}-\hat{\xi}_1\hat{B})$ is zero:
\begin{eqnarray*}
   (\hat{A}-\hat{\xi}_1 \hat{B}) \bm{e}_1 & = &  Q_1^* (A-\hat{\xi}_1 B) \bm{e}_1  \\
  & = & \tilde{\gamma}\, Q_1^* (A-\hat{\xi}_1 B) (A-{\xi}_1 B)^{-1} \bm{e}_1 = 
\frac{\tilde{\gamma}}{\gamma}\, Q_1^* \bm{x} = \frac{\alpha\,\tilde{\gamma}}{\gamma} \bm{e}_1.
\end{eqnarray*}

Theoretically, under the  assumption that $B$ is
nonsingular, we could equally well define $\bm{x}=\gamma  (AB^{-1} - \hat{\xi}_1 I )
(AB^{-1} - {\xi}_1 I)^{-1} \bm{e}_1.$ 
Practically, however, to avoid the nonsingularity assumption of $B$, 
and for reasons of numerical stability,
we stick to \eqref{eq:intro_single_pole}.

\begin{remark}

As $\LinOpInv{\xi_1}{}\bm{e}_1$ is scalar multiple of
$\bm{e}_1$ there is no need to compute this in practice.
Moreover, even if $\xi_1$ is an eigenvalue, a scalar multiple
of $\bm{e}_1$ is always a solution of $\LinOp{\xi}{}\bm{y} = \bm{e}_1$.
The inverse factor is included to emphasize the rational function
used to update the pole and moreover, it is consistent with the analysis
of Vandebril \& Watkins \cite{Vandebril2012,Vandebril2013} where it does play
a role in the multishift setting.
In practice we compute $\bm{x} = \gamma \LinOp{\hat{\xi}_1}{} \bm{e}_1$ in
$O(1)$ operations.
\end{remark}

We can compute an equivalence transformation to change the last pole,
by operating from the right on the Hessenberg pair in a comparable way.
Assume $\xi_{n-1}$ different from the eigenvalues of $(A,B)$.
We can change the pole $\xi_{n-1}$  to $\hat{\xi}_{n-1} \in
\bar{\CC}$. 
If we consider the row vector
$
\bm{x}^{T} = \gamma \bm{e}_{n}^{T} \LinOpInv{\xi_{n-1}}{}\LinOp{\hat{\xi}_{n-1}}{},
$
with $\gamma$ a convenient scaling factor,
and a transformation $Z_{n-1}$ that introduces a zero in the penultimate position of $\bm{x}^T$,
$\bm{x}^T Z_{n-1} = \alpha \bm{e}_{n}^T,$
then the last pole in the Hessenberg pair $(\hat{A},\hat{B}) = (A,B) Z_{n-1}$ is changed to $\hat{\xi}_{n-1}$.

Again, the system $\bm{e}_{n}^{T} \LinOpInv{\xi_{n-1}}{}$ is never solved in practice
as the solution is a scalar multiple of $\bm{e}_{n}^{T}$ but only included form theoretical purposes.

\paragraph{Swapping poles.}
Any two consecutive poles $\xi_i$ and $\xi_{i+1}$ in a proper Hessenberg
pair $(A,B)$ can be \textit{swapped} via a unitary equivalence on $(A,B)$.  We assume both
poles to be different, otherwise nothing needs to be done.
This procedure
is illustrated in \cref{fig:swap}, where poles $\xi_3 = $ \circled{3}/\circledletter{c}
and $\xi_4 = $ \circled{4}/\circledletter{d} are swapped. The swapping is achieved by
computing unitary matrices $Q_4$ and $Z_3$ that change the order of the eigenvalues in the
$2{\times}2$ blocks $A_{4:5,3:4}$ and $B_{4:5,3:4}$. These blocks are indicated with the
shaded region in \cref{fig:swap}. The equivalence transformation affects all elements
marked with $\otimes$ in pane II of \cref{fig:swap}.  Note that the ratios
\blcircled{\textcolor{white}{\textbf{4}}}{/}\blcircledletter{\textcolor{white}{\textbf{d}}}
and
\blcircled{\textcolor{white}{\textbf{3}}}{/}\blcircledletter{\textcolor{white}{\textbf{c}}}
are preserved under swapping but the subdiagonal values themselves can change.
\begin{figure}[H]
\centering
\input{fig/fig_swap.tikz}
\label{fig:swap}
\caption{Swapping poles in a Hessenberg pair: (I) before swap, (II) after swap.}
\end{figure}
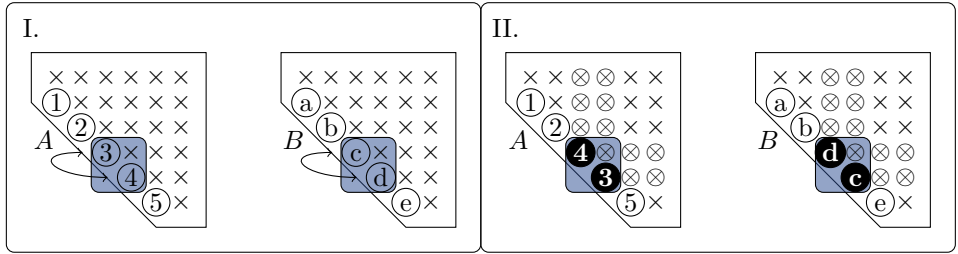
Swapping diagonal elements in an upper triangular matrix is a classical problem,
also used to reorder the (generalized) Schur form. It can be solved as the solution of
a coupled Sylvester equation \cite{Kagstrom1996} or by direct computations \cite{VanDooren1981}.
Its solution is unique if $\xi_i$
differs from $\xi_{i+1}$.

Details and solutions are found, e.g., in Watkins \cite{p509},
K{\aa}gstr\"om \& Poromaa \cite{Kagstrom1996}, and Van Dooren \cite{VanDooren1981}.
In \cite{VanDooren1981} it is also proven that the problem can be solved in a backwards
stable manner.

%% file: fig/fig_swap.tikz
\begin{tikzpicture}[scale=1.66,y=-1cm]
\def\xosA{0}
\node[] at (0.1+\xosA,0.7) {$A$};
\draw (0+\xosA,0) -- (1.4+\xosA,0) -- (1.4+\xosA,1.4) -- (1+\xosA,1.4) -- (0+\xosA,0.4) -- cycle;

\draw[fill={rgb:red,66;green,124;blue,244}, fill opacity=0.5,rounded corners=3pt] (0.48+\xosA,0.68) rectangle (0.92+\xosA,1.12);

\foreach \y in {1,...,6}{
	\foreach \x in {\y,...,6}{
		\node[] at (\x/5+\xosA,\y/5) {$\times$};
	}
};	

\foreach \x in {1,...,5}{
	\node[] (n\x) at (\x/5+\xosA,\x/5+0.2) {\circled{\x}};
}

\path[<->] (n3) edge [min distance=4mm,in=180,out=180,looseness=4] (n4);

\def\xosB{2}
\node[] at (0.1+\xosB,0.7) {$B$};

\draw (0+\xosB,0) -- (1.4+\xosB,0) -- (1.4+\xosB,1.4) -- (1+\xosB,1.4) -- (0+\xosB,0.4) -- cycle;

\draw[fill={rgb:red,66;green,124;blue,244}, fill opacity=0.5,rounded corners=3pt] (0.48+\xosB,0.68) rectangle (0.92+\xosB,1.12);

\foreach \y in {1,...,6}{
	\foreach \x in {\y,...,6}{
		\node[] at (\x/5+\xosB,\y/5) {$\times$};
	}
};	
\foreach \x in {1,...,5}{
	\node[] (l\x) at (\x/5+\xosB,\x/5+0.2) {\circledletter{\alphalph{\x}}};
}

\path[<->] (l3) edge [min distance=4mm,in=180,out=180,looseness=4] (l4);

\draw[rounded corners=3pt] (-0.2+\xosA,-0.4) rectangle (\xosB+1.6,1.6);
\node[] at (\xosA,-0.2) {I.};

\def\xosA{3.8}
\node[] at (0.1+\xosA,0.7) {$A$};
\draw (0+\xosA,0) -- (1.4+\xosA,0) -- (1.4+\xosA,1.4) -- (1+\xosA,1.4) -- (0+\xosA,0.4) -- cycle;

\draw[fill={rgb:red,66;green,124;blue,244}, fill opacity=0.5,rounded corners=3pt] (0.48+\xosA,0.68) rectangle (0.92+\xosA,1.12);

\node[] at (0.2+\xosA,0.2) {$\times$};
\node[] at (0.4+\xosA,0.2) {$\times$};
\node[] at (0.6+\xosA,0.2) {$\otimes$};
\node[] at (0.8+\xosA,0.2) {$\otimes$};
\node[] at (1+\xosA,0.2) {$\times$};
\node[] at (1.2+\xosA,0.2) {$\times$};
\node[] at (0.4+\xosA,0.4) {$\times$};
\node[] at (0.6+\xosA,0.4) {$\otimes$};
\node[] at (0.8+\xosA,0.4) {$\otimes$};
\node[] at (1+\xosA,0.4) {$\times$};
\node[] at (1.2+\xosA,0.4) {$\times$};
\node[] at (0.6+\xosA,0.6) {$\otimes$};
\node[] at (0.8+\xosA,0.6) {$\otimes$};
\node[] at (1+\xosA,0.6) {$\times$};
\node[] at (1.2+\xosA,0.6) {$\times$};
\node[] at (0.8+\xosA,0.8) {$\otimes$};
\node[] at (1+\xosA,0.8) {$\otimes$};
\node[] at (1.2+\xosA,0.8) {$\otimes$};
\node[] at (1+\xosA,1) {$\otimes$};
\node[] at (1.2+\xosA,1) {$\otimes$};
\node[] at (1.2+\xosA,1.2) {$\times$};

\node[] at (0.2+\xosA,0.4) {\circled{1}};
\node[] at (0.4+\xosA,0.6) {\circled{2}};
\node[] at (0.6+\xosA,0.8) {\blcircled{\textcolor{white}{\textbf{4}}}};
\node[] at (0.8+\xosA,1) {\blcircled{\textcolor{white}{\textbf{3}}}};
\node[] at (1+\xosA,1.2) {\circled{5}};

\def\xosB{5.8}
\node[] at (0.1+\xosB,0.7) {$B$};
\draw (0+\xosB,0) -- (1.4+\xosB,0) -- (1.4+\xosB,1.4) -- (1+\xosB,1.4) -- (0+\xosB,0.4) -- cycle;

\draw[fill={rgb:red,66;green,124;blue,244}, fill opacity=0.5,rounded corners=3pt] (0.48+\xosB,0.68) rectangle (0.92+\xosB,1.12);

\node[] at (0.2+\xosB,0.2) {$\times$};
\node[] at (0.4+\xosB,0.2) {$\times$};
\node[] at (0.6+\xosB,0.2) {$\otimes$};
\node[] at (0.8+\xosB,0.2) {$\otimes$};
\node[] at (1+\xosB,0.2) {$\times$};
\node[] at (1.2+\xosB,0.2) {$\times$};
\node[] at (0.4+\xosB,0.4) {$\times$};
\node[] at (0.6+\xosB,0.4) {$\otimes$};
\node[] at (0.8+\xosB,0.4) {$\otimes$};
\node[] at (1+\xosB,0.4) {$\times$};
\node[] at (1.2+\xosB,0.4) {$\times$};
\node[] at (0.6+\xosB,0.6) {$\otimes$};
\node[] at (0.8+\xosB,0.6) {$\otimes$};
\node[] at (1+\xosB,0.6) {$\times$};
\node[] at (1.2+\xosB,0.6) {$\times$};
\node[] at (0.8+\xosB,0.8) {$\otimes$};
\node[] at (1+\xosB,0.8) {$\otimes$};
\node[] at (1.2+\xosB,0.8) {$\otimes$};
\node[] at (1+\xosB,1) {$\otimes$};
\node[] at (1.2+\xosB,1) {$\otimes$};
\node[] at (1.2+\xosB,1.2) {$\times$};

\node[] at (0.2+\xosB,0.4) {\circledletter{a}};
\node[] at (0.4+\xosB,0.6) {\circledletter{b}};
\node[] at (0.6+\xosB,0.8) {\blcircledletter{\textcolor{white}{\textbf{d}}}};
\node[] at (0.8+\xosB,1) {\blcircledletter{\textcolor{white}{\textbf{c}}}};
\node[] at (1+\xosB,1.2) {\circledletter{e}};

\draw[rounded corners=3pt] (-0.2+\xosA,-0.4) rectangle (\xosB+1.6,1.6);
\node[] at (\xosA,-0.2) {II.};
\end{tikzpicture}

%% file: sections/03_reduction/reduction.tex
\section{Direct reduction to a proper Hessenberg pair.}
\label{sec:dirred}
\label{sec:red}

The rational QZ algorithm we propose in \Cref{sec:rqz} operates on a proper
Hessenberg pair.  If we are given an arbitrary matrix pencil $(A,B)$ not yet in (proper)
Hessenberg form, we first need to transform it to this form.  We use equivalences since we
are interested in the eigenvalues and, for reasons of numerical stability we will stick to
unitary equivalences. At the end of the section we will illustrate with a numerical
experiment that clever pole selection can lead to deflations, already in the reduction
process.

\subsection{The reduction algorithm.}
\label{subsec:dirred_algo}

We will transform an $n{\times}n$ matrix pair $(A,B)$ to a unitary equivalent Hessenberg
pair with a prescribed tuple of poles $\Xi = ( \xi_1$, $\hdots$ $\xi_{n-1})$.
The algorithm proceeds similarly to the direct reduction
to Hessenberg, triangular form, with the difference that a pole is introduced at every
step.

As in the classical reduction to Hessenberg, upper triangular pair we commence with
computing a QR factorization of $B=QR$ and updating the matrix pair to $(Q^*A, Q^*B)$.
The matrix  $Q^*B$ is now already in upper triangular form. This is shown in pane I of
\cref{fig:red1} for our running example matrix pair of size $5 \times 5$. Moreover, we assume in
the remainder of this section, that all zeros on the diagonal of $B$ --infinite
eigenvalues-- are removed \cite{b333}.

We will now bring the first column of $A$ to Hessenberg form.
In pane II, a zero is introduced in position $(5,1)$ of matrix $A$ by operating on
the last two rows.  This destroys the upper triangular shape in the last two rows of $B$.
The upper triangular shape can be restored by acting on columns $4$ and $5$ as shown in
pane III without destroying the newly created zero in $A$.


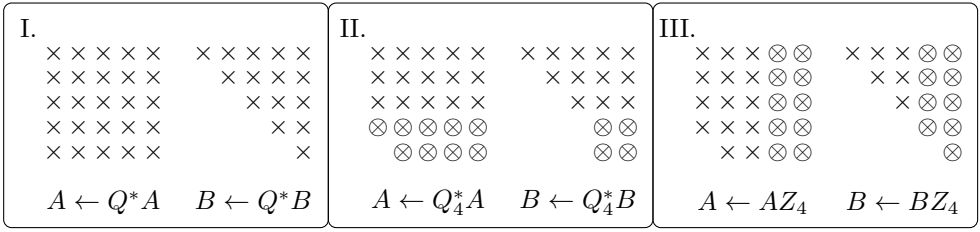
\begin{figure}[htp]
\centering
\input{fig/fig_reduction_1.tikz}
\label{fig:red1}
\caption{Reduction to a Hessenberg pencil. First part.}
\end{figure}

The process of introducing zeros in the first column of $A$ by acting on the rows and
maintaining the upper triangular structure in $B$ by acting on the columns can be repeated
until the first column of $A$ is brought to upper Hessenberg shape. This coincides
 with the standard reduction to a Hessenberg, triangular pair \cite{b333}.
We have arrived at
pane I of \cref{fig:red2}. 
The first column of $(A,B)$ is now already in the correct form, but has a pole at
$\infty$. We replace $\infty$ by another pole using the techniques from \Cref{subsec:manipulate}
applied to the first column of $(A,B)$ which is in Hessenberg form.
This is always possible, except when there is an obvious deflation in the top left corner,
meaning that the current pole is undefined as $0/0$.
This does not pose any problems: deflate and continue.
We start by introducing the last
pole $\xi_4 =$ \circled{4}{/}\circledletter{d} first, as in the following steps of the
reduction procedure this pole will move down to end up at the correct position at the
bottom of the subdiagonal.  The current state of the pair is visualized in pane II of
\cref{fig:red2}.

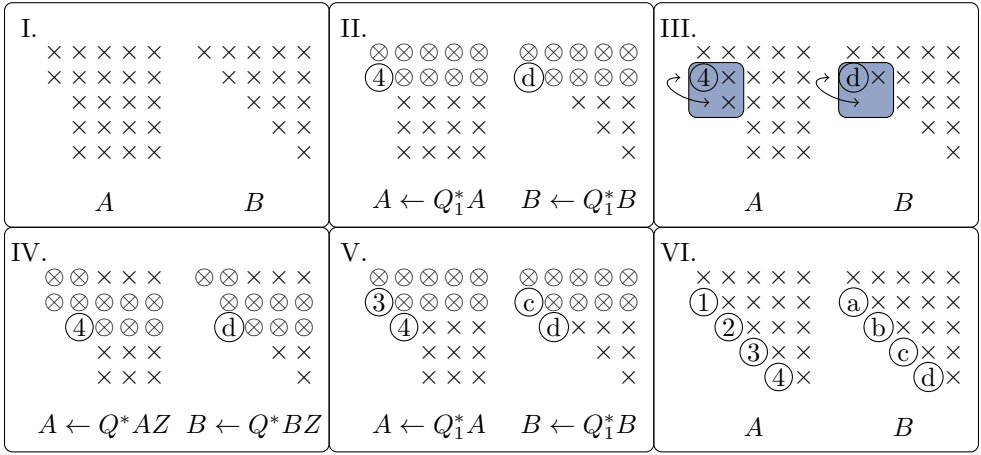
\begin{figure}[htp]
\centering
\input{fig/fig_reduction_2.tikz}
\label{fig:red2}
\caption{Reduction to a Hessenberg pencil. Second part.}
\end{figure}

The second column has been brought to Hessenberg, triangular form in pane III of
\cref{fig:red2} via the classical procedure of introducing zeros in the second column of
$A$ and maintaining the upper triangular structure in $B$.  This
procedure does not affect the existing pole $\xi_4$. At this stage, the first pole 
 equals $\xi_4$, while the second pole is $\infty$. 
The poles in the shaded region of pane III are now swapped via the techniques 
from \Cref{subsec:manipulate}, which moves the pole at $\infty$ to the top of the
matrix pair in pane IV.
The swapping technique can be used, as the two leading columns of $(A,B)$ are in
Hessenberg form at this stage of the reduction algorithm.
The swapping is always well defined, even if there is a succession of identical poles.
The pole $\xi_4$ has moved one row down and one column to the right.
The pair is now ready for the introduction of pole $\xi_3$ as
shown in pane V.  
This entire process of creating zeros, swapping poles, and introducing a new pole, can be
repeated until the end result of pane VI is obtained, and the matrix is in the desired
 Hessenberg form. 

After the reduction process, the matrix does not necessarily need to
 be in proper Hessenberg form. Possibly the pole $\xi_{n-1}$ coincides with an eigenvalue,
 allowing for deflation in the lower right corner. In this case one deflates $\xi_{n-1}$
 and checks whether $\xi_{n-2}$  leads to a deflation, and so forth, until the
 matrix has become proper.
It can also happen that during the reduction any of the interior poles deflate.
In this case the reduction can be continued on the separated parts of the pencil.
This situation is studied in the example of \Cref{subsec:example1}.

 The introduction of the poles takes an additional $6 n^3$ flops on top of the $8n^3$ operations
 required to reduce a pencil to Hessenberg, triangular form \cite{b037}. 

\subsection{Numerical experiment.} 
\label{subsec:example1}

We study two matrix pairs from the magnetohydrodynamics (MHD) dataset available 
in the Matrix Market collection \cite{Boisvert1997}.
The matrices are of sizes $416$ and $1280$ respectively, and known to be ill-conditioned.
They originate from a Galerkin finite element discretization of the underlying MHD problem.
Their spectrum consists of a tail along the negative real axis and a set of eigenvalues close to the imaginary axis.
In this numerical experiment we determine deflating subspaces for the two regions of
eigenvalues already during the reduction phase. The tests were run in
Matlab R2017b.

The idea is to introduce poles that make up a rational filter that mainly affects one
region of eigenvalues. To achieve this effect, the poles are chosen on a contour 
$\Gamma$ in the complex plane that contains the eigenvalues along the negative real axis.
This approach is inspired by the link between
contour integration methods \cite{Sakurai2003,Polizzi2009} and rational
filtering techniques \cite{VanBeeumen2017,VanBarel2016}.
In \Cref{sec:subspaceiter} we explain in full detail how introducing and swapping poles
implicitly applies a rational filter.

The poles are chosen on an elliptical contour $\Gamma = e(c,r_x,r_y,\theta)$, 
where $c$ is the center of the ellipse, $r_x$ is the radius in $x$-direction, $r_y$ is the
radius in the $y$-direction, and $\theta$ is the angle over which the axes of the ellipse are rotated.
For the smaller problem, $\Gamma$ is selected as $e(-1.3,1.5,3,0)$ and discretized in $120$ nodes.
For the larger problem, $\Gamma = e(-25,27,6,0)$ and it is discretized in $400$ nodes.
These nodes are the poles  introduced during the reduction to Hessenberg form. The aim is
to get the pair improper, enforcing thereby a middle deflation separating the two
regions. In case of a \emph{middle} deflation we continue introducing poles on the separated parts.

The results are presented in \Cref{fig:numexp11,fig:numexp12,fig:numexp13}.
\Cref{fig:numexp11} shows an overview of the spectrum of both matrix pairs.
The two regions of eigenvalues are indicated with different markers.
The box in \Cref{fig:numexp11} marks the area in which
\Cref{fig:numexp12} will zoom in; it shows where the regions meet in detail,
together with the poles of the Hessenberg pair.  

\Cref{fig:numexp13} displays the magnitude of the subdiagonal elements
$|a_{i+1,i}| + |b_{i+1,i}|$.  All poles which are considered numerically zero and thus
lead to a deflation are emphasized in a
shaded rectangle.
Typically some of the first and last poles are deflated, but more important is the
presence of
interior deflations.  This happens at poles $103$ to $106$ after $160$ poles
have been introduced in the pair of size $416$.  For the larger pair, poles $317$ to $321$
are deflated after $621$ poles have been introduced.
The eigenvalues outside $\Gamma$ are located in the top left part of the Hessenberg pair,
those inside $\Gamma$ appear after the interior deflation.

\begin{figure}[htp]
\centering
\input{fig/numexp1/fig_numexp1_1.tikz}
\caption{Eigenvalues in region 1 (\protect\scalebox{1.4}{\ref{line:ne1clusterone}}; bow-shape) and region 2
(\protect\scalebox{1.4}{\ref{line:ne1clustertwo}}; close to the real axis).
On the left we have the problem of size 416 and on the right 1280.}
\label{fig:numexp11}
\end{figure}
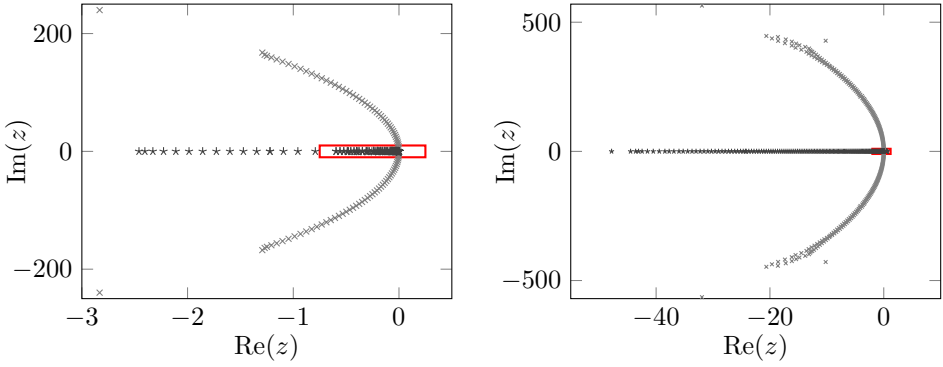

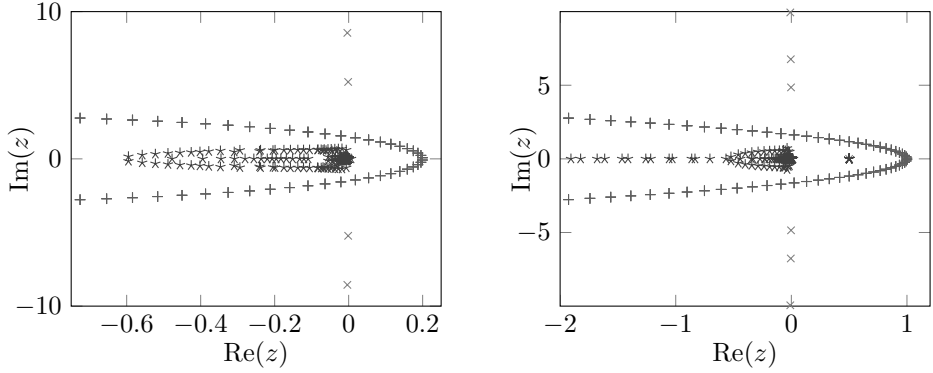
\begin{figure}[htp]
\centering
\input{fig/numexp1/fig_numexp1_2.tikz}
\caption{Close-up of the central part where the regions meet for the problem of size
  416 and 1280. The legend is identical to the one of \Cref{fig:numexp11} extended with
  the poles (\ref{line:ne1poles}; on the ellipse around the real axis) .}
\label{fig:numexp12}
\end{figure}

\begin{figure}[htp]
\centering
\includegraphics[scale=0.97]{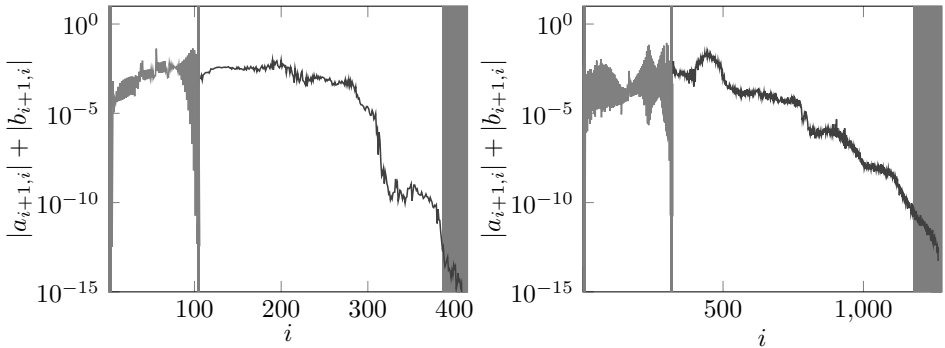}
\caption{Magnitudes of the subdiagonal elements in the matrix pair after the Hessenberg reduction
for the problem of size 416 (left) and 1280 (right).}
\label{fig:numexp13}
\end{figure}

This numerical experiment shows that deflating subspaces containing regions of
eigenvalues can be found already during the reduction to Hessenberg form.  We like to
stress that deflation is obtained without any of the poles converging towards an
eigenvalue, but by choosing poles on a contour such that they construct an effective
rational filter.

%% file: fig/fig_reduction_1.tikz
\begin{tikzpicture}[scale=1.66,y=-1cm]
\def\xosA{0}
\node[] at (0.6+\xosA,1.4) {$A \leftarrow Q^{*}A$};
\foreach \y in {1,...,5}{
	\foreach \x in {1,...,5}{
		\node[] at (\x/5+\xosA,\y/5) {$\times$};
	}
};	

\def\xosB{1.2}
\node[] at (0.6+\xosB,1.4) {$B \leftarrow Q^{*}B$};
\foreach \y in {1,...,5}{
	\foreach \x in {\y,...,5}{
		\node[] at (\x/5+\xosB,\y/5) {$\times$};
	}
};
\draw[rounded corners=3pt] (-0.2+\xosA,-0.2) rectangle (\xosB+1.2,1.6);
\node[] at (\xosA,0) {I.};

\def\xosA{2.6}
\node[] at (0.6+\xosA,1.4) {$A \leftarrow Q_4^{*}A$};
\foreach \y in {1,...,3}{
	\foreach \x in {1,...,5}{
		\node[] at (\x/5+\xosA,\y/5) {$\times$};
	}
};
\foreach \y in {4,5}{
	\foreach \x in {2,...,5}{
		\node[] at (\x/5+\xosA,\y/5) {$\otimes$};
	}
};
\node[] at (0.2+\xosA,0.8) {$\otimes$};	

\def\xosB{3.8}
\node[] at (0.6+\xosB,1.4) {$B \leftarrow Q_4^{*}B$};
\foreach \y in {1,...,3}{
	\foreach \x in {\y,...,5}{
		\node[] at (\x/5+\xosB,\y/5) {$\times$};
	}
};
\foreach \y in {4,5}{
	\foreach \x in {4,5}{
		\node[] at (\x/5+\xosB,\y/5) {$\otimes$};
	}
};

\draw[rounded corners=3pt] (-0.2+\xosA,-0.2) rectangle (\xosB+1.2,1.6);
\node[] at (\xosA,0) {II.};

\def\xosA{5.2}
\node[] at (0.6+\xosA,1.4) {$A \leftarrow A Z_4$};
\foreach \y in {1,...,4}{
		\node[] at (0.2+\xosA,\y/5) {$\times$};
};
\foreach \y in {1,...,5}{
	\foreach \x in {2,3}{
		\node[] at (\x/5+\xosA,\y/5) {$\times$};
	}
};
\foreach \y in {1,...,5}{
	\foreach \x in {4,5}{
		\node[] at (\x/5+\xosA,\y/5) {$\otimes$};
	}
};

\def\xosB{6.4}
\node[] at (0.6+\xosB,1.4) {$B \leftarrow B Z_4$};
\foreach \y in {1,...,3}{
	\foreach \x in {\y,...,3}{
		\node[] at (\x/5+\xosB,\y/5) {$\times$};
	}
};
\foreach \y in {1,...,4}{
	\foreach \x in {4,5}{
		\node[] at (\x/5+\xosB,\y/5) {$\otimes$};
	}
};
\node[] at (1+\xosB,1) {$\otimes$};

\draw[rounded corners=3pt] (-0.2+\xosA,-0.2) rectangle (\xosB+1.2,1.6);
\node[] at (\xosA,0) {III.};

\end{tikzpicture}

%% file: fig/fig_reduction_2.tikz
\begin{tikzpicture}[scale=1.66,y=-1cm]
\def\xosA{0}
\node[] at (0.6+\xosA,1.4) {$A$};
\node[] at (0.2+\xosA,0.2) {$\times$};
\node[] at (0.2+\xosA,0.4) {$\times$};
\foreach \y in {1,...,5}{
	\foreach \x in {2,...,5}{
		\node[] at (\x/5+\xosA,\y/5) {$\times$};
	}
};	

\def\xosB{1.2}
\node[] at (0.6+\xosB,1.4) {$B$};
\foreach \y in {1,...,5}{
	\foreach \x in {\y,...,5}{
		\node[] at (\x/5+\xosB,\y/5) {$\times$};
	}
};
\draw[rounded corners=3pt] (-0.2+\xosA,-0.2) rectangle (\xosB+1.2,1.6);
\node[] at (\xosA,0) {I.};

\def\xosA{2.6}
\node[] at (0.6+\xosA,1.4) {$A \leftarrow Q_1^{*}A$};
\node[] at (0.2+\xosA,0.2) {$\otimes$};
\foreach \y in {1,2}{
	\foreach \x in {2,...,5}{
		\node[] at (\x/5+\xosA,\y/5) {$\otimes$};
	}
};
\foreach \y in {3,...,5}{
	\foreach \x in {2,...,5}{
		\node[] at (\x/5+\xosA,\y/5) {$\times$};
	}
};
\node[] at (0.2+\xosA,0.4) {\circled{4}};

\def\xosB{3.8}
\node[] at (0.6+\xosB,1.4) {$B \leftarrow Q_1^{*}B$};
\foreach \y in {1,2}{
	\foreach \x in {\y,...,5}{
		\node[] at (\x/5+\xosB,\y/5) {$\otimes$};
	}
};
\node[] at (0.2+\xosB,0.4) {\circledletter{d}};

\foreach \y in {3,...,5}{
	\foreach \x in {\y,...,5}{
		\node[] at (\x/5+\xosB,\y/5) {$\times$};
	}
};

\draw[rounded corners=3pt] (-0.2+\xosA,-0.2) rectangle (\xosB+1.2,1.6);
\node[] at (\xosA,0) {II.};

\def\xosA{5.2}

\draw[fill={rgb:red,66;green,124;blue,244}, fill opacity=0.5,rounded corners=3pt] (0.08+\xosA,0.28) rectangle (0.52+\xosA,0.72);

\node[] at (0.6+\xosA,1.4) {$A$};
\node[] at (0.2+\xosA,0.2) {$\times$};
\node[] at (0.4+\xosA,0.2) {$\times$};
\node[] at (0.4+\xosA,0.4) {$\times$};
\node[] (l2) at (0.4+\xosA,0.6) {$\times$};
\foreach \y in {1,...,5}{
	\foreach \x in {3,...,5}{
		\node[] at (\x/5+\xosA,\y/5) {$\times$};
	}
};
\node[] (l1) at (0.2+\xosA,0.4) {\circled{4}};

\path[<->] (l1) edge [min distance=1mm,in=180,out=180,looseness=2] (l2);

\def\xosB{6.4}
\node[] at (0.6+\xosB,1.4) {$B$};

\draw[fill={rgb:red,66;green,124;blue,244}, fill opacity=0.5,rounded corners=3pt] (0.08+\xosB,0.28) rectangle (0.52+\xosB,0.72);

\foreach \y in {1,...,5}{
	\foreach \x in {\y,...,5}{
		\node[] at (\x/5+\xosB,\y/5) {$\times$};
	}
};
\node[] (r1) at (0.2+\xosB,0.4) {\circledletter{d}};
\node[] (r2) at (0.4+\xosB,0.6) {\phantom{$\times$}};

\path[<->] (r1) edge [min distance=1mm,in=180,out=180,looseness=2] (r2);

\draw[rounded corners=3pt] (-0.2+\xosA,-0.2) rectangle (\xosB+1.2,1.6);
\node[] at (\xosA,0) {III.};

\def\yosr2{1.8}

\def\xosA{0}
\node[] at (0.6+\xosA,1.4+\yosr2) {$A \leftarrow Q^* A Z$};
\node[] at (0.2+\xosA,0.2+\yosr2) {$\otimes$};
\node[] at (0.2+\xosA,0.4+\yosr2) {$\otimes$};
\node[] at (0.4+\xosA,0.4+\yosr2) {$\otimes$};
\node[] at (0.6+\xosA,0.2+\yosr2) {$\times$};
\node[] at (0.4+\xosA,0.2+\yosr2) {$\otimes$};
\foreach \y in {2,3}{
	\foreach \x in {3,...,5}{
		\node[] at (\x/5+\xosA,\y/5+\yosr2) {$\otimes$};
	}
};

\foreach \y in {1}{
	\foreach \x in {4,5}{
		\node[] at (\x/5+\xosA,\y/5+\yosr2) {$\times$};
	}
};
\foreach \y in {4,5}{
	\foreach \x in {3,4,5}{
		\node[] at (\x/5+\xosA,\y/5+\yosr2) {$\times$};
	}
};

\node[] (l1) at (0.4+\xosA,0.6+\yosr2) {\circled{4}};

\def\xosB{1.2}
\node[] at (0.6+\xosB,1.4+\yosr2) {$B \leftarrow Q^* B Z$};
\foreach \y in {1}{
	\foreach \x in {3,4,5}{
		\node[] at (\x/5+\xosB,\y/5+\yosr2) {$\times$};
	}
};

\foreach \y in {4,5}{
	\foreach \x in {\y,...,5}{
		\node[] at (\x/5+\xosB,\y/5+\yosr2) {$\times$};
	}
};

\node[] at (0.4+\xosB,0.2+\yosr2) {$\otimes$};
\node[] at (0.2+\xosB,0.2+\yosr2) {$\otimes$};
\foreach \y in {2,3}{
	\foreach \x in {\y,...,5}{
		\node[] at (\x/5+\xosB,\y/5+\yosr2) {$\otimes$};
	}
};

\node[] at (0.4+\xosB,0.6+\yosr2) {\circledletter{d}};

\draw[rounded corners=3pt] (-0.2+\xosA,-0.2+\yosr2) rectangle (\xosB+1.2,1.6+\yosr2);
\node[] at (\xosA,0+\yosr2) {IV.};

\def\xosA{2.6}
\node[] at (0.6+\xosA,1.4+\yosr2) {$A \leftarrow Q_1^* A$};
\node[] at (0.2+\xosA,0.2+\yosr2) {$\otimes$};
\foreach \y in {1,2}{
	\foreach \x in {2,...,5}{
		\node[] at (\x/5+\xosA,\y/5+\yosr2) {$\otimes$};
	}
};

\foreach \y in {3,4,5}{
	\foreach \x in {3,4,5}{
		\node[] at (\x/5+\xosA,\y/5+\yosr2) {$\times$};
	}
};

\node[] at (0.2+\xosA,0.4+\yosr2) {\circled{3}};
\node[] at (0.4+\xosA,0.6+\yosr2) {\circled{4}};

\def\xosB{3.8}
\node[] at (0.6+\xosB,1.4+\yosr2) {$B \leftarrow Q_1^* B$};

\foreach \y in {3,4,5}{
	\foreach \x in {\y,...,5}{
		\node[] at (\x/5+\xosB,\y/5+\yosr2) {$\times$};
	}
};

\foreach \y in {1,2}{
	\foreach \x in {\y,...,5}{
		\node[] at (\x/5+\xosB,\y/5+\yosr2) {$\otimes$};
	}
};

\node[] at (0.2+\xosB,0.4+\yosr2) {\circledletter{c}};
\node[] at (0.4+\xosB,0.6+\yosr2) {\circledletter{d}};

\draw[rounded corners=3pt] (-0.2+\xosA,-0.2+\yosr2) rectangle (\xosB+1.2,1.6+\yosr2);
\node[] at (\xosA,0+\yosr2) {V.};

\def\xosA{5.2}
\node[] at (0.6+\xosA,1.4+\yosr2) {$A$};

\foreach \y in {1,...,5}{
	\foreach \x in {\y,...,5}{
		\node[] at (\x/5+\xosA,\y/5+\yosr2) {$\times$};
	}
};

\node[] at (0.2+\xosA,0.4+\yosr2) {\circled{1}};
\node[] at (0.4+\xosA,0.6+\yosr2) {\circled{2}};
\node[] at (0.6+\xosA,0.8+\yosr2) {\circled{3}};
\node[] at (0.8+\xosA,1+\yosr2) {\circled{4}};

\def\xosB{6.4}
\node[] at (0.6+\xosB,1.4+\yosr2) {$B$};

\foreach \y in {1,...,5}{
	\foreach \x in {\y,...,5}{
		\node[] at (\x/5+\xosB,\y/5+\yosr2) {$\times$};
	}
};

\node[] at (0.2+\xosB,0.4+\yosr2) {\circledletter{a}};
\node[] at (0.4+\xosB,0.6+\yosr2) {\circledletter{b}};
\node[] at (0.6+\xosB,0.8+\yosr2) {\circledletter{c}};
\node[] at (0.8+\xosB,1+\yosr2) {\circledletter{d}};

\draw[rounded corners=3pt] (-0.2+\xosA,-0.2+\yosr2) rectangle (\xosB+1.2,1.6+\yosr2);
\node[] at (\xosA,0+\yosr2) {VI.};

\end{tikzpicture}

%% file: fig/numexp1/fig_numexp1_1.tikz

\definecolor{clust1color}{gray}{0.5}
\definecolor{clust2color}{gray}{0.25}
\definecolor{polecolor}{gray}{0.37}

\begin{tikzpicture}


\pgfplotstableread{fig/numexp1/data416/poles416.dat}
 {\poletable} 

\pgfplotstableread{fig/numexp1/data416/C1416.dat}
 {\clusteronetable}
 
\pgfplotstableread{fig/numexp1/data416/C2416.dat}
 {\clustertwotable}

\begin{axis}[%
xlabel=$\operatorname{Re}(z)$,
ylabel=$\operatorname{Im}(z)$,
xmin=-3,
xmax=0.5,
ymin=-250,
ymax=250,
at={(0cm,0cm)},
width=6.5cm,
height=5.5cm,
ylabel shift = -0.5 cm,
xlabel shift = -0.15 cm
]

\addplot [only marks, mark=x, mark options={solid, color=clust1color, scale=0.8}] table[x index=0, y index=1] from \clusteronetable ; \label{line:ne1clusterone}

\addplot [only marks, mark=star, mark options={solid, color=clust2color, scale=0.8}] table[x index=0, y index=1] from \clustertwotable ; \label{line:ne1clustertwo}


\draw [red, thick] (axis cs:-0.75,-10) rectangle (axis cs:0.25,10);

\end{axis}

\def\yos{0cm}
\def\xos{6.5cm}

\pgfplotstableread{fig/numexp1/data1280/poles1280.dat}
 {\poletable} 

\pgfplotstableread{fig/numexp1/data1280/C11280.dat}
 {\clusteronetable}
 
\pgfplotstableread{fig/numexp1/data1280/C21280.dat}
 {\clustertwotable}
 
\begin{axis}[%
xlabel=$\operatorname{Re}(z)$,
ylabel=$\operatorname{Im}(z)$,
xmin=-55,
xmax=10,
restrict x to domain=-56:11,
ymin=-570,
ymax=570,
restrict y to domain=-575:575,
at={(0cm+\xos,0cm-\yos)},
width=6.5cm,
height=5.5cm,
ylabel shift = -0.5 cm,
xlabel shift = -0.15 cm
]

\draw [red, thick] (axis cs:-2,-10) rectangle (axis cs:1.2,10);

\addplot [only marks, mark=x, mark options={solid, color=clust1color, scale=0.5}] table[x index=0, y index=1] from \clusteronetable ;

\addplot [only marks, mark=star, mark options={solid, color=clust2color, scale=0.5}] table[x index=0, y index=1] from \clustertwotable ; 


\end{axis}

\end{tikzpicture}

%% file: fig/numexp1/fig_numexp1_2.tikz

\definecolor{clust1color}{gray}{0.5}
\definecolor{clust2color}{gray}{0.25}
\definecolor{polecolor}{gray}{0.37}

\begin{tikzpicture}


\pgfplotstableread{fig/numexp1/data416/poles416.dat}
 {\poletable} 

\pgfplotstableread{fig/numexp1/data416/C1416.dat}
 {\clusteronetable}
 
\pgfplotstableread{fig/numexp1/data416/C2416.dat}
 {\clustertwotable}

\begin{axis}[%
xlabel=$\operatorname{Re}(z)$,
ylabel=$\operatorname{Im}(z)$,
xmin=-0.75,
xmax=0.25,
restrict x to domain=-0.8:0.3,
ymin=-10,
ymax=10,
restrict y to domain=-11:11,
at={(0cm,0cm)},
width=6.5cm,
height=5.5cm,
ylabel shift = -0.5 cm,
xlabel shift = -0.15 cm,
ytick={-10,0,10}
]

\addplot [only marks, mark=x, mark options={solid, color=clust1color, scale=1}] table[x index=0, y index=1] from \clusteronetable ; 

\addplot [only marks, mark=star, mark options={solid, color=clust2color, scale=1}] table[x index=0, y index=1] from \clustertwotable ;

\addplot [only marks, mark=+, mark options={solid, color=polecolor, scale=1}] table[x index=0, y index=1] from \poletable ; \label{line:ne1poles}

\end{axis}

\def\yos{0cm}
\def\xos{6.5cm}

\pgfplotstableread{fig/numexp1/data1280/poles1280.dat}
 {\poletable} 

\pgfplotstableread{fig/numexp1/data1280/C11280.dat}
 {\clusteronetable}
 
\pgfplotstableread{fig/numexp1/data1280/C21280.dat}
 {\clustertwotable}
 
\begin{axis}[%
xlabel=$\operatorname{Re}(z)$,
ylabel=$\operatorname{Im}(z)$,
restrict x to domain=-2.2:1.4,
xmin=-2,
xmax=1.2,
restrict y to domain=-11:11,
ymin=-10,
ymax=10,
at={(0cm+\xos,0cm-\yos)},
width=6.5cm,
height=5.5cm,
ylabel shift = -0.5 cm,
xlabel shift = -0.15 cm,
ytick={-5,0,5}
]

\addplot [only marks, mark=x, mark options={solid, color=clust1color, scale=1}] table[x index=0, y index=1] from \clusteronetable ;

\addplot [only marks, mark=star, mark options={solid, color=clust2color, scale=1}] table[x index=0, y index=1] from \clustertwotable ;

\addplot [only marks, mark=+, mark options={solid, color=polecolor, scale=1}] table[x index=0, y index=1] from \poletable ;

\end{axis}

\end{tikzpicture}

%% file: sections/04_rqz1/rqz1.tex
\section{Implicitly single shifted rational QZ step.}
\label{sec:rqz}
In this section we present the implicit RQZ step for a Hessenberg pair. At the end of this
section numerical experiments are included to illustrate the performance and accuracy of the
algorithm.

The algorithm operates on proper Hessenberg pairs. These pairs could be the result of the
reduction procedure presented in \Cref{sec:red} or they could be given directly,
e.g., as coming from an iterative rational Krylov method, where one would like to compute
the eigenvalues of the projected Hessenberg pair. These eigenvalues are
approximations to the eigenvalues of the original problem and are called the
Ritz values if the final pole is at $\infty$ or Harmonic Ritz values for a final pole at $0$ \cite{DeSamblanx1999, Ruhe1998a}.

\subsection{The algorithm.}
\label{subsec:implRQZ}

Before describing the algorithm we like to comment on the nomenclature.  We use
both the terms \emph{poles} $\xi$ and \emph{shifts} $\varrho$ to refer to elements on the
subdiagonal of a Hessenberg pair. In fact our shifts are poles as well, but we 
%
typically consider poles as subdiagonal elements that are sustained in the Hessenberg
pair, while shifts are introduced and removed in a single implicit RQZ step. A shift is pushed in
at the top, chased to the bottom, and pulled out at the end.

We introduce the RQZ procedure with an example.  
Given a
$5{\times}5$ Hessenberg pair $(A,B)$ with poles $\xi_1 = $
\circled{1}{/}\circledletter{a}, $\xi_2 = $ \circled{2}{/}\circledletter{b}, $\xi_3 =
$ \circled{3}{/}\circledletter{c}, $\xi_4 = $ \circled{4}{/}\circledletter{d} $\in
\bar{\CC}$.
The RQZ step consists out of three stages, similar to all algorithms of implicit QR type. These are an initialization,
a chasing, and a finalization phase. 

\paragraph{Initialization.}
Suppose we are given a shift
$\varrho = \oplus / \oslash \in \bar{\CC}$,
for instance the Wilkinson shift. Pane I in \cref{fig:implRQZ} shows the
Hessenberg pair in its initial state.  The shift\footnote{A shift equal to a pole will not
  result in a breakdown, but leads to slow or no convergence at all (see 
  \Cref{sec:subspaceiter}). In practice shifts should be taken different from the poles.}  is introduced in pane II by changing the first pole with a transformation
$Q_1$.

\paragraph{Chasing.}
Panes III-V show how the shift is relocated from the first position on the subdiagonal to
position $n{-}1$ by repeatedly swapping it with the poles of the Hessenberg pair. The shift is
chased to the bottom.
The matrix elements that are changed in every step are marked with an $\otimes$. 

During this procedure the shift will move from the top left to the bottom right and all
poles will move up one position to the top-left corner.
The assumption that the shift differs from the poles $\varrho\neq\xi_i$,
for all $i$, ensures that none of
 the swapping operations equals an identity, otherwise the downward movement of the shift
 will undo the upward movement of the corresponding pole.

\paragraph{Finalization.}
Finally, in pane VI, one last operation can be performed where we have the possibility to
remove the shift $\varrho$ and introduce any new pole $\hat{\xi}_4 = $ 
\circled{5}{/}\circledletter{e} $\in \bar{\CC}$, via the procedure described in
\Cref{subsec:manipulate}.

\begin{figure}[htp]
\label{fig:implRQZ}
\centering
\input{fig/fig_implRQZ.tikz}
\caption{Single shifted implicit RQZ step on a $5{\times}5$ Hessenberg pair with shift $\varrho$.}
\end{figure}
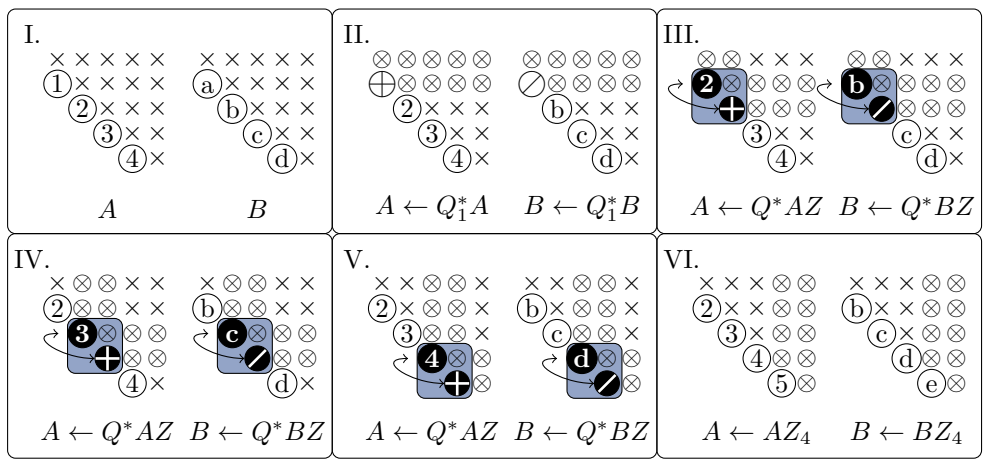

The reader familiar with the classical QZ algorithm \cite{Moler1973} or the
condensed QZ algorithm \cite{Vandebril2013} can verify that the algorithm described here 
generalizes both methods. 
In the QZ algorithm \cite{b333} one chases a bulge and in the final step the new pole was
always put to $\infty$ thereby restoring the upper triangular form of $B$.
In the condensed QZ algorithm \cite{Vandebril2013} a rotation was chased and in the final
step one allowed for a pole to be at $0$ or $\infty$. 

In the rational QZ algorithm we chase a shift instead of a bulge or a rotation. However, the
shift is encoded in the rotation and bulge as well, as it is found as an eigenvalue of
Watkins' bulge pencil \cite[Section 5]{Watkins2000}, \cite{b333}; the other eigenvalue in the bulge pencil
is $\infty$. If we consider the same bulge pencil in the rational QZ case we see that the
eigenvalue at $\infty$ is replaced by a pole of the pencil.
Moreover, also the pole swapping technique is nothing else than the bulge exchange
interpretation of Watkins \cite{p509}. 




\subsection{Shifts, poles, and deflation.}

In order to implement the RQZ algorithm and in particular a single RQZ step, we need good
strategies to select the shift, the new pole introduced at the very end, and a
procedure to check if there are deflations. 

For the shifts we typically take the Wilkinson
shift. This is the eigenvalue of the trailing $2{\times}2$ block that
is closest to $a_{nn}/b_{nn}$.
For the poles there are several options: one could as well consider a Wilkinson
strategy determined by the $2{\times}2$ block in the upper-left corner or one could use other
techniques such as poles on a contour to do filtering, see, e.g., 
\Cref{subsec:example1}.  Optimal pole selection is a difficult issue and very
problem specific, this is beyond the scope of this manuscript; in the numerical
experiments we will test some straightforward options.

The deflation criterion for the poles $\xi_2, \hdots, \xi_{n-2}$ is obvious.  If one  of
these is not in $\bar{\CC}$, 
the problem can be split into
smaller, independent problems. This means in fact that for a certain $i$, two subdiagonal
elements $a_{i+1,i}$ and $b_{i+1,i}$ are simultaneously zero.  To numerically check this
we use the classical relative criterion taking the sizes of the neighbouring elements into
consideration \cite{b037}, i.e., $|a_{i+1,i}| \leq c \epsilon_m (|a_{i,i}|+|a_{i+1,i+1}|)$ and $|b_{i+1,i}| \leq c \epsilon_m (|b_{i,i}|+|b_{i+1,i+1}|)$,
with $\epsilon_m$ the machine precision and $c$ a small constant.


The situation at the top or at the bottom, which are the exterior poles $\xi_1$ and
$\xi_{n-1}$ is more peculiar.  Whereas the interior poles are fixed, the exterior ones can
be altered. Instead of changing $\xi_1$ or $\xi_{n-1}$ to another pole, we would like to
know whether it is possible to move them outside of $\bar{\CC}$: we would like to
deflate an eigenvalue.
To this end we need to create two zeros with a single operation
such that the pair is no longer proper. We discuss the situation at the bottom right, the
top-left corner proceeds similar. Suppose 
we
would like to introduce zeros in the penultimate subdiagonal positions, which are
$a_{n,n-1}$ and $b_{n,n-1}$. 
This is possible if the matrix $\bigl[ \begin{smallmatrix}a_{n,n-1} & a_{nn}\\ b_{n,n-1} &
  b_{nn}\end{smallmatrix}\bigr]$ is of rank $1$. If this is the case we can simultaneously
annihilate the subdiagonal elements by operating on the columns of $(A,B)$.
In our experiments we assume the matrix to be of rank 1 if the smallest singular value
is less than $\epsilon_m$


\subsection{Numerical experiment.}
\label{subsec:example2}
We apply the RQZ method on two sets of problems: random matrix pairs and two problems from
fluid dynamics. We are interested in the accuracy and speed.

\paragraph{Random matrix pairs.}
We test the single shift RQZ algorithm on $9$ randomly generated, complex-valued
matrix pairs with sizes ranging from $100$ to $1000$. The results are averaged
over $10$ runs.
The pairs are first reduced to Hessenberg pairs with all poles at infinity, implying
 that no additional computational work has been done compared to the reduction phase of the QZ method.
The shift is always taken as the Wilkinson shift. 
The poles are selected according to four different strategies: 
poles at infinity, poles at zero, random poles, and poles chosen as the Wilkinson shift from the upper-left $2 \times 2$ block.
The last choice is called the Wilkinson pole.
 
The results are summarized in \cref{fig:numexp2,fig:numexp2a}. The left pane of \cref{fig:numexp2} shows the relative backward errors 
$\|\hat{A} - Q^* A Z\|_2/\|A\|_2$ and $\|\hat{B} - Q^* B Z\|_2/\|B\|_2$ for the reduction
to a Hessenberg pair 
(lines without markers) and the backward error on the Schur form for the four different pole strategies.
The backward error is small in all cases.
The right pane shows the average number of iterations per eigenvalue.
Clearly, the Wilkinson pole requires the least number of iterations per eigenvalue. It
requires on average 
$1.5\%$ less iterations than the classic choice of poles at infinity. Random poles and poles
at zero perform the worst.

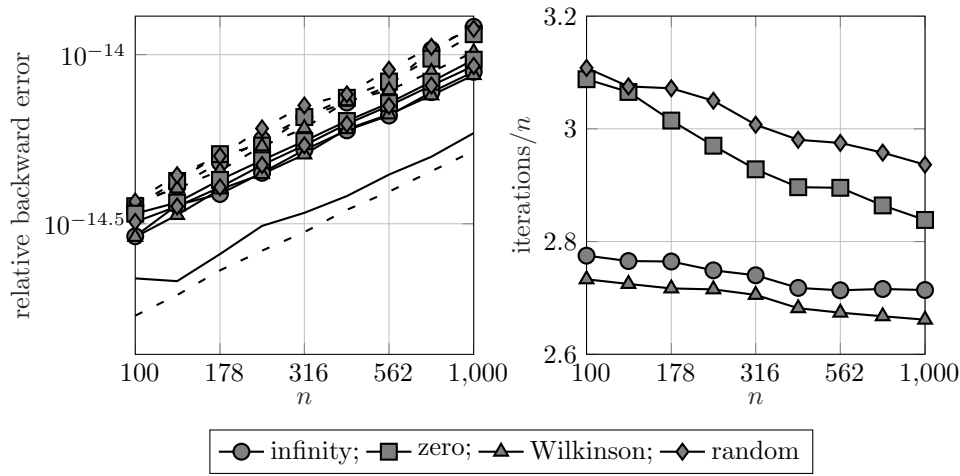
\begin{figure}[htp]
\centering
\input{fig/numexp2/fig_numexp2_bwe_itpev.tikz}
\caption{On the left the relative backward errors related to the reduction to a Hessenberg
  pair (no markers) and to the Schur form (with markers) are demonstrated. The error on $A$ is represented with a
  dashed line and the error on $B$ with a full line. On the right we see the 
average number of iterations per eigenvalue for the four different pole strategies.}
\label{fig:numexp2}
\end{figure}

\Cref{fig:numexp2a} shows the total number of pole swaps scaled with $n^2$. 
The scaling factor is used since the number of pole swaps per iteration
is $O(n)$ and the number of iterations is also $O(n)$.
This measure of performance depends heavily on the positions were deflations occur
and as such gives a much better view on the algorithmic behavior.
The order of the four strategies remains the
same, but the gains with Wilkinson poles increase up to $4\%$. This signals the
occurrence of deflations at other spots than only in the lower-right corner as is
typically the case in the classical setting.

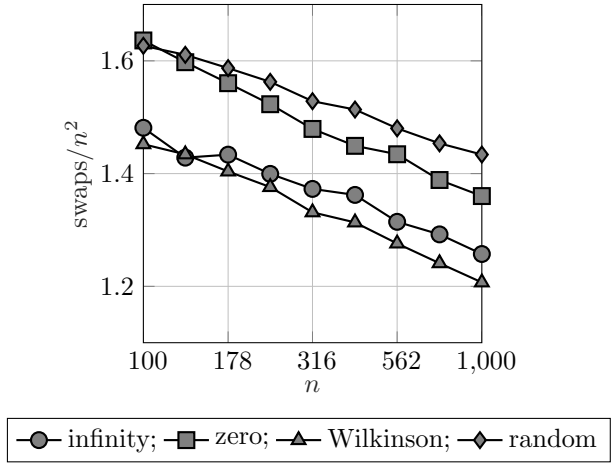
\begin{figure}[htp]
\centering
\input{fig/numexp2/fig_numexp2_swapspev2.tikz}
\caption{The total number of swaps scaled by $n^2$ for the computation of the Schur form for the four different pole strategies.}
\label{fig:numexp2a}
\end{figure}

\paragraph{IFISS problems.}
In this experiment we apply the RQZ method on two problems from fluid dynamics generated with IFISS \cite{Elman07,Elman14}.
The first problem originates from a model for the flow in a unit-square
cavity, the second problem comes from a model for the flow around an obstacle. Both models are  discretized, resulting 
 in two real, generalized eigenvalue problems. The \emph{cavity flow} problem is of size $2467$, the \emph{obstacle flow} 
problem of size $2488$. We applied the single shift RQZ method after initial reduction to Hessenberg form with poles at infinity.
Wilkinson shifts are employed in all cases. We used poles at infinity and Wilkinson poles.
The spectra of the matrices is shown in \Cref{fignumexp2b}.

\begin{figure}[htp]
\centering
\includegraphics[scale=0.8]{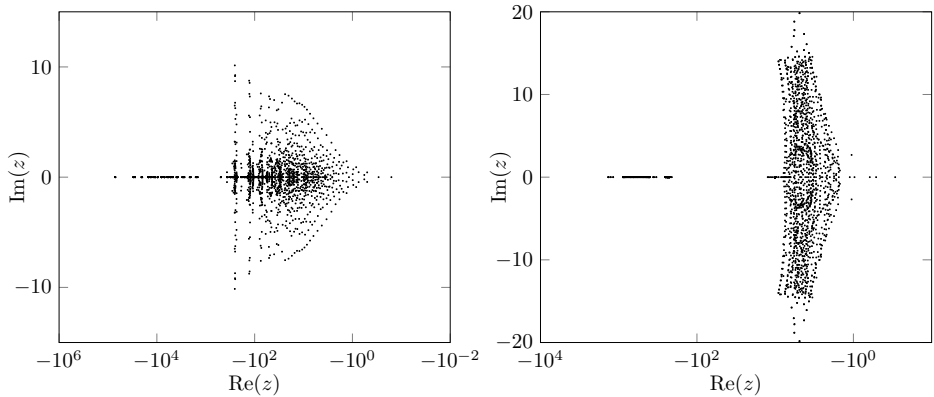}
\caption{On the left the spectrum of the cavity flow problem and on the right the spectrum
  of the obstacle flow problem are shown.
}
\label{fignumexp2b}
\end{figure}  

The results of the experiment are summarized in \Cref{tab:numexp2b}.
It lists the relative backward error on the Schur form for both $A$ and $B$ for both
problems and the two pole strategies. The backward error is very good  in all cases.
The table also lists the average iterations per eigenvalue and how this
compares relatively to the result of poles at infinity.
We observe that the average number of swaps and iterations when employing
Wilkinson poles is always below the numbers generated by 
the classical approach.
\begin{table}[htp]
\centering
\caption{Results of the RQZ method on the IFISS problems. 
The first column lists the problem, the second column the pole strategy.
Columns 3 and 4 present the backward error on $A$ and $B$, columns 5 and
6 the average number of iterations and performance compared to QZ, and columns 7 and 8 the total number of swaps and the performance compared to QZ.} 
\scalebox{0.92}{
\input{tab/tab_numexp2b_n.tex}
}
\label{tab:numexp2b}
\end{table}


%% file: fig/fig_implRQZ.tikz
\begin{tikzpicture}[scale=1.66,y=-1cm]

\tikzset{bloplus/.style={draw,circle,append after command={
        [shorten >=\pgflinewidth, shorten <=\pgflinewidth,]
        (\tikzlastnode.north) edge[white,line width=1.0pt] (\tikzlastnode.south)
        (\tikzlastnode.east) edge[white,line width=1.0pt] (\tikzlastnode.west)
        }
    }
}

\tikzset{bloslash/.style={draw,circle,append after command={
        [shorten >=\pgflinewidth, shorten <=\pgflinewidth,]
        (\tikzlastnode.north east) edge[white,line width=1.0pt] (\tikzlastnode.south west)
        }
    }
}

\def\yos{0}
\def\xosA{0}
\node[] at (0.6+\xosA,1.4+\yos) {$A$};

\foreach \y in {1,...,5}{
	\foreach \x in {\y,...,5}{
		\node[] at (\x/5+\xosA,\y/5+\yos) {$\times$};
	}
};

\node[] at (0.2+\xosA,0.4+\yos) {\circled{1}};
\node[] at (0.4+\xosA,0.6+\yos) {\circled{2}};
\node[] at (0.6+\xosA,0.8+\yos) {\circled{3}};
\node[] at (0.8+\xosA,1+\yos) {\circled{4}};

\def\xosB{1.2}
\node[] at (0.6+\xosB,1.4+\yos) {$B$};

\foreach \y in {1,...,5}{
	\foreach \x in {\y,...,5}{
		\node[] at (\x/5+\xosB,\y/5+\yos) {$\times$};
	}
};

\node[] at (0.2+\xosB,0.4+\yos) {\circledletter{a}};
\node[] at (0.4+\xosB,0.6+\yos) {\circledletter{b}};
\node[] at (0.6+\xosB,0.8+\yos) {\circledletter{c}};
\node[] at (0.8+\xosB,1+\yos) {\circledletter{d}};

\draw[rounded corners=3pt] (-0.2+\xosA,-0.2+\yos) rectangle (\xosB+1.2,1.6+\yos);
\node[] at (\xosA,0+\yos) {I.};

\def\xosA{2.6}
\node[] at (0.6+\xosA,1.4+\yos) {$A \leftarrow Q_1^* A$};

\foreach \y in {1,2}{
	\foreach \x in {\y,...,5}{
		\node[] at (\x/5+\xosA,\y/5+\yos) {$\otimes$};
	}
}	

\foreach \y in {3,...,5}{
	\foreach \x in {\y,...,5}{
		\node[] at (\x/5+\xosA,\y/5+\yos) {$\times$};
	}
};

\node[] at (0.2+\xosA,0.4+\yos) {\scalebox{1.5}{$\oplus$}};
\node[] at (0.4+\xosA,0.6+\yos) {\circled{2}};
\node[] at (0.6+\xosA,0.8+\yos) {\circled{3}};
\node[] at (0.8+\xosA,1+\yos) {\circled{4}};

\def\xosB{3.8}
\node[] at (0.6+\xosB,1.4+\yos) {$B \leftarrow Q_1^* B$};

\foreach \y in {1,2}{
	\foreach \x in {\y,...,5}{
		\node[] at (\x/5+\xosB,\y/5+\yos) {$\otimes$};
	}
}	

\foreach \y in {3,...,5}{
	\foreach \x in {\y,...,5}{
		\node[] at (\x/5+\xosB,\y/5+\yos) {$\times$};
	}
};

\node[] at (0.2+\xosB,0.4+\yos) {\scalebox{1.5}{$\oslash$}};
\node[] at (0.4+\xosB,0.6+\yos) {\circledletter{b}};
\node[] at (0.6+\xosB,0.8+\yos) {\circledletter{c}};
\node[] at (0.8+\xosB,1+\yos) {\circledletter{d}};

\draw[rounded corners=3pt] (-0.2+\xosA,-0.2+\yos) rectangle (\xosB+1.2,1.6+\yos);
\node[] at (\xosA,0+\yos) {II.};

\def\xosA{5.2}
\node[] at (0.6+\xosA,1.4) {$A \leftarrow Q^* A Z$};

\draw[fill={rgb:red,66;green,124;blue,244}, fill opacity=0.5,rounded corners=3pt] (0.08+\xosA,0.28) rectangle (0.52+\xosA,0.72);

\node[] at (0.2+\xosA,0.2+\yos) {$\otimes$};
\node[] at (0.4+\xosA,0.4+\yos) {$\otimes$};
\node[] at (0.6+\xosA,0.2+\yos) {$\times$};
\node[] at (0.4+\xosA,0.2+\yos) {$\otimes$};
\foreach \y in {2,3}{
	\foreach \x in {3,...,5}{
		\node[] at (\x/5+\xosA,\y/5+\yos) {$\otimes$};
	}
};

\foreach \y in {1}{
	\foreach \x in {4,5}{
		\node[] at (\x/5+\xosA,\y/5+\yos) {$\times$};
	}
};
\foreach \y in {4,5}{
	\foreach \x in {\y,...,5}{
		\node[] at (\x/5+\xosA,\y/5+\yos) {$\times$};
	}
};

\node[] (l1) at (0.2+\xosA,0.4+\yos) {\blcircled{\textcolor{white}{\textbf{2}}}};
\node[bloplus,draw=black, fill=black] (l2) at (0.4+\xosA,0.6+\yos) {};
\node[] at (0.6+\xosA,0.8+\yos) {\circled{3}};
\node[] at (0.8+\xosA,1+\yos) {\circled{4}};

\path[<->] (l1) edge [min distance=1mm,in=180,out=180,looseness=2] (l2);

\def\xosB{6.4}
\node[] at (0.6+\xosB,1.4+\yos) {$B \leftarrow Q^* B Z$};

\draw[fill={rgb:red,66;green,124;blue,244}, fill opacity=0.5,rounded corners=3pt] (0.08+\xosB,0.28) rectangle (0.52+\xosB,0.72);

\node[] at (0.2+\xosB,0.2+\yos) {$\otimes$};
\node[] at (0.4+\xosB,0.4+\yos) {$\otimes$};
\node[] at (0.6+\xosB,0.2+\yos) {$\times$};
\node[] at (0.4+\xosB,0.2+\yos) {$\otimes$};
\foreach \y in {2,3}{
	\foreach \x in {3,...,5}{
		\node[] at (\x/5+\xosB,\y/5+\yos) {$\otimes$};
	}
};

\foreach \y in {1}{
	\foreach \x in {4,5}{
		\node[] at (\x/5+\xosB,\y/5+\yos) {$\times$};
	}
};
\foreach \y in {4,5}{
	\foreach \x in {\y,...,5}{
		\node[] at (\x/5+\xosB,\y/5+\yos) {$\times$};
	}
};

\node[] (r1) at (0.2+\xosB,0.4+\yos) {\blcircledletter{\textcolor{white}{\textbf{b}}}};
\node[bloslash,draw=black, fill=black] (r2) at (0.4+\xosB,0.6+\yos) {};
\node[] at (0.6+\xosB,0.8+\yos) {\circledletter{c}};
\node[] at (0.8+\xosB,1+\yos) {\circledletter{d}};

\path[<->] (r1) edge [min distance=1mm,in=180,out=180,looseness=2] (r2);

\draw[rounded corners=3pt] (-0.2+\xosA,-0.2+\yos) rectangle (\xosB+1.2,1.6+\yos);
\node[] at (\xosA,0+\yos) {III.};

\def\yos{1.8}
\def\xosA{0}
\node[] at (0.6+\xosA,1.4+\yos) {$A \leftarrow Q^* A Z$};

\draw[fill={rgb:red,66;green,124;blue,244}, fill opacity=0.5,rounded corners=3pt] (0.28+\xosA,0.48+\yos) rectangle (0.72+\xosA,0.92+\yos);

\node[] at (0.2+\xosA,0.2+\yos) {$\times$};
\node[] at (1+\xosA,1+\yos) {$\times$};

\foreach \y in {1,2}{
	\foreach \x in {4,5}{
		\node[] at (\x/5+\xosA,\y/5+\yos) {$\times$};
	}
};

\foreach \x in {2,3}{
	\foreach \y in {1,2}{
		\node[] at (\x/5+\xosA,\y/5+\yos) {$\otimes$};
	}
};

\foreach \y in {3,4}{
	\foreach \x in {\y,...,5}{
		\node[] at (\x/5+\xosA,\y/5+\yos) {$\otimes$};
	}
};

\node[] at (0.2+\xosA,0.4+\yos) {\circled{2}};
\node[] (l3) at (0.4+\xosA,0.6+\yos) {\blcircled{\textcolor{white}{\textbf{3}}}};
\node[bloplus,draw=black, fill=black] (l4) at (0.6+\xosA,0.8+\yos) {};
\node[] at (0.8+\xosA,1+\yos) {\circled{4}};

\path[<->] (l3) edge [min distance=1mm,in=180,out=180,looseness=2] (l4);

\def\xosB{1.2}
\node[] at (0.6+\xosB,1.4+\yos) {$B \leftarrow Q^* B Z$};

\draw[fill={rgb:red,66;green,124;blue,244}, fill opacity=0.5,rounded corners=3pt] (0.28+\xosB,0.48+\yos) rectangle (0.72+\xosB,0.92+\yos);

\node[] at (0.2+\xosB,0.2+\yos) {$\times$};
\node[] at (1+\xosB,1+\yos) {$\times$};

\foreach \y in {1,2}{
	\foreach \x in {4,5}{
		\node[] at (\x/5+\xosB,\y/5+\yos) {$\times$};
	}
};

\foreach \x in {2,3}{
	\foreach \y in {1,2}{
		\node[] at (\x/5+\xosB,\y/5+\yos) {$\otimes$};
	}
};

\foreach \y in {3,4}{
	\foreach \x in {\y,...,5}{
		\node[] at (\x/5+\xosB,\y/5+\yos) {$\otimes$};
	}
};

\node[] at (0.2+\xosB,0.4+\yos) {\circledletter{b}};
\node[] (r3) at (0.4+\xosB,0.6+\yos) {\blcircledletter{\textcolor{white}{\textbf{c}}}};
\node[bloslash,draw=black, fill=black] (r4) at (0.6+\xosB,0.8+\yos) {};
\node[] at (0.8+\xosB,1+\yos) {\circledletter{d}};

\path[<->] (r3) edge [min distance=1mm,in=180,out=180,looseness=2] (r4);

\draw[rounded corners=3pt] (-0.2+\xosA,-0.2+\yos) rectangle (\xosB+1.2,1.6+\yos);
\node[] at (\xosA,0+\yos) {IV.};

\def\xosA{2.6}
\node[] at (0.6+\xosA,1.4+\yos) {$A \leftarrow Q^* A Z$};

\draw[fill={rgb:red,66;green,124;blue,244}, fill opacity=0.5,rounded corners=3pt] (0.48+\xosA,0.68+\yos) rectangle (0.92+\xosA,1.12+\yos);

\node[] at (0.8+\xosA,0.8+\yos) {$\otimes$};
\node[] at (1+\xosA,0.8+\yos) {$\otimes$};
\node[] at (1+\xosA,1+\yos) {$\otimes$};

\foreach \x in {1,2}{
	\foreach \y in {1,...,\x}{
		\node[] at (\x/5+\xosA,\y/5+\yos) {$\times$};
	}
};

\foreach \y in {1,...,3}{
	\node[] at (1+\xosA,\y/5+\yos) {$\times$};
}

\foreach \y in {1,2,3}{
	\foreach \x in {3,4}{
		\node[] at (\x/5+\xosA,\y/5+\yos) {$\otimes$};
	}
};

\node[] at (0.2+\xosA,0.4+\yos) {\circled{2}};
\node[] at (0.4+\xosA,0.6+\yos) {\circled{3}};
\node[] (l5) at (0.6+\xosA,0.8+\yos) {\blcircled{\textcolor{white}{\textbf{4}}}};
\node[bloplus,draw=black, fill=black] (l6) at (0.8+\xosA,1+\yos) {};

\path[<->] (l5) edge [min distance=1mm,in=180,out=180,looseness=2] (l6);

\def\xosB{3.8}

\node[] at (0.6+\xosB,1.4+\yos) {$B \leftarrow Q^* B Z$};

\draw[fill={rgb:red,66;green,124;blue,244}, fill opacity=0.5,rounded corners=3pt] (0.48+\xosB,0.68+\yos) rectangle (0.92+\xosB,1.12+\yos);

\node[] at (0.8+\xosB,0.8+\yos) {$\otimes$};
\node[] at (1+\xosB,0.8+\yos) {$\otimes$};
\node[] at (1+\xosB,1+\yos) {$\otimes$};

\foreach \x in {1,2}{
	\foreach \y in {1,...,\x}{
		\node[] at (\x/5+\xosB,\y/5+\yos) {$\times$};
	}
};

\foreach \y in {1,...,3}{
	\node[] at (1+\xosB,\y/5+\yos) {$\times$};
}

\foreach \y in {1,2,3}{
	\foreach \x in {3,4}{
		\node[] at (\x/5+\xosB,\y/5+\yos) {$\otimes$};
	}
};

\node[] at (0.2+\xosB,0.4+\yos) {\circledletter{b}};
\node[] at (0.4+\xosB,0.6+\yos) {\circledletter{c}};
\node[] (r5) at (0.6+\xosB,0.8+\yos) {\blcircledletter{\textcolor{white}{\textbf{d}}}};
\node[bloslash,draw=black, fill=black] (r6) at (0.8+\xosB,1+\yos) {};

\path[<->] (r5) edge [min distance=1mm,in=180,out=180,looseness=2] (r6);

\draw[rounded corners=3pt] (-0.2+\xosA,-0.2+\yos) rectangle (\xosB+1.2,1.6+\yos);
\node[] at (\xosA,0+\yos) {V.};

\def\xosA{5.2}

\node[] at (0.6+\xosA,1.4+\yos) {$A \leftarrow A Z_4$};
\node[] at (1+\xosA,1+\yos) {$\otimes$};
\foreach \x in {1,...,3}{
	\foreach \y in {1,...,\x}{
		\node[] at (\x/5+\xosA,\y/5+\yos) {$\times$};
	}
};

\foreach \y in {1,...,4}{
	\foreach \x in {4,5}{
		\node[] at (\x/5+\xosA,\y/5+\yos) {$\otimes$};
	}
};

\node[] at (0.2+\xosA,0.4+\yos) {\circled{2}};
\node[] at (0.4+\xosA,0.6+\yos) {\circled{3}};
\node[] at (0.6+\xosA,0.8+\yos) {\circled{4}};
\node[] at (0.8+\xosA,1+\yos)   {\circled{5}};

\def\xosB{6.4}

\node[] at (0.6+\xosB,1.4+\yos) {$B \leftarrow B Z_4$};
\node[] at (1+\xosB,1+\yos) {$\otimes$};
\foreach \x in {1,...,3}{
	\foreach \y in {1,...,\x}{
		\node[] at (\x/5+\xosB,\y/5+\yos) {$\times$};
	}
};

\foreach \y in {1,...,4}{
	\foreach \x in {4,5}{
		\node[] at (\x/5+\xosB,\y/5+\yos) {$\otimes$};
	}
};

\node[] at (0.2+\xosB,0.4+\yos) {\circledletter{b}};
\node[] at (0.4+\xosB,0.6+\yos) {\circledletter{c}};
\node[] at (0.6+\xosB,0.8+\yos) {\circledletter{d}};
\node[] at (0.8+\xosB,1+\yos)   {\circledletter{e}};

\draw[rounded corners=3pt] (-0.2+\xosA,-0.2+\yos) rectangle (\xosB+1.2,1.6+\yos);
\node[] at (\xosA,0+\yos) {VI.};

\end{tikzpicture}

%% file: fig/numexp2/fig_numexp2_bwe_itpev.tikz
\begin{tikzpicture}
\pgfplotsset{
  log x ticks with fixed point/.style={
      xticklabel={
        \pgfkeys{/pgf/fpu=true}
        \pgfmathparse{exp(\tick)}%
        \pgfmathprintnumber[fixed relative, precision=3]{\pgfmathresult}
        \pgfkeys{/pgf/fpu=false}
      }
  }}
  
\begin{axis}[%
at={(0cm,0cm)},
width=4.5cm,
height=4.5cm,
ylabel shift = -0.05 cm,
xlabel shift = -0.15 cm,
scale only axis,
xmode=log,
xmin=100,
xmax=1000,
log x ticks with fixed point,
xtick={100, 178, 316, 562, 1000},
xlabel style={font=\color{white!15!black}},
xlabel={$n$},
ymode=log,
ymin=1.3e-15,
ymax=1.3e-14,
yminorticks=true,
ylabel style={font=\color{white!15!black}},
ylabel={relative backward error},
axis background/.style={fill=white},
every axis plot/.append style={thick},
grid=both,
grid style={line width=.1pt, draw=gray!10},
major grid style={line width=.2pt,draw=gray!50},
]
\addplot [color=black, loosely dashed]
  table[row sep=crcr]{%
100	1.69182623745406e-15\\
133	1.94708732061205e-15\\
178	2.30329538391719e-15\\
237	2.6336946740237e-15\\
316	2.98742501803995e-15\\
422	3.47826932939142e-15\\
562	3.93553297304772e-15\\
750	4.52398508853689e-15\\
1000	5.18686447532642e-15\\
};
\addplot [color=black, loosely dashed, mark=*, mark options={solid, scale=1.5, fill=gray}]
  table[row sep=crcr]{%
100	3.58091813625549e-15\\
133	4.075236954816e-15\\
178	4.79584385893824e-15\\
237	5.61278252564157e-15\\
316	6.53570093801717e-15\\
422	7.2181799871265e-15\\
562	8.07275499649437e-15\\
750	1.03458142858844e-14\\
1000	1.20897140646376e-14\\
};
\addplot [color=black, loosely dashed, mark=square*, mark options={solid, scale=1.5, fill=gray}]
  table[row sep=crcr]{%
100	3.56671001521559e-15\\
133	4.23227212349994e-15\\
178	5.08007940389804e-15\\
237	5.38507042845049e-15\\
316	6.54214113881744e-15\\
422	7.45680313556613e-15\\
562	8.33231710835346e-15\\
750	9.73315841886887e-15\\
1000	1.1495962380639e-14\\
};
\addplot [color=black, loosely dashed, mark=triangle*, mark options={solid, scale=1.5, fill=gray}]
  table[row sep=crcr]{%
100	3.66309554491701e-15\\
133	4.02055081638256e-15\\
178	4.54490992137233e-15\\
237	5.35782205933145e-15\\
316	6.0634212733972e-15\\
422	7.28518669378866e-15\\
562	7.85518130928602e-15\\
750	8.89514328372688e-15\\
1000	1.0172788125871e-14\\
};
\addplot [color=black, loosely dashed, mark=diamond*, mark options={solid, scale=1.5, fill=gray}]
  table[row sep=crcr]{%
100	3.68100177094935e-15\\
133	4.40267868108885e-15\\
178	5.0163030057257e-15\\
237	6.05307882980362e-15\\
316	7.07631787080278e-15\\
422	7.63427161497591e-15\\
562	9.02685028955085e-15\\
750	1.05487619610497e-14\\
1000	1.19345342191215e-14\\
};
\addplot [color=black]
  table[row sep=crcr]{%
100	2.18077089256397e-15\\
133	2.1398517261843e-15\\
178	2.57266676252797e-15\\
237	3.11739956824322e-15\\
316	3.40738933741685e-15\\
422	3.81913716150156e-15\\
562	4.41941560267269e-15\\
750	4.99231362476261e-15\\
1000	5.864203618555e-15\\
};

\addplot [color=black, mark=*, mark options={solid, scale=1.5, fill=gray}]
  table[row sep=crcr]{%
100	2.90447248087585e-15\\
133	3.59246758726589e-15\\
178	3.87394515851869e-15\\
237	4.47477162634507e-15\\
316	5.22777877349495e-15\\
422	5.9743048988425e-15\\
562	6.61001899612315e-15\\
750	7.7494029046912e-15\\
1000	8.89808106034398e-15\\
};

\addplot [color=black, mark=square*, mark options={solid, scale=1.5, fill=gray}]
  table[row sep=crcr]{%
100	3.38849972269807e-15\\
133	3.65503826649386e-15\\
178	4.26295412577156e-15\\
237	4.87438877346261e-15\\
316	5.55321082956008e-15\\
422	6.36693449211746e-15\\
562	7.19332801034147e-15\\
750	8.31976833097726e-15\\
1000	9.66169407307484e-15\\
};

\addplot [color=black, mark=triangle*, mark options={solid, scale=1.5, fill=gray}]
  table[row sep=crcr]{%
100	2.89187782061909e-15\\
133	3.35768604291459e-15\\
178	3.99818109553203e-15\\
237	4.43445225620473e-15\\
316	5.06905043977271e-15\\
422	6.02369157099309e-15\\
562	6.65794907961521e-15\\
750	7.58020628059224e-15\\
1000	8.70322305418245e-15\\
};

\addplot [color=black, mark=diamond*, mark options={solid, scale=1.5, fill=gray}]
  table[row sep=crcr]{%
100	3.20245461945953e-15\\
133	3.55907683643352e-15\\
178	4.05458243940221e-15\\
237	4.71792717872386e-15\\
316	5.39354334422341e-15\\
422	6.23820024279332e-15\\
562	7.05854761410872e-15\\
750	8.10451159121338e-15\\
1000	9.26156132445666e-15\\
};

\end{axis}
%

\begin{axis}[%
at={(6cm,0cm)},
width=4.5cm,
height=4.5cm,
ylabel shift = -0.15cm,
xlabel shift = -0.15cm,
scale only axis,
xmode=log,
xmin=100,
xmax=1000,
xminorticks=true,
log x ticks with fixed point,
xtick={100, 178, 316, 562, 1000},
xlabel style={font=\color{white!15!black}},
xlabel={$n$},
ymin=2.6,
ymax=3.2,
ylabel style={font=\color{white!15!black}},
ylabel={iterations$/n$},
axis background/.style={fill=white},
every axis plot/.append style={thick},
legend columns=-1,
legend entries={infinity;,zero;,Wilkinson;,random},
legend style={at={(8cm,-1cm)},anchor=north},
grid=both,
grid style={line width=.1pt, draw=gray!10},
major grid style={line width=.2pt,draw=gray!50},
]

\addplot [color=black, mark=*, mark options={solid, scale=1.5, fill=gray}]
  table[row sep=crcr]{%
100	2.775000000000000\\
133	2.765413533834586\\
178	2.764606741573034\\
237	2.748945147679325\\
316	2.740189873417721\\
422	2.717772511848342\\
562	2.713523131672598\\
750	2.715866666666667\\
1000 2.714000000000000\\
};

\addplot [color=black, mark=square*, mark options={solid, scale=1.5, fill=gray}]
  table[row sep=crcr]{%
100	3.088000000000000\\
133	3.065413533834586\\
178	3.014606741573034\\
237	2.970042194092827\\
316	2.928164556962025\\
422	2.896445497630332\\
562	2.895373665480427\\
750	2.864133333333333\\
1000 2.838300000000000\\
};

\addplot [color=black, mark=triangle*, mark options={solid, scale=1.5, fill=gray}]
  table[row sep=crcr]{%
100	2.733000000000000\\
133	2.724812030075188\\
178	2.716853932584270\\
237	2.715189873417721\\
316	2.705379746835443\\
422	2.681753554502370\\
562	2.673843416370107\\
750	2.667466666666666\\
1000 2.661400000000000\\
};

\addplot[color=black, mark=diamond*, mark options={solid, scale=1.5, fill=gray}]
  table[row sep=crcr]{%
100	3.108000000000000\\
133	3.075187969924812\\
178	3.071910112359551\\
237	3.049789029535865\\
316	3.007278481012658\\
422	2.980568720379147\\
562	2.974911032028470\\
750	2.957733333333334\\
1000 2.936400000000000\\
};

\end{axis}

\end{tikzpicture}%

%% file: fig/numexp2/fig_numexp2_swapspev2.tikz
\begin{tikzpicture}
\pgfplotsset{
  log x ticks with fixed point/.style={
      xticklabel={
        \pgfkeys{/pgf/fpu=true}
        \pgfmathparse{exp(\tick)}%
        \pgfmathprintnumber[fixed relative, precision=3]{\pgfmathresult}
        \pgfkeys{/pgf/fpu=false}
      }
  }}

\begin{axis}[%
at={(0cm,0cm)},
width=4.5cm,
height=4.5cm,
ylabel shift = -0.15cm,
xlabel shift = -0.15cm,
scale only axis,
xmode=log,
xmin=100,
xmax=1000,
xminorticks=true,
log x ticks with fixed point,
xtick={100, 178, 316, 562, 1000},
xlabel style={font=\color{white!15!black}},
xlabel={$n$},
ymin=1.1,
ymax=1.7,
ylabel style={font=\color{white!15!black}},
ylabel={swaps$/n^2$},
axis background/.style={fill=white},
every axis plot/.append style={thick},
legend columns=-1,
legend entries={infinity;,zero;,Wilkinson;,random},
legend style={at={(11.2cm,-1cm)},anchor=north},
grid=both,
grid style={line width=.1pt, draw=gray!10},
major grid style={line width=.2pt,draw=gray!50},
]

\addplot [color=black, mark=*, mark options={solid, scale=1.5, fill=gray}]
  table[row sep=crcr]{%
100	1.481220000000000\\
133	1.428277460568715\\
178	1.433581618482515\\
237	1.399462336876213\\
316	1.372813851946804\\
422	1.362165607241527\\
562	1.314437190511772\\
750	1.292152533333334\\
1000 1.257348400000000\\
};

\addplot [color=black, mark=square*, mark options={solid, scale=1.5, fill=gray}]
  table[row sep=crcr]{%
100	1.636180000000000\\
133	1.597597376900899\\
178	1.560380633758364\\
237	1.523066104078762\\
316	1.479214068258292\\
422	1.449217784865569\\
562	1.434297944554907\\
750	1.388628088888889\\
1000 1.360143000000000\\
};

\addplot [color=black, mark=triangle*, mark options={solid, scale=1.5, fill=gray}]
  table[row sep=crcr]{%
100	1.452070000000000\\
133 1.434128554468879\\
178	1.403758995076379\\
237	1.376097135430576\\
316	1.331296066335523\\
422	1.313348195233710\\
562	1.276001127138714\\
750	1.240970488888889\\
1000 1.206960000000000\\
};

\addplot[color=black, mark=diamond*, mark options={solid, scale=1.5, fill=gray}]
  table[row sep=crcr]{%
100	1.626670000000000\\
133	1.610701565944938\\
178	1.587091276353996\\
237	1.562753476116719\\
316	1.528432943438552\\
422	1.513609869499787\\
562	1.480282354580109\\
750	1.453626666666667\\
1000 1.433762000000000\\
};

\end{axis}  
  
\end{tikzpicture}

%% file: tab/tab_numexp2b_n.tex
\begin{tabular}{c | l | l | l | l | l | l | l}
Problem & pole & error $A$ & error $B$ & it/$n$ & $\%$ & swaps/$n^2$ & $\%$ \\ \hline
\multirow{2}{*}{\emph{Cavity flow}}    & $\infty$  & $7.5 \cdot 10^{-15}$ & $4.4 \cdot 10^{-15}$ & $2.49$ & $100$  & $0.446$  & $100$ \\
                                       & Wilk.     & $7.8 \cdot 10^{-15}$ & $4.1 \cdot 10^{-15}$ & $2.34$ & $94.2$ & $0.443$  & $99.3$ \\ \hline
\multirow{2}{*}{\emph{Obstacle flow}}  & $\infty$  & $9.2 \cdot 10^{-15}$ & $7.8 \cdot 10^{-15}$ & $2.54$ & $100$  & $0.617$  & $100$ \\
                                       & Wilk.     & $8.8 \cdot 10^{-15}$ & $7.8 \cdot 10^{-15}$ & $2.36$ & $93.0$ & $0.595$ & $96.3$\\ \hline                                     
\end{tabular}

%% file: sections/07_tightly_packed_shifts/tightly_packed_shifts_short.tex
\subsection{Tightly packed shifts.}
\label{sec:tightly_packed}

The single shifted RQZ method is, just like the classical QZ method sequential in nature
and not very cache efficient. To enhance cache performance one can go for
\emph{multishift} and chase $m$ shifts simultaneously or one can chase $m$ single shifts
as close as possible after each other. Since the theory in this manuscript is not suited for a multishift
setting we will confine ourselves to a description  and numerical experiment for \emph{tightly packed shifts}.

Assume we would like to chase $m$ tightly packed shifts, which are typically
the eigenvalues of the bottom-right $m{\times}m$ block of $(A,B)$.  
These shifts are introduced one after another in the Hessenberg pair. The first
shift is introduced and swapped down one row. Next the second shift is introduced and 
both shifts need to be swapped down a single row, starting with the lower-right one first.
As a result there is space to introduce the third shift, and the procedure continues. After having introduced
the shifts, the first $m$ subdiagonal elements of the pair $(A,B)$ link to these shifts.

In order to chase the block of $m$ shifts 
one needs to swap all shifts down one row, starting again with the one in lower-right corner first.
In total there are $m$ equivalence transformations which should be
accumulated to update the necessary parts of the matrices in a cache efficient
manner.  

The finalization phase commences when the shifts occupy the last subdiagonal positions 
in the Hessenberg pair. We can now introduce $m$ new poles.
The first new pole is introduced in the final subdiagonal element and swapped up $m$
positions thereby swapping all remaining shifts down. The second new pole is now
introduced and this course of action continues until the new poles occupy the last $m$
subdiagonal elements.

We test the tightly packed RQZ method on randomly generated matrix pairs of size $1000$ that are first
reduced to Hessenberg pairs with poles at infinity.
We run the  RQZ method
for shift batches of sizes $m = 2, 4, 8, 16, 32$.
The results are averaged over $2$ runs.
The poles are selected following three criteria: always at infinity (classical QZ), 
$m$ times the Wilkinson pole of the leading $2{\times}2$ block, or as
as the eigenvalues of the leading $m{\times}m$ block, the \emph{Rayleigh} poles.

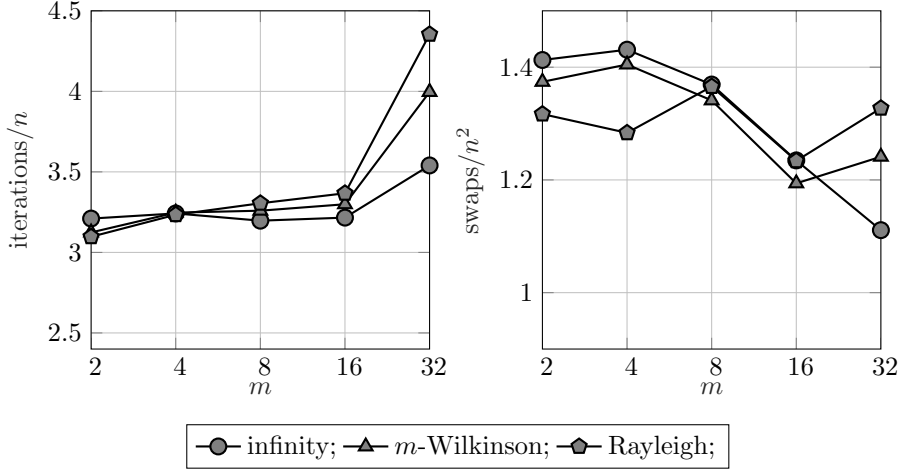
\begin{figure}[htp]
\centering
\input{fig/numexp3/fig_numexp3_1cluster.tikz}
\caption{On the left the average number of iterations per eigenvalue
is depicted in function of batch size $m$ for three different pole strategies.
On the right the average number of swaps scaled with $n^2$
in function of the batch size $m$.
These results are for the random problem.
}
\label{fig:numexp31}
\end{figure}

\Cref{fig:numexp31} 
displays the performance in terms of the average number of iterations
per eigenvalue (left) and total number of swaps scaled with $n^2$ (right)
in function of the batch size $m$ for the three types of poles.
We observe that the number of iterations remains constant up to a batch size of $16$ but increases significantly for $m=32$.
This effect is most pronounced
with the Wilkinson and Rayleigh poles. Also in terms of the number of
swaps the poles at infinity are the most efficient choice for $m=32$.
We attribute this effect to the spectrum of the randomly generated problems.
All, except typically one,  of the eigenvalues are located in one cluster around zero. 
Likely, due to the increased batch size, 
some of the Wilkinson and Rayleigh poles will somehow be too close to each other, thereby
deteriorating the convergence.


Therefore, we have repeated this experiment with randomly generated matrix pairs of size
$1000$ having two equally sized clusters of eigenvalues centered around $0$ and $10$. The
results are shown in \Cref{fig:numexp32}.  Now the
Wilkinson and Rayleigh poles outperform the poles at infinity in terms of total number of
swaps for all batch sizes.

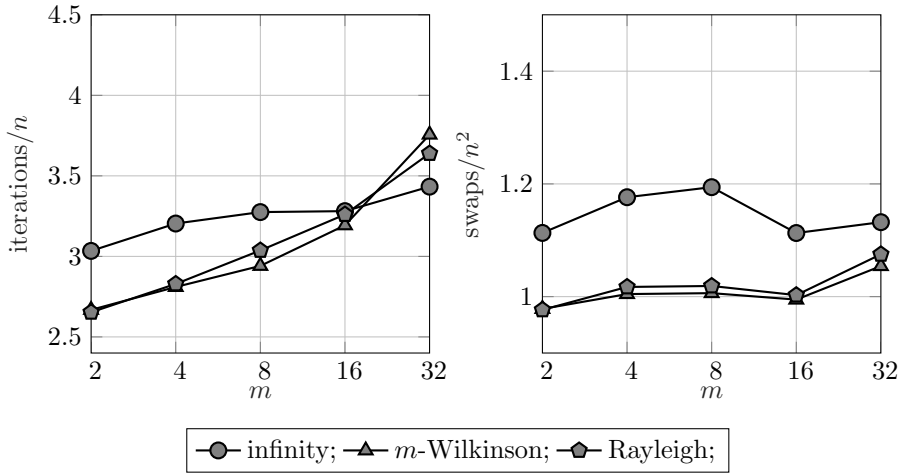
\begin{figure}[htp]
\centering
\input{fig/numexp3/fig_numexp3_2clusters.tikz}
\caption{On the left the average number of iterations per eigenvalue
is depicted in function of batch size $m$ for three different pole strategies.
On the right the average number of swaps scaled with $n^2$
in function of the batch size $m$.
These are the results for the random problems with two clusters.
}
\label{fig:numexp32}
\end{figure}

We conclude that we can pack the shifts tightly without a significant degradation in
convergence behavior. The advantages of allowing pole selection remain but become more problem specific.
A cache efficient implementation as well as a good criterion to pick the poles is, however, future work.

%% file: fig/numexp3/fig_numexp3_1cluster.tikz
\begin{tikzpicture}

\pgfplotsset{
  log x ticks with fixed point/.style={
      xticklabel={
        \pgfkeys{/pgf/fpu=true}
        \pgfmathparse{exp(\tick)}%
        \pgfmathprintnumber[fixed relative, precision=3]{\pgfmathresult}
        \pgfkeys{/pgf/fpu=false}
      }
  }}

\begin{axis}[%
at={(0cm,0cm)},
width=4.5cm,
height=4.5cm,
ylabel shift = -0.05 cm,
xlabel shift = -0.15 cm,
scale only axis,
xmode=log,
xmin=2,
xmax=32,
log x ticks with fixed point,
xtick={2, 4, 8, 16, 32},
xlabel style={font=\color{white!15!black}},
xlabel={$m$},
ymin=2.4,
ymax=4.5,
yminorticks=true,
ylabel style={font=\color{white!15!black}},
ylabel={iterations$/n$},
axis background/.style={fill=white},
every axis plot/.append style={thick},
grid=both,
grid style={line width=.1pt, draw=gray!10},
major grid style={line width=.2pt,draw=gray!50},
]

\addplot [color=black, mark=*, mark options={solid, scale=1.5, fill=gray}]
  table[row sep=crcr]{%
32	3.54\\
16	3.216\\
8	3.197\\
4	3.242\\
2	3.21\\
};

\addplot [color=black, mark=triangle*, mark options={solid, scale=1.5, fill=gray}]
  table[row sep=crcr]{%
32	3.996\\
16	3.299\\
8	3.259\\
4	3.247\\
2	3.124\\
};

\addplot [color=black, mark=pentagon*, mark options={solid, scale=1.5, fill=gray}]
  table[row sep=crcr]{%
32	4.354\\
16	3.367\\
8	3.305\\
4	3.234\\
2	3.097\\
};

\end{axis}

\begin{axis}[%
at={(6cm,0cm)},
width=4.5cm,
height=4.5cm,
ylabel shift = -0.05 cm,
xlabel shift = -0.15 cm,
scale only axis,
xmode=log,
xmin=2,
xmax=32,
log x ticks with fixed point,
xtick={2, 4, 8, 16, 32},
xlabel style={font=\color{white!15!black}},
xlabel={$m$},
ymin=0.9,
ymax=1.5,
yminorticks=true,
ylabel style={font=\color{white!15!black}},
ylabel={swaps$/n^2$},
axis background/.style={fill=white},
every axis plot/.append style={thick},
legend columns=-1,
legend entries={infinity;,$m$-Wilkinson;,Rayleigh;},
legend style={at={(0cm,-1cm)},anchor=north},
grid=both,
grid style={line width=.1pt, draw=gray!10},
major grid style={line width=.2pt,draw=gray!50},
]

\addplot [color=black, mark=*, mark options={solid, scale=1.5, fill=gray}]
  table[row sep=crcr]{%
32	1.110877\\
16	1.234656\\
8	1.369201\\
4	1.431037\\
2	1.41275\\
};

\addplot [color=black, mark=triangle*, mark options={solid, scale=1.5, fill=gray}]
  table[row sep=crcr]{%
32	1.241158\\
16	1.193942\\
8	1.341107\\
4	1.404774\\
2	1.374006\\
};

\addplot [color=black, mark=pentagon*, mark options={solid, scale=1.5, fill=gray}]
  table[row sep=crcr]{%
32	1.327183\\
16	1.233916\\
8	1.365697\\
4	1.283523\\
2	1.316812\\
};

\end{axis}

\end{tikzpicture}

%% file: fig/numexp3/fig_numexp3_2clusters.tikz
\begin{tikzpicture}

\pgfplotsset{
  log x ticks with fixed point/.style={
      xticklabel={
        \pgfkeys{/pgf/fpu=true}
        \pgfmathparse{exp(\tick)}%
        \pgfmathprintnumber[fixed relative, precision=3]{\pgfmathresult}
        \pgfkeys{/pgf/fpu=false}
      }
  }}

\begin{axis}[%
at={(0cm,0cm)},
width=4.5cm,
height=4.5cm,
ylabel shift = -0.05 cm,
xlabel shift = -0.15 cm,
scale only axis,
xmode=log,
xmin=2,
xmax=32,
log x ticks with fixed point,
xtick={2, 4, 8, 16, 32},
xlabel style={font=\color{white!15!black}},
xlabel={$m$},
ymin=2.4,
ymax=4.5,
yminorticks=true,
ylabel style={font=\color{white!15!black}},
ylabel={iterations$/n$},
axis background/.style={fill=white},
every axis plot/.append style={thick},
grid=both,
grid style={line width=.1pt, draw=gray!10},
major grid style={line width=.2pt,draw=gray!50},
]

\addplot [color=black, mark=*, mark options={solid, scale=1.5, fill=gray}]
  table[row sep=crcr]{%
32	3.433\\
16	3.281\\
8	3.275\\
4	3.204\\
2	3.034\\
};

\addplot [color=black, mark=triangle*, mark options={solid, scale=1.5, fill=gray}]
  table[row sep=crcr]{%
32	3.756\\
16	3.193\\
8	2.941\\
4	2.81\\
2	2.668\\
};

\addplot [color=black, mark=pentagon*, mark options={solid, scale=1.5, fill=gray}]
  table[row sep=crcr]{%
32	3.639\\
16	3.26\\
8	3.037\\
4	2.829\\
2	2.653\\
};

\end{axis}

\begin{axis}[%
at={(6cm,0cm)},
width=4.5cm,
height=4.5cm,
ylabel shift = -0.05 cm,
xlabel shift = -0.15 cm,
scale only axis,
xmode=log,
xmin=2,
xmax=32,
log x ticks with fixed point,
xtick={2, 4, 8, 16, 32},
xlabel style={font=\color{white!15!black}},
xlabel={$m$},
ymin=0.9,
ymax=1.5,
yminorticks=true,
ylabel style={font=\color{white!15!black}},
ylabel={swaps$/n^2$},
axis background/.style={fill=white},
every axis plot/.append style={thick},
legend columns=-1,
legend entries={infinity;,$m$-Wilkinson;,Rayleigh;},
legend style={at={(0cm,-1cm)},anchor=north},
grid=both,
grid style={line width=.1pt, draw=gray!10},
major grid style={line width=.2pt,draw=gray!50},
]

\addplot [color=black, mark=*, mark options={solid, scale=1.5, fill=gray}]
  table[row sep=crcr]{%
32	1.132111\\
16	1.112874\\
8	1.194212\\
4	1.176379\\
2	1.113045\\
};

\addplot [color=black, mark=triangle*, mark options={solid, scale=1.5, fill=gray}]
  table[row sep=crcr]{%
32	1.054211\\
16	0.994758\\
8	1.006218\\
4	1.004564\\
2	0.978231\\
};

\addplot [color=black, mark=pentagon*, mark options={solid, scale=1.5, fill=gray}]
  table[row sep=crcr]{%
32	1.075103\\
16	1.002701\\
8	1.018934\\
4	1.01721\\
2	0.976562\\
};

\end{axis}
\end{tikzpicture}

%% file: sections/05_implicitq/implicitq.tex
\section{Implicit Q theorem.}
\label{sec:implQthm}

In this section we prove the following implicit Q theorem for proper Hessenberg pairs justifying
the implicit approach since the result of a rational QZ step is uniquely determined. 

\begin{theorem}[Implicit Q theorem for proper Hessenberg pairs]
\label{thm:implQ}
Let $(A,B)$ 
be a regular matrix pair and let $\hat{Q}$, $\check{Q}$, $\hat{Z}$, $\check{Z}$
be unitary matrices with
$\hat{Q}\bm{e}_1 = \sigma \check{Q}\bm{e}_1$, $|\sigma|=1$,
such that,
\begin{align*}
(\hat{A}, \hat{B}) = \hat{Q}^* (A,B) \hat{Z} \quad \text{and} 
\quad (\check{A}, \check{B}) = \check{Q}^* (A,B) \check{Z},
\end{align*}
are both proper Hessenberg pairs 
having both the same pole tuple $\Xi = ( \xi_1, \hdots,
\xi_{n-1} )$, $\xi_i \in \bar{\mathbb{C}}$, 
with the poles different from the spectrum of the pair.

Then the pairs
$(\hat{A},\hat{B})$ and $(\check{A}, \check{B})$
are \emph{essentially identical}, meaning that,
\begin{equation}
\label{eq:essential_uniq}
\hat{A} = D_1^* \check{A} D_2 \quad \text{and} \quad \hat{B} = D_1^* \check{B} D_2,
\end{equation}
with $D_1$ and $D_2$ unitary diagonal matrices.
\end{theorem}

The implicit Q theorem guarantees that the unitary equivalence transformations, which are
implicitly applied in the direct reduction to a Hessenberg pair and in a rational QZ step
are essentially unique. Once the reduction or the rational QZ step is initiated, the
outcome is determined.

The remainder of this section contains all ingredients to prove this theorem. 
Various related implicit Q theorems already exist. Mastronardi, Vandebril, and 
Van Barel \cite{b164} provide one for semiseparable plus diagonal matrices
linked to rational Krylov spaces.
Pranic, Mach, and Vandebril \cite{Mach2014} formulate a variant for extended Hessenberg
plus diagonal matrices linked to general rational Krylov subspaces as did 
Berljafa and G\"uttel \cite{Berljafa2015}  for (rectangular) Hessenberg pairs.

The proof we provide here significantly differs from the one by Berljafa and G\"uttel,
who rely on direct computations and utilize the invertibility of $B$ to formulate the theory for
the single matrix setting. We make use of
the properties of the associated Krylov matrices, as done by Watkins for the classical
case \cite{b333}; this allows us to easily prove that the rational QZ algorithm performs
nested subspace iteration driven by a rational function. Also no constraints on the
invertibility of $B$ are imposed.
%

\subsection{Rational Krylov matrices and subspaces.}
\label{ssec:ratkrylmat}

We define \emph{rational Krylov matrices} generated by a matrix pair $(A,B)$, a vector
$\bm{v}$, and a driving rational function determined by shifts $\Rho$ and poles $\Xi$.
These rational Krylov matrices span Krylov subspaces, which, for consistency, we will name
\emph{rational Krylov subspaces}. The description holds for regular matrix pairs, so
the matrices do not need to be of Hessenberg form.

For the aim of a concise notation we introduce two
elementary rational functions of a pair $(A,B)$ with shift $\varrho = \mu/\nu \in \bar{\CC}$ and pole 
$\xi = \alpha/\beta \in \bar{\CC}$:
\begin{equation}
\label{eq:elementary_rationals}
\begin{split}
M(\varrho,\xi) & = \LinOp{\mu}{\nu} \LinOpInv{\alpha}{\beta}, \\
N(\varrho,\xi) & = \LinOpInv{\alpha}{\beta} \LinOp{\mu}{\nu}.
\end{split}
\end{equation}
We assume, throughout the remainder of the text, the shift different from the pole $\varrho
\neq \xi$ and since we take inverses, the pole may not be an eigenvalue $\xi \notin
\Lambda$.
Note that $N(\varrho,\xi)$ and $M(\varrho,\xi)$ represent an entire class of matrices generated by parameters
that result in the correct shift and pole.
These are all scalar multiples of one another and as the theory remains scale invariant,
every representative is fine.

We remark that in case $B$ were invertible, which we do not assume in the remainder of the
text, that the following relations hold,
\begin{equation}
\label{eq:elementary_rationals_I}
\begin{split}
M(\varrho,\xi) & = (\nu AB^{-1}- \mu I) (\beta AB^{-1} - \alpha I)^{-1}, \\
N(\varrho,\xi) & = (\beta B^{-1}A- \alpha I)^{-1} (\nu B^{-1}A - \mu I). 
\end{split}
\end{equation}
This could be helpful to link this analysis to existing theorems
of Berljafa \& G\"uttel \cite{Berljafa2015}, and Watkins \cite{b333}.

The elementary rational functions are used to define rational Krylov matrices.
\begin{definition}[rational Krylov matrices]
\label{def:ratkrylmat}
Let $(A,B) \in \mathbb{C}^{n \times n}$ be a regular matrix pair,
$\bm{v} \in \mathbb{C}^n$ a nonzero vector, $\Xi = ( \xi_1, \hdots, \xi_{k-1}
)$, $\xi_i \in \bar{\CC}$, the pole tuple with the poles different from the spectrum,
and $P = ( \varrho_1, \hdots, \varrho_{k-1} )$, $\varrho_i \in \bar{\CC}$, 
the tuple of shifts distinct from the poles, with $k\leq n$.
The corresponding rational Krylov matrices are defined as: 
\begin{equation}
\label{eq:ratkrylmat}
\begin{split}
K^{\text{rat}}_{k}(A,B,\bm{v}, \Xi, P) & = 
\left[ \bm{v}, 
M(\varrho_1,\xi_1) \bm{v}, 
M(\varrho_2,\xi_2) M(\varrho_1,\xi_1) \bm{v},
\, \hdots, 
\prod_{i=1}^{k-1}M(\varrho_i,\xi_i) \bm{v} \right], \\
L^{\text{rat}}_{k}(A,B,\bm{v}, \Xi, P) & = 
\left[ \bm{v}, 
N(\varrho_1,\xi_1)  \bm{v}, 
N(\varrho_2,\xi_2) N(\varrho_1,\xi_1) \bm{v},
\, \hdots, 
\prod_{i=1}^{k-1} N(\varrho_i,\xi_i) \bm{v} \right].
\end{split}
\end{equation}
\end{definition}

The following properties of the elementary rational functions are frequently
used in the remainder of the text.
\begin{lemma}
\label{lemma:prop_elementary_rationals}
The elementary rational functions \eqref{eq:elementary_rationals} satisfy: 
\begin{enumerate}[I.]
\item Commutativity: For shifts $\varrho, \check{\varrho}$ different from the poles
$\xi, \check{\xi}$,
\begin{equation}
\label{eq:commutativity_rationals}
\begin{split}
M(\varrho,\xi) \; M(\check{\varrho},\check{\xi}) & = M(\check{\varrho},\check{\xi}) \; M(\varrho,\xi), \\
N(\varrho,\xi) \; N(\check{\varrho},\check{\xi}) & = N(\check{\varrho},\check{\xi}) \; N(\varrho,\xi).
\end{split}
\end{equation}

\item Inverse: If the shift is not an eigenvalue $\varrho \notin \Lambda$ and different
from the pole $\varrho \neq \xi$ then,
\begin{equation}
\label{eq:inverse_rationals}
\begin{split}
M(\varrho,\xi)^{-1} & = M(\xi,\varrho), \\
N(\varrho,\xi)^{-1} & = N(\xi,\varrho).
\end{split}
\end{equation}

\item Shift invariance: For any nonzero vector $\bm{v} \in \CC^n$, 
and parameters $\varrho, \check{\varrho} \neq \xi$,
\begin{equation}
\label{eq:shift_invariance_rationals}
\begin{split}
\mathcal{R}(\bm{v}, \; M(\varrho,\xi) \bm{v}) & = \mathcal{R}(\bm{v}, \; M(\check{\varrho},\xi) \bm{v}), \\
\mathcal{R}(\bm{v}, \; N(\varrho,\xi) \bm{v}) & = \mathcal{R}(\bm{v}, \; N(\check{\varrho},\xi) \bm{v}).
\end{split}
\end{equation}

\item Nested shift invariance: For any nonzero vector $\bm{v} \in \CC^n$,
all shifts $\varrho_i$ different from all poles $\xi_j$ for $i,j$ from $1$ to
$k{-}1$, and an alternative shift $\check{\varrho} \notin \Xi$,
$k \leq n$,
\begin{equation}
\label{eq:nested_shift_invariance_rationals}
\begin{split}
\mathcal{R}\left(\bm{v}, M(\varrho_1,\xi_1) \bm{v}, \hdots, \prod_{i=1}^{k-1}M(\varrho_i,\xi_i) \bm{v} \right) 
& = \mathcal{R}\left(\bm{v}, M(\check{\varrho},\xi_1) \bm{v}, \hdots, \prod_{i=1}^{k-1}M(\check{\varrho},\xi_i) \bm{v}\right), \\
\mathcal{R}\left(\bm{v}, N(\varrho_1,\xi_1) \bm{v}, \hdots, \prod_{i=1}^{k-1}N(\varrho_i,\xi_i) \bm{v} \right) 
& = \mathcal{R}\left(\bm{v}, N(\check{\varrho},\xi_1) \bm{v}, \hdots, \prod_{i=1}^{k-1}N(\check{\varrho},\xi_i) \bm{v}\right).
\end{split}
\end{equation}
\end{enumerate}
\end{lemma}

\begin{proof}
If $B$ is invertible, property I of the Lemma follows from
Equation \eqref{eq:elementary_rationals_I} and the property
that any matrix commutes with its shifted inverse. For
singular $B$ the same result follows from an elementary
continuity argument.
Property II is trivial.
Property III follows directly from:
\begin{equation}
\label{eq:alternativeMN}
\begin{split}
M(\varrho,\xi) = I + (\xi - \varrho) B (A - \xi B)^{-1}, \\
N(\varrho,\xi) = I + (\xi - \varrho) (A - \xi B)^{-1} B,
\end{split}
\end{equation}
in case $\varrho \neq \infty$, $\xi \neq \infty$. It is clear that both $\mathcal{R}(\bm{v}, M(\varrho,\xi) \bm{v})$
and $\mathcal{R}(\bm{v}, N(\varrho,\xi) \bm{v})$ are independent of $\varrho$.
Similar expressions hold in case either $\varrho_i = \infty$ or $\xi = \infty$.
Property IV is proven by induction.
The base case $k{=}2$ is equal to the shift invariance property of 
Equation~\eqref{eq:shift_invariance_rationals}.
Now assume the property holds up to index $k$, denote this subspace as $\mathcal{U}_{k}$.
We assume for the induction step that the subspace is of full dimension.
If it becomes an invariant subspace,
the subspace is no longer expanded by adding vectors of the type 
$M(\varrho,\xi_k)\hat{\bm{v}}$ with
$\hat{\bm{v}}$ a vector in the subspace. By the induction hypothesis,
the property holds in this case.
The subspace of dimension $k{+}1$ is equal to:
\begin{equation*}
\mathcal{U}_{k+1}
= \mathcal{R}\left(\bm{v}, \, M(\varrho_1,\xi_1) \bm{v}, \hdots, \, \prod_{i=1}^{k}M(\varrho_i,\xi_i) \bm{v} \right)
= \mathcal{U}_{k}
+ \mathcal{R}\left(\prod_{i=1}^{k}M(\varrho_i,\xi_i) \bm{v} \right).
\end{equation*}
By the induction hypothesis, the result holds for $\mathcal{U}_{k}$.
We now modify the additional term in the subspace $\mathcal{U}_{k+1}$
to prove the result:
\begin{equation*}
\begin{split}
\mathcal{U}_{k+1}
& = \mathcal{U}_{k}
+ \mathcal{R}\left(M(\varrho_k,\xi_k) \prod_{i=1}^{k-1}M(\varrho_i,\xi_i) \bm{v} \right)
= \mathcal{U}_{k}
+ \mathcal{R}\left(M(\check{\varrho},\xi_k) \prod_{i=1}^{k-1}M(\varrho_i,\xi_i) \bm{v} \right) \\
& = \mathcal{U}_{k}
+ M(\check{\varrho},\xi_k) \, \mathcal{R}\left(\prod_{i=1}^{k-1}M(\check{\varrho},\xi_i) \bm{v} \right)
= \mathcal{R}\left(\bm{v}, \, M(\check{\varrho},\xi_1) \bm{v}, \hdots, \, \prod_{i=1}^{k}M(\check{\varrho},\xi_i) \bm{v}\right).
\end{split}
\end{equation*}
In the first equality, the $k$th term is extracted.
In the second equality the shift $\varrho_k$ is changed to $\check{\varrho}$
based on the shift invariance property of Equation~\eqref{eq:shift_invariance_rationals};
this is permitted by the fact that $\prod_{i=1}^{k-1}M(\varrho_i,\xi_i) \bm{v}$ is a vector
in $\mathcal{U}_{k}$.
The third equality extracts the $k$th term and has changed the other $k{-}1$ shifts to
$\check{\varrho}$ based on the induction hypothesis and the nestedness of the involved
subspaces.
The last equality is immediate.
The second result of property IV can be proven with the same reasoning.
\end{proof}

We can now define the \emph{rational Krylov subspaces} as the column spaces of the rational
Krylov matrices from \Cref{def:ratkrylmat}. It follows directly from the nested shift
invariance property of \Cref{lemma:prop_elementary_rationals} that these subspaces are
independent of the choice of $\Rho$.

\begin{definition}[rational Krylov subspaces]
\label{def:rksubspace}
We define the rational Krylov subspaces $\mathcal{K}_{k}^{\text{rat}}$ and
$\mathcal{L}_{k}^{\text{rat}}$, $k\leq n$, associated with the regular pair $(A,B) \in \CC^{n{\times}n}$, 
a vector $\bm{v} \in \CC^n$, and pole tuple $\Xi = ( \xi_1, \hdots, \xi_{k-1} )$, assuming 
the poles different from the eigenvalues as,
\begin{equation}
\label{eq:rksubspace_definition}
\begin{split}
\mathcal{K}_{k}^{\text{rat}}(A,B,\bm{v},\Xi) & = \mathcal{R}(K_{k}^{\text{rat}}(A,B,\bm{v}, \Xi, \Rho)), \\
\mathcal{L}_{k}^{\text{rat}}(A,B,\bm{v},\Xi) & = \mathcal{R}(L_{k}^{\text{rat}}(A,B,\bm{v}, \Xi, \Rho)),
\end{split}
\end{equation}
where the shift tuple $P$ is freely chosen, assuming all shifts different
from all poles.
\end{definition}

The two rational Krylov subspaces reduce to the same subspace if 
$B$ is the identity matrix
which is in agreement with earlier definitions. 
The rational Krylov subspaces satisfy the following elementary
properties.

\begin{lemma}[properties of rational Krylov subspaces]
\label{lemma:prop_rksubspace}
The rational Krylov subspaces $\mathcal{K}^{\text{rat}}$ and $\mathcal{L}^{\text{rat}}$ generated
from $(A,B) \in \CC^{n{\times}n}$, $\bm{v} \in \CC^n$, and 
$\Xi=( \xi_1, \hdots, \xi_{n-1} )$, assuming all poles different from the eigenvalues,
satisfy the following properties.
\begin{enumerate}[I.]
\item{They form a sequence of nested subspaces:
\begin{equation}
\label{eq:nested_rksubspace}
\mathcal{K}^{\text{rat}}_{1} \subseteq \mathcal{K}^{\text{rat}}_{2} \subseteq \hdots \subseteq \mathcal{K}^{\text{rat}}_{n}
\quad \text{and} \quad
\mathcal{L}^{\text{rat}}_{1} \subseteq \mathcal{L}^{\text{rat}}_{2} \subseteq \hdots \subseteq \mathcal{L}^{\text{rat}}_{n}.
\end{equation}
}
\item{
For $k=1, \hdots, n{-}1$, 
with the shift
$\hat{\varrho}$ different from all eigenvalues and poles,
and an alternative shift 
$\check{\varrho} \neq \hat{\varrho}$
we get:
\begin{equation}
\label{eq:rk_to_krylov}
\begin{split}
\mathcal{K}_{k}^{\text{rat}}(A,B,\bm{v},\Xi_{1:k-1}) & 
= \prod_{i=1}^{k-1} M(\hat{\varrho},\xi_i) \
\mathcal{K}_k(M(\check{\varrho},\hat{\varrho}),\bm{v}) \\ 
& =  \mathcal{K}_k \left(M(\check{\varrho},\hat{\varrho}), \prod_{i=1}^{k-1}
M(\hat{\varrho},\xi_i)\ \bm{v}\right), \\ 
\mathcal{L}_{k}^{\text{rat}}(A,B,\bm{v},\Xi_{1:k-1}) & 
= \prod_{i=1}^{k-1} N(\hat{\varrho},\xi_i) \
\mathcal{K}_k(N(\check{\varrho},\hat{\varrho}),\bm{v}) \\
& = \mathcal{K}_k\left(N(\check{\varrho},\hat{\varrho}),\prod_{i=1}^{k-1} N(\hat{\varrho},\xi_i)\ \bm{v}\right),
\end{split}
\end{equation}
which connects rational Krylov subspaces with regular Krylov subspaces.
}
\item{
For $k=1, \hdots, n{-}1$, and $\varrho_k \notin \Xi_{1:k-1}$,
\begin{equation}
\label{eq:increase_rksubspace}
\begin{split}
M(\varrho_k,\xi_k) \ \mathcal{K}_{k}^{\text{rat}}(A,B,\bm{v},\Xi_{1:k-1}) & \subseteq \mathcal{K}_{k+1}^{\text{rat}}(A,B,\bm{v},\Xi_{1:k}), \\
N(\varrho_k,\xi_k)\ \mathcal{L}_{k}^{\text{rat}}(A,B,\bm{v},\Xi_{1:k-1})  & \subseteq \mathcal{L}_{k+1}^{\text{rat}}(A,B,\bm{v},\Xi_{1:k}).
\end{split}
\end{equation}
}
\end{enumerate}
\end{lemma}
\begin{proof}

The nestedness follows directly from the definition.
To prove the second property we rely on
\Cref{lemma:prop_elementary_rationals},
\begin{equation*}
\begin{split}
\mathcal{K}_k^{\text{rat}}(A,B,\bm{v},\Xi_{1:k-1})  
& = \mathcal{R}\left(\bm{v}, M(\hat{\varrho},\xi_1)\,\bm{v}, \hdots, \prod_{i=1}^{k-1} M(\hat{\varrho},\xi_i) \, \bm{v}\right) \\
& = \prod_{i=1}^{k-1} M(\hat{\varrho},\xi_i) \; \mathcal{R}\left( \prod_{i=1}^{k-1} M(\xi_i,\hat{\varrho}) \, \bm{v},\prod_{i=2}^{k-1} M(\xi_i,\hat{\varrho}) \, \bm{v}, \hdots, \, \bm{v}\right) \\
& = \prod_{i=1}^{k-1} M(\hat{\varrho},\xi_i) \; \mathcal{R}\left( \prod_{i=1}^{k-1} M(\check{\varrho},\hat{\varrho}) \, \bm{v},\prod_{i=2}^{k-1} M(\check{\varrho},\hat{\varrho}) \, \bm{v}, \hdots, \, \bm{v}\right) \\
& = \prod_{i=1}^{k-1} M(\hat{\varrho},\xi_i) \; \mathcal{K}_k( M(\check{\varrho},\hat{\varrho}), \bm{v}).
\end{split}
\end{equation*}
The first equality is the definition with $\Rho = ( \hat{\varrho}, \hdots, \hat{\varrho} )$. The second
equality extracts the last rational term. 
The third equality applies the nested shift invariance property of \Cref{lemma:prop_elementary_rationals} 
to change all shifts $\xi_i$ to $\check{\varrho}$.
We end up with a Krylov subspace in the last
equality. The result for $\mathcal{L}^{\text{rat}}$ is proven in a similar way.
The third property follows from the second property and the nestedness of Krylov
subspaces, setting $\hat{\varrho}=\varrho_k$.
\end{proof}

We remark that item~II states that rational Krylov subspaces are nothing else
than Krylov subspaces whose starting vector is modified by a rational function determined
by the poles $\Xi$. 

\subsection{Proper Hessenberg pairs and rational Krylov.}
\label{ssec:hesspairratkryl}

In the previous section $(A,B)$ could be any regular pair.
Now we'll see that if $(A,B)$ is a proper Hessenberg pair, the rational Krylov subspaces
and matrices have a special structure.

\begin{theorem}
\label{thm:rksubspace_structure}
Let $(A,B) \in \mathbb{C}^{n \times n}$ be a proper Hessenberg pair having poles $\Xi =
( \xi_1, \hdots, \xi_{n-1} )$ distinct from the eigenvalues,
Then for $k$ from $1$ to $n$,
\begin{equation}
\label{eq:rksubspace_K_structure}
\mathcal{K}_{k}^{\text{rat}}(A,B,\bm{e}_1,( \xi_1, \hdots, \xi_{k-1} )) =  \mathcal{E}_k,
\end{equation}
while for $k$ from $1$ to $n{-}1$,
\begin{equation}
\label{eq:rksubspace_L_structure}
 \mathcal{L}_{k}^{\text{rat}}(A,B,\bm{e}_1, (\xi_2, \hdots, \xi_k)) = \mathcal{E}_k.
\end{equation}
\end{theorem}

\begin{proof}
We prove the results by induction on the subspace dimension.
The case $k=1$ is trivial for both statements.
To prove Equation~\eqref{eq:rksubspace_K_structure},
assume the result holds up to dimension $k \leq n{-}1$,
\begin{equation*}
\mathcal{K}_{k}^{\text{rat}}(A,B,\bm{e}_1,( \xi_1, \hdots, \xi_{k-1} ) ) =  \mathcal{E}_{k}.
\end{equation*} 
From the nestedness of rational Krylov subspaces, we have by induction,
\begin{equation*}
\mathcal{E}_k \subseteq \mathcal{K}_{k+1}^{\text{rat}}(A,B,\bm{e}_1, ( \xi_1, \hdots, \xi_{k})) .
\end{equation*}
It remains to be shown that
$\bm{e}_{k+1} \in \mathcal{K}_{k+1}^{\text{rat}}(A,B,\bm{e}_1,( \xi_1, \hdots, \xi_{k}))$.
From Equation~\eqref{eq:increase_rksubspace} and the induction hypothesis we deduce,
\begin{equation}
\label{eq:ME_subset}
M(\varrho_k,\xi_k) \ \mathcal{E}_{k} \subseteq 
\mathcal{K}_{k+1}^{\text{rat}}(A,B,\bm{e}_1,( \xi_1, \hdots, \xi_{k} )),
\end{equation}
for $\varrho_k \notin \Xi$.
Now consider the vector $\bm{k}_k = \LinOp{\alpha_{k}}{\beta_{k}} \bm{e}_k$, with
$\alpha_{k} / \beta_{k} = \xi_{k}$. As $\beta_{k} A - \alpha_{k} B$ is an upper
Hessenberg matrix with a zero in position $(k{+}1,k)$, $\bm{k}_k \in \mathcal{E}_k$.
It follows that,
\begin{equation*}
\bm{k}_{k+1}
= M(\varrho_k,\xi_k) \; \bm{k}_k 
= \LinOp{\mu_{k}}{\nu_{k}} \LinOpInv{\alpha_{k}}{\beta_{k}} \; \bm{k}_k
= \LinOp{\mu_{k}}{\nu_{k}} \ \bm{e}_k,
\end{equation*}
is a vector in $\mathcal{E}_{k+1}$ with $k_{k+1} \neq 0$ and by Equation~\eqref{eq:ME_subset}, 
$\bm{k}_{k+1} \in \mathcal{K}_{k+1}^{\text{rat}}$.
This proves the first result.

In order to prove Equation~\eqref{eq:rksubspace_L_structure}, we can start
in a similar way. Assume the result holds up to dimension $k<n{-}1$\footnote{For
$\mathcal{K}^{\text{rat}}_k$, $k{+}1$ can be as large as $n$ since 
Equation~\eqref{eq:rksubspace_K_structure} goes up to $\xi_{k-1}$. For
$\mathcal{L}^{\text{rat}}_k$, $k{+}1$ is limited to $n{-}1$ as we don't want
to run out of poles.}.
We get from the
nestedness of rational Krylov subspaces and the induction hypothesis that,
\begin{equation*}
\mathcal{E}_k \subseteq \mathcal{L}_{k+1}^{\text{rat}}(A,B,\bm{e}_1, ( \xi_2, \hdots, \xi_{k+1} )).
\end{equation*}
From Equation~\eqref{eq:increase_rksubspace} and the induction hypothesis we deduce,
\begin{equation*}
N(\varrho_{k+1},\xi_{k+1})\ \mathcal{E}_{k} \subseteq \mathcal{L}_{k+1}^{\text{rat}}(A,B,\bm{e}_1, ( \xi_2, \hdots, \xi_{k+1} )),
\end{equation*}
for $\varrho_{k+1} \notin \Xi$.
To complete the proof, we need to show as before that there exists a pair of vectors
$\bm{\ell}_k,\bm{\ell}_{k+1}$,
with $\bm{\ell}_k \in \mathcal{E}_k$ and $\bm{\ell}_{k+1} \in \mathcal{E}_{k+1}$
whose $(k{+}1)$st element $\ell_{k+1} \neq 0$,
that are related as,
\begin{equation}
\label{eq:cond_vec}
\bm{\ell}_{k+1} = N(\varrho_{k+1},\xi_{k+1}) \; \bm{\ell}_k = \LinOpInv{\alpha_{k+1}}{\beta_{k+1}} \LinOp{\mu_{k+1}}{\nu_{k+1}} \; \bm{\ell}_k,
\end{equation}
An explicit construction is not
possible in this case. 
Nonetheless, by Equation~\eqref{eq:cond_vec} we have that 
$(\bm{\ell}_k,\bm{\ell}_{k+1})$ must satisfy
\begin{equation*}
\LinOp{\alpha_{k+1}}{\beta_{k+1}} \ \bm{\ell}_{k+1} = \LinOp{\mu_{k+1}}{\nu_{k+1}} \ \bm{\ell}_k.
\end{equation*}
From properties I.\ and II.\ of \Cref{lemma:properHessproperties}, 
we have that the matrix $\beta_{k+1} A - \alpha_{k+1} B$ is an upper Hessenberg
matrix that admits a block upper triangular partition with a leading block of size
$(k{+}1){\times}(k{+}1)$, while the matrix $\nu_{k+1} A - \mu_{k+1} B$ is a proper
upper Hessenberg matrix since the shift $\varrho_{k+1}$ is different from all the poles.
Observe that all vectors $\bm{\ell}_k \in \mathcal{E}_k$ would lead to a vector
$\bm{\ell}_{k+1}$ with element $\ell_{k{+}1}=0$ if and only if the first $k$ columns
of $\LinOp{\alpha_{k+1}}{\beta_{k+1}}$ would span the same subspace as the first $k$
columns of $\LinOp{\mu_{k+1}}{\nu_{k+1}}$.
It follows from property III.\ and IV.\ of \Cref{lemma:properHessproperties} that this
cannot be true.
We conclude that a valid pair $(\bm{\ell}_k,\bm{\ell}_{k+1})$ must exist.
\end{proof}

A direct corollary of the theorem considers the structure of rational Krylov matrices 
generated from proper Hessenberg pairs.

\begin{corollary}
\label{corr:rkmatrix_triu}
Let $(A,B) \in \mathbb{C}^{n{\times}n}$ be a proper Hessenberg pair with poles
$\Xi = ( \xi_1, \hdots, \xi_{n-1} )$ not in the spectrum.
Then for $k$ from $1$ to $n$ and any $\Rho_k$ different from the spectrum and poles,
 $K_{k}^{\text{rat}}(A,B,\bm{e}_1, ( \xi_1, \hdots, \xi_{k-1} ),\Rho_k)$
and for $k$ from $1$ to $n{-}1$,
 $L_{k}^{\text{rat}}(A,B,\bm{e}_1, ( \xi_2, \hdots, \xi_k ),\Rho_k)$ 
are upper triangular, with non-vanishing diagonal elements.
\end{corollary}

\subsection{Proof of the implicit Q theorem.}
\label{ssec:proofimplQ}

We are  ready to  prove \Cref{thm:implQ}.
\begin{proof}
Choose a tuple of $n{-}1$ shifts $\Rho_{n{-}1}$ different from the poles.
\Cref{corr:rkmatrix_triu} states that
$K_{n}^{\text{rat}}(\hat{A}, \hat{B}, \bm{e}_1, \Xi_{n{-}1},\Rho_{n{-}1})$ and
$K_{n}^{\text{rat}}(\check{A}, \check{B}, \bm{e}_1, \Xi_{n{-}1},\Rho_{n{-}1})$ 
are $n{\times}n$ upper triangular matrices.
The elementary rational function $M(\varrho,\xi)$ is transformed via $\hat{Q}$ and $\check{Q}$ to 
$\hat{M}(\varrho,\xi) = \hat{Q}^{*} \; M(\varrho,\xi) \; \hat{Q}$ and
$\check{M}(\varrho,\xi) = \check{Q}^{*} \; M(\varrho,\xi) \; \check{Q}$.

It follows that,
\begin{align*}
& \hat{Q} \, K_{n}^{\text{rat}}(\hat{A}, \hat{B}, \bm{e}_1, \Xi_{n{-}1}, \Rho_{n{-}1}) \\
= \ & \hat{Q} \, \left[ \bm{e}_1 \ \hat{M}(\varrho_1,\xi_1) \; \bm{e}_1 
\ \hdots \
\left( \prod_{i=1}^{n-1} \hat{M}(\varrho_i,\xi_i) \right) \; \bm{e}_1 \right] \\
= \ & \hat{Q} \, \left[ \bm{e}_1 \ \hat{Q}^* M(\varrho_1,\xi_1) \hat{Q} \; \bm{e}_1 
\ \hdots \
\hat{Q}^* \left( \prod_{i=1}^{n-1} M(\varrho_i,\xi_i) \right) \hat{Q} \; \bm{e}_1
\right] \\
%
= \ & \phantom{Q} \,  \left[ \bm{\hat{q}}_1 \ M(\varrho_1,\xi_1) \; \bm{\hat{q}}_1
\ \hdots \
\left( \prod_{i=1}^{n-1} M(\varrho_i,\xi_i) \right) \; \bm{\hat{q}}_1 \right] \\
= \ & \sigma \;  \left[ \bm{\check{q}}_1 \ M(\varrho_1,\xi_1) \; \bm{\check{q}}_1
\ \hdots \
\left( \prod_{i=1}^{n-1} M(\varrho_i,\xi_i) \right) \; \bm{\check{q}}_1 \right] \\
= \ & \sigma \check{Q} \, K_{n}^{\text{rat}}(\check{A},\check{B}, \bm{e}_1, \Xi_{n{-}1},\Rho_{n{-}1}).
\end{align*}
Since the upper triangular matrices $K^{\text{rat}}_n$ are nonsingular, the uniqueness
of the QR factorization implies the existence of a unitary diagonal matrix
$D_1$ such that $\hat{Q} = \check{Q}D_1$. 

It remains to prove that a similar relation
holds for the matrices $\hat{Z}$ and $\check{Z}$.
Let us first prove that $\hat{Z}$
and $\check{Z}$ also share a first column up to unimodular scaling.
From the relations
$(\beta_1 \hat{A} - \alpha_1 \hat{B})= \hat{Q}^* \LinOp{\alpha_1}{\beta_1} \hat{Z}$ 
and
$(\beta_1 \check{A} - \alpha_1 \check{B})= \check{Q}^* \LinOp{\alpha_1}{\beta_1} \check{Z},$
with $\xi_1=\alpha_1/\beta_1$,
it follows that,
\begin{equation}
\label{eq:z1}
\begin{split}
& \hat{\bm{z}}_1 = \hat{Z}\bm{e}_1 = \LinOpInv{\alpha_1}{\beta_1} \hat{Q} (\beta_1 \hat{A} - \alpha_1 \hat{B}) \bm{e}_1, \\
& \check{\bm{z}}_1 = \check{Z}\bm{e}_1 =  \LinOpInv{\alpha_1}{\beta_1} \check{Q} (\beta_1 \check{A} - \alpha_1 \check{B}) \bm{e}_1.
\end{split}
\end{equation}
Since both $(\beta_1 \hat{A} - \alpha_1 \hat{B}) \bm{e}_1$ and $(\beta_1 \check{A} - \alpha_1 \check{B}) \bm{e}_1$ 
reduce to a scalar multiple of $\bm{e}_1$ and $\hat{Q}\bm{e}_1 = \sigma \check{Q}\bm{e}_1$ we get 
$\check{\bm{z}}_1 = \tilde{\sigma} \hat{\bm{z}}_1$.
Via similar reasoning as before the following two QR factorizations are equal,
\begin{align*}
\hat{Z} \, L_{n{-}1}^{\text{rat}}(\hat{A}, \hat{B}, \bm{e}_1, \Xi_{2:n{-}1}, \Rho_{2{:}n{-}1}) = 
\tilde{\sigma} \check{Z} L_{n{-}1}^{\text{rat}}(\check{A}, \check{B}, \bm{e}_1, \Xi_{2:n{-}1},\Rho_{2{:}n{-}1}).
\end{align*}
In this case the $L_{n{-}1}$ matrices are of size $n{\times}n{-}1$.
Uniqueness of the $QR$ factorization implies essential uniqueness of the
first $n{-}1$ columns of $\hat{Z}$ and $\check{Z}$. Nonetheless also the
last column of $\hat{Z}$ and $\check{Z}$ are essentially the same as they 
are orthogonal to the first $n{-}1$ columns. We conclude that
$\hat{Z} = \check{Z} D_2$, with $D_2$
a unitary diagonal matrix.
\end{proof}

When the Hessenberg pair is not proper, uniqueness can only be guaranteed up to the pole
that causes the problem. This is similar to the Hessenberg case. In practice
this is in fact good news as a breakdown signals a deflation.

%% file: sections/06_subspace_iteration/subspace_iteration.tex
\section{Implicit rational subspace iteration.}
\label{sec:subspaceiter}
\label{ssec:subspaceiter}
\label{sec:subspace:iteration}

It is well-known that Francis' $QR$ algorithm \cite{Fra61,Fra62} effects nested subspace iteration
with a change of coordinate system driven by polynomial Krylov subspaces \cite[Theorem 6.3]{Vandebril2012},
\cite[p.396]{Watkins2011}.  
This result is generalized in this section for the rational QZ method.

Starting with a proper Hessenberg pair $(A,B)$ with 
$\Xi=( \xi_1, \hdots, \xi_{n-1})$,
a single iteration of the rational QZ method with shift $\varrho$ and new pole 
$\hat{\xi}_{n-1}$
results in a new proper Hessenberg pair,
\begin{align*}
(\hat{A},\hat{B}) = Q^* \, (A,B) \, Z,
\end{align*}
with $\hat{\Xi} = ( \xi_2, \hdots, \xi_{n-1}, \hat{\xi}_{n-1} )$.
This equivalence transformation simultaneously performs two similarity transformations
on the matrices,
\begin{equation}
\label{eq:2qr_iters}
\hat{M}(\varrho,\xi) = Q^* \, M(\varrho,\xi) \, Q \quad\mbox{ and }\quad
\hat{N}(\varrho,\xi) = Z^* \, N(\varrho,\xi) \, Z,
\end{equation}
for all $\varrho$ and $\xi$.

The following theorem formalizes the convergence behavior of the RQZ method.

\begin{theorem}
\label{thm:rational_subspace_it}
Consider a single RQZ step 
$(\hat{A},\hat{B}) = Q^* \, (A,B) \, Z$, with shift $\varrho$,
pole tuple ${\Xi} = ( \xi_1, \hdots, \xi_{n-1})$ prior to the RQZ step,
and $\hat{\Xi} = ( \xi_2, \hdots, \xi_{n-1}, \hat{\xi}_{n-1} )$ afterwards.
Assume all poles different from the eigenvalues, and the shift $\varrho$ different from all
eigenvalues and poles.
%
For $k = 1, \hdots, n{-}1$, this effects subspace iteration  driven by $M(\varrho,\xi_k)$
and $N(\varrho,\xi_{k+1})$, we get:
\begin{equation}
\label{eq:subsp_iter}
\mathcal{R}(Q_{:,1:k}) = M(\varrho,\xi_{k}) \, \mathcal{E}_k,
\qquad \text{and} \qquad
\mathcal{R}(Z_{:,1:k}) = N(\varrho,\xi_{k+1}) \, \mathcal{E}_k,
\end{equation}
with, $\xi_n = \hat{\xi}_{n-1}$.
The change of coordinate system  
maps both $\mathcal{R}(Q_{:,1:k})$ and $\mathcal{R}(Z_{:,1:k})$ back to $\mathcal{E}_k$.
\end{theorem}

\begin{proof}
We make use of the properties of \Cref{lemma:prop_elementary_rationals}, \Cref{lemma:prop_rksubspace}, 
\Cref{thm:rksubspace_structure}, Equation~\eqref{eq:2qr_iters} and
$\bm{q}_1 = \gamma M(\varrho,\xi_1) \, \bm{e}_1$ (Equation~\eqref{eq:intro_single_pole}).
We get,
\begin{equation*}
\begin{split}
\mathcal{R}(Q_{:,1:k}) 
& = Q \, \mathcal{E}_k  
= Q \, \mathcal{K}^{\text{rat}}_{k}(\hat{A},\hat{B},\bm{e}_1,\Xi_{2:k}) \\
& = Q \, \prod_{i=2}^{k} \hat{M}(\varrho,\xi_i) \cdot \mathcal{K}_k(\hat{M}(\check{\varrho},\varrho),\bm{e}_1) \\
& = \prod_{i=2}^{k} M(\varrho,\xi_i) \cdot \mathcal{K}_k(M(\check{\varrho},\varrho),Q\bm{e}_1) \\
& = \prod_{i=2}^{k} M(\varrho,\xi_i) \cdot \mathcal{K}_k(M(\check{\varrho},\varrho),M(\varrho,\xi_1) \bm{e}_1)\\
& = M(\varrho,\xi_k) \prod_{i=1}^{k-1} M(\varrho,\xi_i) \cdot \mathcal{K}_k(M(\check{\varrho},\varrho),\bm{e}_1) \\
& = M(\varrho,\xi_k) \, \mathcal{E}_k .
\end{split}
\end{equation*}
The first equality is clear, the second equality uses \Cref{thm:rksubspace_structure}.
The third equality applies part II of \Cref{lemma:prop_rksubspace}.
The fourth equality relies on Equation \eqref{eq:2qr_iters} to change from $\hat{M}$ to $M$. The fifth equality
uses the expression for $\bm{q}_1$, the sixth uses the commutativity property, and the last equality again applies \Cref{lemma:prop_rksubspace}
and \Cref{thm:rksubspace_structure}.

The second result follows a similar reasoning.
The only difference is the relation between $\bm{z}_1$ and $\bm{e}_1$.
Starting from the same argument as in Equation \eqref{eq:z1} we get, for some constants $\gamma,\check{\gamma}$ and $\tilde{\gamma}$,
\begin{align*}
\bm{z}_1 
= \ & {\gamma} \LinOpInv{\alpha_2}{\beta_2} \bm{q}_1 
= \ {\check{\gamma}} \, \LinOpInv{\alpha_2}{\beta_2} M(\varrho,\xi_1) \, \bm{e}_1
= \ {\tilde{\gamma}} \, N(\varrho,\xi_2) \, \bm{e}_1.
\end{align*}
\end{proof}

A single shifted RQZ step will
execute a QR step with shift $\varrho$ on the entire space simultaneously with RQ steps
having shifts $\xi_i$ on selected subspaces. 
The shift $\varrho$ is rapidly moving from top
to bottom and thus affects all subspaces. 
The poles on the other hand are slowly moving
upwards, one row during each step, and as such do not act on all subspaces in a single
RQZ step. 
The shifts will rapidly initiate convergence at
the bottom, the poles slowly push converged eigenvalues to the top. This is another
justification of why, in the
classical QZ algorithm, the zero eigenvalues in $B$ appear at the top: they are
pushed there by the poles at infinity. Moreover, it is also clear from the analysis
that picking a shift equal to a pole will lead to
cancellation in some of the factors thereby slowing down convergence.

Note that in the formulation of \Cref{thm:rational_subspace_it} the shift and poles
are assumed to be different from the eigenvalues of the matrix pair. This is imposed to
ensure that the required inverses exist.
However, in practical implementations, these parameters will typically converge towards an eigenvalue.
This is in fact a desirable situation as it will lead to deflations.

In the QZ algorithm \cite{Moler1973}, all poles are at $\infty$ and the two driving functions 
reduce to $M(\varrho,\infty)$ and $N(\varrho,\infty)$ which is equivalent to $AB^{-1} - \varrho I$
and $B^{-1}A - \varrho I$.
In the RQZ method, the poles can be chosen freely and as such they can be utilized to influence 
the convergence of the method as was illustrated in the numerical experiments of \Cref{subsec:example1,subsec:example2}.
Note that, as the poles only shift one row up during every RQZ step, 
it takes $n{-}1$ iterations before a pole has moved from the bottom to the top and has
influenced all
vectors in the subspace iteration.

To further clarify the result of \Cref{thm:rational_subspace_it} consider the simplified case
where all the poles of the Hessenberg pair are equal to same value $\xi$ different
from the eigenvalues of $(A,B)$.
Assume that the RQZ algorithm is applied $s$ times on this proper Hessenberg pair
with the same shift $\varrho$. At the end of each RQZ step the last pole is again restored
to $\xi$.
Then the subspace iterations, as considered from the initial pair, are given by,
\begin{equation*}
Q: \; \mathcal{E}_k \rightarrow M(\varrho,\xi)^s \mathcal{E}_k,
\quad \text{and} \quad
Z: \; \mathcal{E}_k \rightarrow N(\varrho,\xi)^s \mathcal{E}_k.
\end{equation*}
Denote $q(z) = (z - \varrho)/(z - \xi)$ and let $\lambda_1, \hdots, \lambda_n$
be the eigenvalues of the pair $(A,B)$, so that $q(\lambda_i)^s$ is the rational
filter that is implicitly applied during these $s$ iterations to $\lambda_i$.
Assume the eigenvalues are ordered such that,
\begin{equation*}
|q(\lambda_1)^s| \leq |q(\lambda_2)^s| \leq \hdots \leq |q(\lambda_{n-1})^s| \leq |q(\lambda_n)^s|,
\end{equation*}
then the convergence factor of an eigenvalue at the end of the Hessenberg pencil
is given by $|q(\lambda_1)^s|/|q(\lambda_2)^s|$, while the convergence factor
at the top of the Hessenberg pencil is given by $|q(\lambda_{n-1})^s|/|q(\lambda_n)^s|$.
As such, a good choice of both poles and shifts can accelerate convergence and
lead to deflations.

As an example consider a problem of size $11$ with eigenvalues located
on the unit circle in the complex plane. \Cref{fig:ssiter} shows the absolute
value of the rational filter after $s{=}2$ iterations for two different choices
for the rational function $q$.
Figure~\subref*{subfig:ssiterpoly} shows the filter, $q_{\infty}(z)^2$,
with shift $\varrho = -0.95$
and all the poles at $\infty$. This situation corresponds to the QZ method
applied twice with the same shift to a Hessenberg, triangular pair.
The shift $\varrho$ is located close to the eigenvalue $\lambda_1{=}-1$ such that
$|q_{\infty}(\lambda_1)^2| = 2.5 \cdot 10^{-3}$ is the minimal value of 
the filter over all eigenvalues.
The convergence factor of $\lambda_1$ at the end of the pencil is approximately
$8.22 \cdot 10^{-3}$. At the top of the pencil there is no convergence in this
case as $|q(\lambda_{n-1})^2|/|q(\lambda_n)^2| = 1$.
Figure~\subref*{subfig2:ssiterrat} shows the same experiment but this time the poles are located
at $\xi = 0.1{+}1i$ which is in the vicinity of another eigenvalue. This
situation corresponds to the RQZ method applied twice with the same shift to a Hessenberg
pair with $\Xi = (\xi, \hdots, \xi)$.
The rational filter, $q_{\xi}(z)^2$, leads to a convergence factor of $\lambda_1$ at the 
end of the pencil of approximately $1.21 \cdot 10^{-2}$.
The convergence of $\lambda_1$ at the end of the pencil is slower with $q_{\xi}^2$
compared to $q_{\infty}^2$.
However, $q_{\xi}^2$ will also lead to convergence at the top of the pencil as the
convergence factor is $|q(\lambda_{n-1})^2|/|q(\lambda_n)^2|
\approx 7.46 \cdot 10^{-3}.$
We observe that using $q_{\xi}$ leads to convergence of another eigenvalue, where
$q_{\infty}$ does not.

\begin{figure}[htp]
     \subfloat[Filter with $\xi = \infty$\label{subfig:ssiterpoly}]{%
       \includegraphics[width=0.45\textwidth]{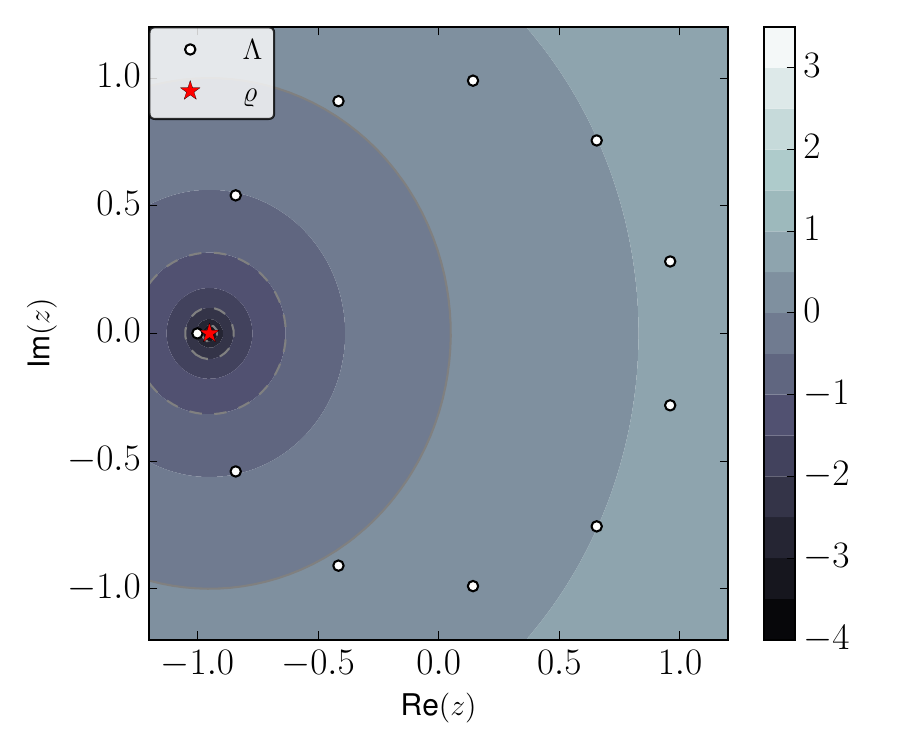}
     }
     \hfill
     \subfloat[Filter with $\xi = 0.1 + 1i$\label{subfig2:ssiterrat}]{%
       \includegraphics[width=0.45\textwidth]{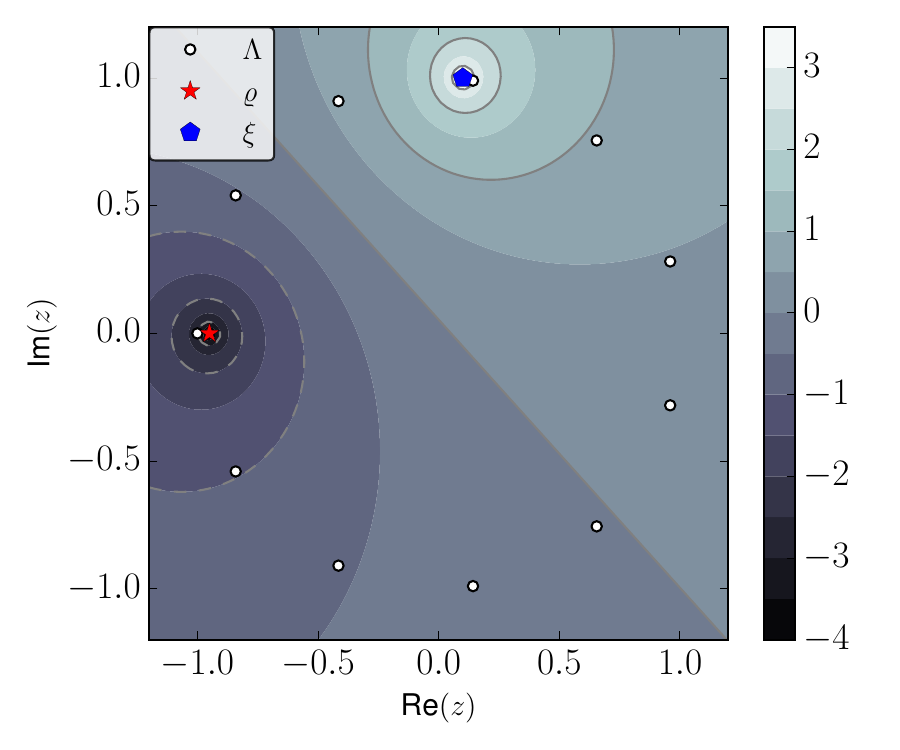}
     }
     \caption{Logarithm of the absolute value of the rational filter, $|q(\lambda_i)^s|$, after
     $s=2$ iterations with $\varrho = -0.95$, and $\xi$ either at $\infty$ or
     at $0.1 + 1i$. The eigenvalues $\lambda_i$ are shown with cirlces, the 
     shift $\varrho$ is indicated with a star, the pole with a pentagon. Darker
     regions agree with convergence at the end of the pencil and lighter regions 
     with convergence at the top of the pencil.}
     \label{fig:ssiter}
\end{figure}

It is clear that both the shifts and the poles can accelerate the convergence
but they do influence each other.

When the shifts are changed in every iteration and the poles of the Hessenberg
pair are not the same then the filter $q$ becomes dependent on the index $k$ and will be a
product of terms with different shifts,
\begin{equation}
\label{eq:ratfilter}
q_k(\lambda) = \prod_{i=1}^{s} (\lambda - \varrho_i)/(\lambda - \xi_k^{(i)}),
\end{equation}
with $\xi_k^{(i)}$ the pole at iteration $i$ in position $k$ (or $k{+}1$) for $Q$ (or
$Z$) as shown in \Cref{thm:rational_subspace_it}.

Provided a good choice of shifts and poles is made during repeated application of the RQZ
algorithm the pair $(A,B)$ will converge to a pair of upper triangular matrices.

%% file: sections/08_filter_rational_krylov/filter_rational_krylov.tex
\section{Filtering rational Krylov subspaces.}
\label{sec:RK}

In the last part, we will apply the concept of the RQZ method within
the rational Krylov (RK) method \cite{Ruhe1984,Ruhe1994a, Ruhe1994b,
Ruhe1998a} to filter and restart the RK method.
The RK method is an iterative method  applicable for, for example, computing
a select subset of eigenvalues of large-scale eigenvalue problems.
Our formulation is in terms of a large-scale complex-valued matrix pair
$(A,B)$ of dimension $N{\times}N$.

Starting from a regular pair $(A,B)$, a nonzero vector $\bm{v} \in \CC^{N}$,
and a tuple of poles $\Xi = ( \xi_1, \hdots, \xi_k )$ different from the
spectrum, 
the rational Krylov method iteratively constructs an orthonormal basis 
$V_{k+1} \in \mathbb{C}^{N{\times}k+1}$ of the rational Krylov
subspace $\mathcal{K}^{\text{rat}}_{k+1}(A,B,\bm{v},\Xi)$.

It also constructs a $(k{+}1){\times}k$ recurrence matrix pair $(\underline{H}_k, \underline{G}_k)$
in Hessenberg form. This RK Hessenberg pair contains the poles that are used in the
RK method as its subdiagonal elements: $\xi_i = h_{i+1,i}/g_{i+1,i}$, for $i$ from $1$
to $k$. 
This is similar to the square Hessenberg pairs that are used in the first part
of this paper.
As long as the RK method does not break down, the pair $(\underline{H}_k, \underline{G}_k)$
can be considered as proper according to two out of three conditions of \Cref{def:ratHess}.
The third condition concerning the linear independence of the last row of the pair 
does not hold for the rectangular RK Hessenberg pencil.
If the other two conditions are satisfied, we nonetheless say that $(\underline{H}_k, \underline{G}_k)$
is a proper RK Hessenberg pair.

The RK recurrence relation,
\begin{equation}
\label{eq:RK_recurrence}
A \, V_{k+1} \, \underline{G}_k = B \, V_{k+1} \, \underline{H}_k,
\end{equation}
holds throughout the RK method.
We refer the interested reader to \cite{Ruhe1998a} for further algorithmic details
on the iterative scheme.

The RK method generalizes Arnoldi's method \cite{Arnoldi1951} which generates 
an orthonormal basis for a standard Krylov subspace $\mathcal{K}_{k+1}(A,\bm{v})$.
These Krylov methods have a growing orthogonalization cost and growing memory requirements
with increasing subspace dimension. To overcome this one could apply implicit filtering and
restart.

Sorensen \cite{Sorensen1992} applied Francis' QR algorithm to filter a standard Krylov subspace
and implicitly restart the Arnoldi iteration.
The implicit QR algorithm can be applied to restart Arnoldi's method because they are both based
on the Hessenberg matrix structure.
As a generalization, the RQZ algorithm can be applied to filter a rational Krylov subspace
and restart the rational Krylov iteration because both methods use the structure of Hessenberg pairs.
Berljafa \& G\"uttel \cite[section 4.3]{Berljafa2015} already proposed this technique of
changing and swapping the poles in the RK method as a way to implicitly filter a rational
Krylov subspace.
The first algorithm to apply an implicit filter in the RK method is due to
De Samblanx, Meerbergen, and Bultheel \cite{DeSamblanx1997}. 
However, their method relied on an explicit QZ algorithm
which is quite costly and prone to numerical inaccuracies.

To filter a rational Krylov subspace we can thus use the concept of the RQZ method.
The procedure is summarized in \Cref{fig:rkfilt}. The initial situation of the
RK Hessenberg is shown in pane I on the left. In pane II, the first pole $\xi_1$ is
changed to a shift $\varrho$ by computing a unitary transformation $Q$ such that,
\begin{equation}
\bm{q}_1 = 
\check{\gamma} (\underline{H}_k - \varrho \underline{G}_k)(\underline{H}_k{-}\xi_1 \, \underline{G}_k)^{\dagger} \bm{e}_1
= \hat{\gamma} (\underline{H}_k - \varrho \underline{G}_k)  \bm{e}_1.
\end{equation}
The principle is the same as described in \Cref{subsec:manipulate}, the only
difference is that the inverse is replaced with the Moore-Penrose \emph{pseudoinverse} 
$(\underline{H}_k{-}\xi_1 \, \underline{G}_k)^{\dagger}$.

It is well-known \cite{b037} that $\bm{x}_{LS} = (\underline{H}_k{-}\xi_1 \, \underline{G}_k)^{\dagger} \bm{b}$ is the least squares solution
of minimal norm $\| \bm{x} \|_2$.
As $\|\gamma \bm{e}_1 -  (\underline{H}_k - \xi_1 \underline{G}_k)  \bm{e}_1\|_2 = 0$ when
$\gamma = h_{11} - \xi_1 g_{11}$, we conclude that,
\begin{equation}
\label{eq:LSsol}
(\underline{H}_k{-}\xi_1 \, \underline{G}_k)^{\dagger} \bm{e}_1 = \gamma \bm{e}_1.
\end{equation}


Pane II of \Cref{fig:rkfilt} further shows how the shift is swapped to the last position
on the subdiagonal of $(\underline{H}_k, \underline{G}_k)$. The end result is displayed in pane III.

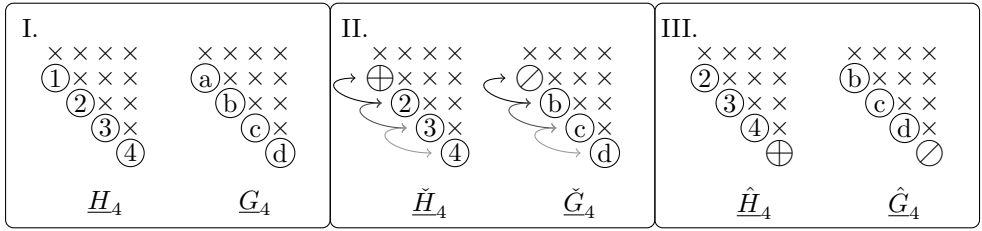
\begin{figure}[htp]
\input{fig/fig_rkfilt.tikz}
\caption{RQZ-like procedure to change the first pole in an RK Hessenberg pair to a new shift (Pane II)
and move it to the last position in the RK Hessenberg pair (Pane II-III).}
\label{fig:rkfilt}
\end{figure}

\noindent The process summarized in \Cref{fig:rkfilt} effectively updates,
$$(\underline{\hat{H}}_k, \underline{\hat{G}}_k) 
= Q^* (\underline{H}_k, \underline{G}_k) Z,$$
in such a way that the pole tuple is changed to
$\hat{\Xi} = (\xi_2,\hdots,\xi_k,\varrho)$.
To maintain the rational Krylov recurrence \cref{eq:RK_recurrence},
the orthonormal basis is updated as $\hat{V}_{k+1} = V_{k+1} Q$.

This does not change the span of $V_{k+1}$, i.e. $\mathcal{R}(\hat{V}_{k+1}) =
\mathcal{R}(V_{k+1})$, but the vectors are \emph{rearranged}.
The new start vector is given by:
\begin{equation}
\label{eq:rkeq2}
\hat{\bm{v}} = \hat{V}_{k+1} \, \bm{e}_1 = V_{k+1}  \, \bm{q}_1 = \gamma V_{k+1} (\underline{H}_k{-}\varrho \underline{G}_k)\bm{e}_1.
\end{equation}
The rational Krylov recurrence \eqref{eq:RK_recurrence} implies,
\begin{equation}
\label{eq:rkeq}
(A{-}\varrho B) \, V_{k+1} \, (\underline{H}_k{-}\xi_1  \underline{G}_k) = (A{-}\xi_1 B) \, V_{k+1} \, (\underline{H}_k{-}\varrho  \underline{G}_k).
\end{equation}
Rearranging terms in Equation~\eqref{eq:rkeq} and combining this with Equations~(\ref{eq:LSsol},~\ref{eq:rkeq2}) we see that
the new starting vector is given by:
\begin{equation}
\label{eq:startvec}
\hat{\bm{v}} = \gamma (A{-}\xi_1 \,B)^{-1}(A{-}\varrho \, B) \, \bm{v}.
\end{equation}

From the uniqueness of a rational Krylov recurrence \eqref{eq:RK_recurrence} 
\cite[Theorem 3.2]{Berljafa2015}, it follows that
$ \mathcal{R}(\hat{V}_{k+1}) = \mathcal{K}^{\text{rat}}_{k+1}(A,B, \hat{\bm{v}}, \hat{\Xi})$.

The filter operation is finalized by removing the last pole $\varrho$ from the subspace by
reducing the order of the rational Krylov recurrence by
one. This means that the trailing column and row of $(\underline{\hat{H}}_k, \underline{\hat{G}}_k)$ 
are removed, as well as the last vector of $\hat{V}_{k+1}$.

With the results of \Cref{sec:implQthm} and \Cref{sec:subspaceiter} in mind it is clear
how this RQZ-like procedure applies a filter in the RK iteration.

As a numerical experiment, we revisit the two fluid flow problems of \Cref{subsec:example2}.
Instead of computing all eigenvalues 
we are now only interested in determining if the problems are stable.
To this end, it is sufficient to determine the rightmost eigenvalues of both problems
and check if they are situated in the left half-plane.
As input to the restarted rational Krylov method we have
the matrix pair $(A,B)$, a start vector $\bm{v}$, a tuple
of poles $\Xi$, a maximal subspace dimension $m$, a restart length $p$, a number of desired
Ritz values $\ell$ and a tolerance \texttt{tol} up to which the $\ell$ Ritz values need to
be converged.
The residual is determined as
$\| (\underline{H}_k - \theta \underline{G}_k) \, \bm{y} \|_2$,
with $(\theta,\bm{y})$ the Ritz value and Ritz vector computed from $(H_k,G_k)$, the upper $k{\times}k$ block of the recurrence pencil.
The iteration starts with computing an initial rational Krylov subspace of dimension $m$. If
the $\ell$ rightmost Ritz values have converged up to a maximal residual \texttt{tol}, the iteration is halted.
Otherwise the $p$ leftmost Ritz values are selected as shifts, the subspace is reduced to dimension
$m{-}p$ by using the RQZ method to filter the subspace, and the subspace is again expanded to full
dimension $m$. This procedure is repeated until the $\ell$ rightmost Ritz values have
converged.

The settings and results are summarized in \Cref{tab:numexp4}. In both cases, we selected
poles along the imaginary axis, $\Xi = ( -20i, -18i, \hdots, 18i, 20i )$, as we expect
the rightmost eigenvalue to be situated close to it.

\begin{table}[htp]
\centering
\caption{Summary of the settings and results of the restarted rational Krylov iteration. The columns list the maximal subspace dimension $m$, the restart length $p$, the number of wanted Ritz values $\ell$, the tolerance \texttt{tol}, and the required number of restarts to reach convergence.} 
\scalebox{0.92}{
\input{tab/tab_numexp4.tex}
}
\label{tab:numexp4}
\end{table}

\Cref{fig:numexp4} shows the rightmost part of the spectrum and the converged Ritz values. As can be seen, the method successfully converged to the correct eigenvalues within a reasonable number of restarts.

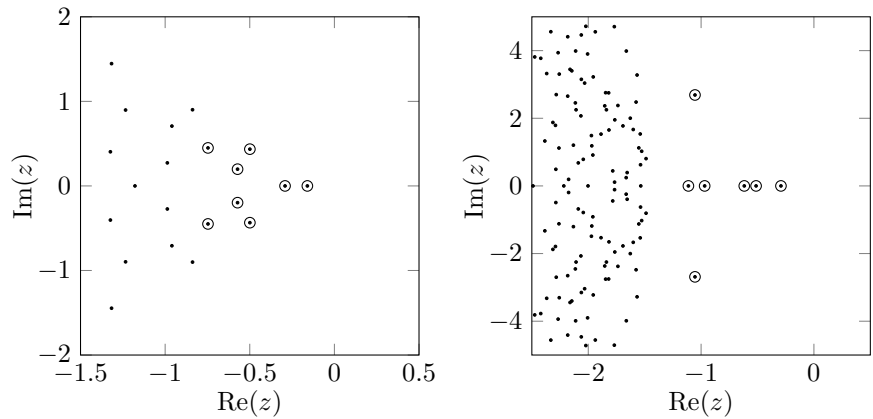
\begin{figure}[htp]
\centering
\input{fig/numexp4/fig_numexp4.tikz}
\caption{Rightmost part of the spectrum of the cavity flow (left) and obstacle flow (right) problems. The eigenvalues (\ref{line:ne4eig}) and Ritz values (\ref{line:ne4ritz}) are shown.}
\label{fig:numexp4}
\end{figure}

%% file: fig/fig_rkfilt.tikz
\begin{tikzpicture}[scale=1.66,y=-1cm]

\tikzset{bloplus/.style={draw,circle,append after command={
        [shorten >=\pgflinewidth, shorten <=\pgflinewidth,]
        (\tikzlastnode.north) edge[white,line width=1.0pt] (\tikzlastnode.south)
        (\tikzlastnode.east) edge[white,line width=1.0pt] (\tikzlastnode.west)
        }
    }
}

\tikzset{bloslash/.style={draw,circle,append after command={
        [shorten >=\pgflinewidth, shorten <=\pgflinewidth,]
        (\tikzlastnode.north east) edge[white,line width=1.0pt] (\tikzlastnode.south west)
        }
    }
}

\def\yos{0}
\def\xosA{0}
\node[] at (0.6+\xosA,1.4+\yos) {$\underline{H}_4$};

\foreach \y in {1,...,4}{
	\foreach \x in {\y,...,4}{
		\node[] at (\x/5+\xosA,\y/5+\yos) {$\times$};
	}
};

\node[] at (0.2+\xosA,0.4+\yos) {\circled{1}};
\node[] at (0.4+\xosA,0.6+\yos) {\circled{2}};
\node[] at (0.6+\xosA,0.8+\yos) {\circled{3}};
\node[] at (0.8+\xosA,1+\yos) {\circled{4}};

\def\xosB{1.2}
\node[] at (0.6+\xosB,1.4+\yos) {$\underline{G}_4$};

\foreach \y in {1,...,4}{
	\foreach \x in {\y,...,4}{
		\node[] at (\x/5+\xosB,\y/5+\yos) {$\times$};
	}
};

\node[] at (0.2+\xosB,0.4+\yos) {\circledletter{a}};
\node[] at (0.4+\xosB,0.6+\yos) {\circledletter{b}};
\node[] at (0.6+\xosB,0.8+\yos) {\circledletter{c}};
\node[] at (0.8+\xosB,1+\yos) {\circledletter{d}};

\draw[rounded corners=3pt] (-0.2+\xosA,-0.2+\yos) rectangle (\xosB+1.2,1.6+\yos);
\node[] at (\xosA,0+\yos) {I.};

\def\xosA{2.6}
\node[] at (0.6+\xosA,1.4+\yos) {$\underline{\check{H}}_4$};

\foreach \y in {1,...,4}{
	\foreach \x in {\y,...,4}{
		\node[] at (\x/5+\xosA,\y/5+\yos) {$\times$};
	}
};

\node[] (s1) at (0.2+\xosA,0.4+\yos) {\scalebox{1.5}{$\oplus$}};
\node[] (s2) at (0.4+\xosA,0.6+\yos) {\circled{2}};
\node[] (s3) at (0.6+\xosA,0.8+\yos) {\circled{3}};
\node[] (s4) at (0.8+\xosA,1+\yos) {\circled{4}};

\path[<->] (s1) edge [min distance=2mm,in=180,out=180,looseness=3,color=black!100] (s2);
\path[<->] (s2) edge [min distance=2mm,in=180,out=180,looseness=3,color=black!70] (s3);
\path[<->] (s3) edge [min distance=2mm,in=180,out=180,looseness=3,color=black!40] (s4);

\def\xosB{3.8}
\node[] at (0.6+\xosB,1.4+\yos) {$\underline{\check{G}}_4$};

\foreach \y in {1,...,4}{
	\foreach \x in {\y,...,4}{
		\node[] at (\x/5+\xosB,\y/5+\yos) {$\times$};
	}
};

\node[] (g1) at (0.2+\xosB,0.4+\yos) {\scalebox{1.5}{$\oslash$}};
\node[] (g2) at (0.4+\xosB,0.6+\yos) {\circledletter{b}};
\node[] (g3) at (0.6+\xosB,0.8+\yos) {\circledletter{c}};
\node[] (g4) at (0.8+\xosB,1+\yos) {\circledletter{d}};

\path[<->] (g1) edge [min distance=2mm,in=180,out=180,looseness=3,color=black!100] (g2);
\path[<->] (g2) edge [min distance=2mm,in=180,out=180,looseness=3,color=black!70] (g3);
\path[<->] (g3) edge [min distance=2mm,in=180,out=180,looseness=3,color=black!40] (g4);

\draw[rounded corners=3pt] (-0.2+\xosA,-0.2+\yos) rectangle (\xosB+1.2,1.6+\yos);
\node[] at (\xosA,0+\yos) {II.};

\def\xosA{5.2}
\node[] at (0.6+\xosA,1.4+\yos) {$\underline{\hat{H}}_4$};

\foreach \y in {1,...,4}{
	\foreach \x in {\y,...,4}{
		\node[] at (\x/5+\xosA,\y/5+\yos) {$\times$};
	}
};

\node[] at (0.2+\xosA,0.4+\yos) {\circled{2}};
\node[] at (0.4+\xosA,0.6+\yos) {\circled{3}};
\node[] at (0.6+\xosA,0.8+\yos) {\circled{4}};
\node[] at (0.8+\xosA,1+\yos) {\scalebox{1.5}{$\oplus$}};

\def\xosB{6.4}
\node[] at (0.6+\xosB,1.4+\yos) {$\underline{\hat{G}}_4$};

\foreach \y in {1,...,4}{
	\foreach \x in {\y,...,4}{
		\node[] at (\x/5+\xosB,\y/5+\yos) {$\times$};
	}
};

\node[] at (0.2+\xosB,0.4+\yos) {\circledletter{b}};
\node[] at (0.4+\xosB,0.6+\yos) {\circledletter{c}};
\node[] at (0.6+\xosB,0.8+\yos) {\circledletter{d}};
\node[] at (0.8+\xosB,1+\yos) {\scalebox{1.5}{$\oslash$}};

\draw[rounded corners=3pt] (-0.2+\xosA,-0.2+\yos) rectangle (\xosB+1.2,1.6+\yos);
\node[] at (\xosA,0+\yos) {III.};

\end{tikzpicture}

%% file: tab/tab_numexp4.tex
\begin{tabular}{l | l | l | l | l  | l}
Problem & $m$ & $p$ & $\ell$ & \texttt{tol} & $\#$ restarts\\
\hline
\emph{Cavity flow} & $40$ & $20$ & $8$ & $10^{-7}$ & $8$ \\
\emph{Obstacle flow} & $60$ & $25$ & $7$  & $10^{-7}$ & $11$
\end{tabular}

%% file: fig/numexp4/fig_numexp4.tikz
\begin{tikzpicture}



\begin{axis}[%
xlabel=$\operatorname{Re}(z)$,
ylabel=$\operatorname{Im}(z)$,
at={(0cm,0cm)},
width=4.5cm,
height=4.5cm,
ylabel shift = -0.25 cm,
xlabel shift = -0.15 cm,
scale only axis,
xmin=-1.5,
xmax=0.5,
ymin=-2,
ymax=2,
axis background/.style={fill=white}
]

\addplot [only marks, mark=*, mark size=0.5, mark options={solid, black}, forget plot]
  table[row sep=crcr]{%
-0.159712918063673	0\\
-0.291580095046349	0\\
-0.499375162059999	0.43481066206788\\
-0.499375162059999	-0.43481066206788\\
-0.572151001835489	0.198144422113899\\
-0.572151001835489	-0.198144422113899\\
-0.746854788844477	0.449523843073493\\
-0.746854788844477	-0.449523843073493\\
-0.838180282655867	0.900877569263931\\
-0.838180282655867	-0.900877569263931\\
-0.960104171547016	0.707683923108078\\
-0.960104171547016	-0.707683923108078\\
-0.98706477433271	0.272696621173848\\
-0.98706477433271	-0.272696621173848\\
-1.31663316397681	1.44619045376916\\
-1.31663316397681	-1.44619045376916\\
-1.23379810802095	0.897665555406166\\
-1.23379810802095	-0.897665555406166\\
-1.17832068796896	0\\
-1.53224069371478	1.33916634643493\\
-1.53224069371478	-1.33916634643493\\
-1.32350258375136	0.403699037416036\\
-1.32350258375136	-0.403699037416036\\
-1.57010011781152	0.333511637170503\\
-1.57010011781152	-0.333511637170503\\
};\label{line:ne4eig}

\addplot [only marks, mark=o, mark options={solid, black}, forget plot]
  table[row sep=crcr]{%
-0.159712918064107	7.56778486713246e-15\\
-0.291580095046721	1.60614563928425e-14\\
-0.499375162060242	-0.434810662067879\\
-0.499375162060561	0.434810662067587\\
-0.572151001831247	-0.198144422111986\\
-0.572151001833212	0.198144422115411\\
-0.746854785647083	-0.449523841356595\\
-0.746854788379787	0.449523848396488\\
};\label{line:ne4ritz}
\end{axis}



\begin{axis}[%
xlabel=$\operatorname{Re}(z)$,
ylabel=$\operatorname{Im}(z)$,
at={(6cm,0cm)},
width=4.5cm,
height=4.5cm,
ylabel shift = -0.25 cm,
xlabel shift = -0.15 cm,
scale only axis,
xmin=-2.5,
xmax=0.5,
ymin=-5,
ymax=5,
axis background/.style={fill=white}
]
\addplot [color=black, draw=none, mark=*, mark size=0.5, mark options={solid, black}, forget plot]
  table[row sep=crcr]{%
-2.59271500162099	9.63460806355728\\
-2.59271500162099	-9.63460806355728\\
-2.58826144323502	9.34261241134703\\
-2.58826144323502	-9.34261241134703\\
-2.57482303431782	8.65293376178496\\
-2.57482303431782	-8.65293376178496\\
-2.59169072674715	8.63636498954289\\
-2.59169072674715	-8.63636498954289\\
-2.51543536575588	8.26974829337135\\
-2.51543536575588	-8.26974829337135\\
-2.58930420782078	8.18256021223072\\
-2.58930420782078	-8.18256021223072\\
-2.29110691400536	7.66440283764981\\
-2.29110691400536	-7.66440283764981\\
-2.38190428480501	7.63089244031131\\
-2.38190428480501	-7.63089244031131\\
-2.59395884043457	7.46768300322401\\
-2.59395884043457	-7.46768300322401\\
-2.23645913427619	7.15801026027781\\
-2.23645913427619	-7.15801026027781\\
-2.2587726994072	7.15217121093921\\
-2.2587726994072	-7.15217121093921\\
-2.19711953128175	6.96119582300824\\
-2.19711953128175	-6.96119582300824\\
-2.34818133946733	7.01082094293896\\
-2.34818133946733	-7.01082094293896\\
-2.46652258980588	6.82663457311761\\
-2.46652258980588	-6.82663457311761\\
-2.05394927227972	6.20857767403149\\
-2.05394927227972	-6.20857767403149\\
-2.23017120758311	6.32488414637513\\
-2.23017120758311	-6.32488414637513\\
-2.33829485355742	6.27965609789076\\
-2.33829485355742	-6.27965609789076\\
-2.54009217612072	6.10289066986332\\
-2.54009217612072	-6.10289066986332\\
-1.90331921247695	5.43816283005138\\
-1.90331921247695	-5.43816283005138\\
-2.23667615529012	5.7166350952523\\
-2.23667615529012	-5.7166350952523\\
-2.25118065433762	5.72476760794826\\
-2.25118065433762	-5.72476760794826\\
-2.14405889155524	5.20791839904344\\
-2.14405889155524	-5.20791839904344\\
-2.34635604416722	5.30608011518909\\
-2.34635604416722	-5.30608011518909\\
-2.11841728334256	5.13358602238996\\
-2.11841728334256	-5.13358602238996\\
-1.76904878296007	4.70840735151439\\
-1.76904878296007	-4.70840735151439\\
-2.02144967509811	4.71949003617102\\
-2.02144967509811	-4.71949003617102\\
-1.93760723039927	4.55612436442593\\
-1.93760723039927	-4.55612436442593\\
-0.29220994533705	0\\
-1.66277901844049	3.98741816408562\\
-1.66277901844049	-3.98741816408562\\
-2.06179127354846	4.46313402042804\\
-2.06179127354846	-4.46313402042804\\
-2.33174670188559	4.5586660090781\\
-2.33174670188559	-4.5586660090781\\
-2.18185277679699	4.41035991219882\\
-2.18185277679699	-4.41035991219882\\
-1.05417149218136	2.68867487397892\\
-1.05417149218136	-2.68867487397892\\
-0.512581198633849	0\\
-0.617276558451016	0\\
-2.00581288249388	3.90092990355196\\
-2.00581288249388	-3.90092990355196\\
-2.11180793679427	3.98894314008162\\
-2.11180793679427	-3.98894314008162\\
-1.56736085168747	3.27891238293217\\
-1.56736085168747	-3.27891238293217\\
-2.26718700030579	3.93849068292847\\
-2.26718700030579	-3.93849068292847\\
-2.5959946345953	4.13109615434547\\
-2.5959946345953	-4.13109615434547\\
-2.42094979132644	3.77370030035644\\
-2.42094979132644	-3.77370030035644\\
-2.47450979249977	3.81327734546832\\
-2.47450979249977	-3.81327734546832\\
-0.968375059086198	0\\
-1.9558451351958	3.22314767066926\\
-1.9558451351958	-3.22314767066926\\
-2.16138717680506	3.44610228430512\\
-2.16138717680506	-3.44610228430512\\
-2.14567570719206	3.4026563591976\\
-2.14567570719206	-3.4026563591976\\
-1.57620839378037	2.47918225656029\\
-1.57620839378037	-2.47918225656029\\
-2.06124133965359	3.15153022706239\\
-2.06124133965359	-3.15153022706239\\
-1.11248792298412	0\\
-2.03209513586719	3.03990605726079\\
-2.03209513586719	-3.03990605726079\\
-1.8191668250326	2.75194175358989\\
-1.8191668250326	-2.75194175358989\\
-1.84600191389424	2.75582523113899\\
-1.84600191389424	-2.75582523113899\\
-2.25708344333188	3.30583015736473\\
-2.25708344333188	-3.30583015736473\\
-2.36827437591197	3.32235715086626\\
-2.36827437591197	-3.32235715086626\\
-1.73716366981212	2.377640317302\\
-1.73716366981212	-2.377640317302\\
-1.62762487408061	2.00218566455309\\
-1.62762487408061	-2.00218566455309\\
-1.8538919241351	2.367355586805\\
-1.8538919241351	-2.367355586805\\
-1.53935196112627	1.5356177716047\\
-1.53935196112627	-1.5356177716047\\
-1.83912405754011	2.2512438933837\\
-1.83912405754011	-2.2512438933837\\
-1.60143722568513	1.6677973563918\\
-1.60143722568513	-1.6677973563918\\
-1.76470303367534	1.95343621179175\\
-1.76470303367534	-1.95343621179175\\
-1.69254004502293	1.77431270225741\\
-1.69254004502293	-1.77431270225741\\
-1.52641853026536	1.02520382653642\\
-1.52641853026536	-1.02520382653642\\
-1.55197939598621	1.12485663951532\\
-1.55197939598621	-1.12485663951532\\
-1.48779656956898	0.806836006910084\\
-1.48779656956898	-0.806836006910084\\
-2.18209142417956	2.65386470625523\\
-2.18209142417956	-2.65386470625523\\
-2.11438575812238	2.45461236665996\\
-2.11438575812238	-2.45461236665996\\
-2.28374901014749	2.69934795034849\\
-2.28374901014749	-2.69934795034849\\
-1.53664919515953	0.625410225454939\\
-1.53664919515953	-0.625410225454939\\
-1.81649491547019	1.65393612286974\\
-1.81649491547019	-1.65393612286974\\
-1.5345511879634	0\\
-2.10810860728106	2.25072926751078\\
-2.10810860728106	-2.25072926751078\\
-2.06515326521161	2.07045704963437\\
-2.06515326521161	-2.07045704963437\\
-1.88686829524557	1.5314231778861\\
-1.88686829524557	-1.5314231778861\\
-1.65534581348564	0.394088452050773\\
-1.65534581348564	-0.394088452050773\\
-1.66401150745776	0.244201065232653\\
-1.66401150745776	-0.244201065232653\\
-1.97204254685297	1.48814061887978\\
-1.97204254685297	-1.48814061887978\\
-1.78302691831722	0.444181040763783\\
-1.78302691831722	-0.444181040763783\\
-1.76677435998857	0.111542240571662\\
-1.76677435998857	-0.111542240571662\\
-1.9686645280974	1.18508498117067\\
-1.9686645280974	-1.18508498117067\\
-1.95925326405819	0.913022952184482\\
-1.95925326405819	-0.913022952184482\\
-2.31349278242437	1.87615432497466\\
-2.31349278242437	-1.87615432497466\\
-2.29062324868326	1.79064311378766\\
-2.29062324868326	-1.79064311378766\\
-2.13166070975787	1.20476067790126\\
-2.13166070975787	-1.20476067790126\\
-2.04417461501809	0.785996019664727\\
-2.04417461501809	-0.785996019664727\\
-2.00221735942775	0\\
-2.08842121490662	0.680930521699399\\
-2.08842121490662	-0.680930521699399\\
-2.25982507962394	1.12003877256647\\
-2.25982507962394	-1.12003877256647\\
-2.38475426992435	1.32799046034014\\
-2.38475426992435	-1.32799046034014\\
-2.17412534154196	0.194131654184283\\
-2.17412534154196	-0.194131654184283\\
-2.21664747631151	0\\
-2.28723868322552	0.491772163679451\\
-2.28723868322552	-0.491772163679451\\
-2.51133007329602	0.525068145369337\\
-2.51133007329602	-0.525068145369337\\
-2.48970352223899	0\\
-2.59828590894367	0\\
};
\addplot [color=black, draw=none, mark=o, mark options={solid, black}, forget plot]
  table[row sep=crcr]{%
-0.29220994533686	-7.30342228170809e-14\\
-0.512581198634323	2.46341983032941e-14\\
-0.617276558451237	-4.20319920230348e-12\\
-0.968375054322356	1.3448788520285e-08\\
-1.05417149158515	2.68867487422862\\
-1.0541714921566	-2.68867487412548\\
-1.1124877622406	-1.07571751384994e-06\\
};
\end{axis}

\end{tikzpicture}

%% file: sections/09_conclusion/conclusion.tex
\section{Conclusion.}
\label{sec:conclusion}

In this paper we proposed a rational QZ algorithm (RQZ) for the numerical solution 
of the dense (unsymmetric) generalized eigenvalue problem.
The new algorithm operates on matrix pairs in Hessenberg, Hessenberg form rather than the Hessenberg,
triangular form used in the classical QZ method. Hessenberg pairs link to
rational Krylov and the associated poles are encoded in the subdiagonal elements of both
Hessenberg matrices.
A direct reduction method of a  regular matrix pair to Hessenberg, Hessenberg form 
was proposed. Moreover, we have demonstrated that during the reduction a good 
choice of poles can lead to premature deflations. The iterative rational QZ algorithm
differs from the classical algorithm in the sense that also poles can be introduced in
each QZ step. Numerical experiments confirm that a good choice of poles allows the RQZ method
to outperform the QZ algorithm by reducing the number of iterations per eigenvalue.
The implicit chasing technique is justified by an implicit Q theorem, which is proved in a
novel manner operating directly on the matrix pair and exploiting the connections with
rational Krylov.
Our theoretical analysis revealed that an RQZ iteration implicitly performs nested subspace
iteration driven by a pair of rational functions.
Finally, we have applied the RQZ method as a filter in rational Krylov. 

\section*{Acknowledgments.}
The authors thank Jared Aurentz and Thomas Mach for the fruitful discussions and
suggestions related to this project, and the referees for their valuable feedback.
